\documentclass[11pt]{article} 
\usepackage[a4paper,left=1in, right=1in]{geometry} 
\usepackage{amsmath,amssymb,amsthm,amsfonts,dsfont,mathrsfs}
\usepackage{mathabx}
\usepackage{setspace, parskip}
\usepackage{bbm}
\usepackage{physics}
\usepackage{graphicx, fancyhdr}
\usepackage{algorithm}
\usepackage[noend]{algpseudocode}
\usepackage[margin=1cm,size=small]{caption}
\usepackage{enumitem}
\usepackage[margin=0.1cm]{subcaption}
\usepackage{footnote}
\usepackage[dvipsnames]{xcolor}
\usepackage{comment}
\usepackage{overpic}

\usepackage{thmtools}
\usepackage{thm-restate}

\usepackage{hyperref} 
\usepackage[nameinlink]{cleveref}
\let\cref = \Cref

\hypersetup{colorlinks=true,
            citecolor=ForestGreen,
            linkcolor=ForestGreen,    
            urlcolor=Brown}

\usepackage[backend=biber,style=numeric,maxbibnames=99]{biblatex}
\AtEveryBibitem{
    \clearlist{address}
    \clearfield{date}
    \clearfield{eprint}
    \clearfield{isbn}
    \clearfield{issn}
    \clearlist{location}
    \clearfield{month}
    \clearfield{series}
    \clearfield{doi}
    \clearfield{language}
    \ifentrytype{book}{}{
        \clearlist{publisher}
        \clearname{editor}
    }
}

\newtheorem{theorem}{Theorem}[section]
\newtheorem{lemma}[theorem]{Lemma}
\newtheorem{proposition}[theorem]{Proposition}

\newtheorem{example}[theorem]{Example}
\newtheorem{remark}[theorem]{Remark}
\newtheorem{definition}[theorem]{Definition}

 \let\gb=\beta  \let\gd=\delta

   \let\gL=\Lambda

\newcommand{\cO}{\mathcal{O}}

\newcommand{\bs}[1]{\boldsymbol{#1}}

\DeclareMathOperator{\E}{\mathds{E}}
\DeclareMathOperator{\R}{\mathbb{R}}

\DeclareMathOperator{\pr}{\mathds{P}}

\DeclareMathOperator*{\argmax}{argmax}
\DeclareMathOperator*{\argmin}{argmin}

\DeclareMathOperator{\MAT}{mat}

\let\cp\relax
\DeclareMathOperator*{\cp}{CP}

\renewcommand{\norm}[1]{\|#1\|}

\newcommand{\mat}[1]{\MAT_{(#1)}}

\renewcommand{\v}[1]{\bs{#1}} 

\renewcommand{\t}[1]{\bs{\mathcal{#1}}} 

\newcommand{\m}[1]{\bs{#1}}

\def\a{\v a}
\def\b{\v b}
\def\u{\v u}
\def\w{\v w}
\def\x{\v x}

\def\A{\m A}
\def\B{\m B}

\def\H{\m H}
\def\I{\m I}
\def\J{\m J}
\def\K{\m K}
\def\L{\m L}
\def\M{\m M}

\def\Q{\m Q}
\def\S{\m S}

\def\V{\m V}

\def\X{\m X}
\def\Y{\m Y}

\newcommand{\diag}[1]{\textbf{diag}\left(#1\right)}
\let\wh = \widehat
\let\wt = \widetilde

\renewcommand{\leq}{\leqslant} 
\renewcommand{\geq}{\geqslant} 

\let\it = \textit
\let\bf = \textbf

\def\eset{\varnothing}
\let\sset = \subseteq
\let\union = \cup

\newcommand{\set}[1]{\left\{#1\right\}}

\let\eps = \varepsilon

\let\l = \ell

\let\wt = \widetilde

\newcommand{\ind}[1]{\mathds{1}_{#1}}

\def\qed{ \hfill $\blacksquare$\par}

\DeclareMathOperator{\tucker}{Tucker}

\DeclareMathOperator{\sampleIndex}{\textsc{SampleIndex}}
\DeclareMathOperator{\col}{col}
\DeclareMathOperator{\row}{row}
\DeclareMathOperator{\nnz}{nnz}

\usepackage[export]{adjustbox}

\definecolor{comment}{rgb}{0.0, 0.40, 0.0}
\newcommand{\CommentIL}[1]{\hfill {\color{comment}// #1\ }}
\newcommand{\CommentNL}[1]{\State{\color{comment} // #1}}

\title{Fast and Accurate Interpolative Decompositions for \\ General, Sparse, and Structured Tensors}

\author{
    Yifan Zhang\thanks{Oden Institute, University of Texas at Austin}
    \and
    Mark Fornace\thanks{Lawrence Berkeley National Laboratory}
    \and
    Michael Lindsey\thanks{University of California, Berkeley and Lawrence Berkeley National Laboratory}}

\begin{document}

\maketitle

\begin{abstract}
    \noindent
    In this work, we develop deterministic and random sketching-based algorithms for two types of tensor interpolative decompositions (ID): the core interpolative decomposition (CoreID, also known as the structure-preserving HOSVD) and the satellite interpolative decomposition (SatID, also known as the HOID or CURT). We adopt a new adaptive approach that leads to ID error bounds independent of the size of the tensor. In addition to the adaptive approach, we use tools from random sketching to enable an efficient and provably accurate calculation of these decompositions. We also design algorithms specialized to tensors that are sparse or given as a sum of rank-one tensors, i.e., in the CP format. Besides theoretical analyses, numerical experiments on both synthetic and real-world data demonstrate the power of the proposed algorithms. 
\end{abstract}


\section{Introduction}

Low-rank matrix approximation is a powerful tool for a wide range of applications, e.g., data compression, noise reduction, and machine learning (see \cite{kishore2017literature,markovsky2012low} for a survey).
As opposed to the singular value decomposition (SVD), which directly optimizes the reconstruction quality given a target rank, methods based on interpolative decomposition (ID) use entries from the original matrix as its factors.
Compared to SVD, ID has many advantages.
First, using entries from the original matrix can make the decomposition more interpretable, leading to applications in experimental design, data inference, and model compression \cite{compton2012hybrid,pan2012fast,oseledets2010tt,tyrtyshnikov2000incomplete,mitrovic2013cur,thurau2012deterministic,cstefuanescu2013pod}.
Second, the decomposition is structure-preserving. 
When the input matrix $\M$ is sparse or has non-negative entries, so are the corresponding factors in the decomposition.
Finally, an ID is typically computed using a direct algorithm that is more efficient than an SVD, while not compromising too much in terms of approximation error (e.g., \cite{martinsson2011randomized}).

However, the benefits of interpolative decomposition are by no means limited to matrix data.
As higher-dimensional tensorial datasets become ubiquitous in data analysis, machine learning, numerical simulation, and many engineering applications, extensions of ID to tensors are increasingly attractive, resulting in a number of investigations over the past two decades \cite{drineas2007randomized,minster2020randomized,cai2021mode,saibaba2016hoid,song2019relative,kielstra2024linear}.
Among many variants, two particular types of ID have received particular attention, which we henceforth denote as the core interpolative decomposition (CoreID, also known as the structure-preserving HOSVD \cite{minster2020randomized}) and the satellite interpolative decomposition (SatID, also known as the higher-order ID (HOID) \cite{saibaba2016hoid} or CURT \cite{song2019relative}).
The overall goal of this paper is to develop scalable and provably accurate algorithms for computing the CoreID and SatID of general, sparse, and CP-formatted~\cite{kolda2009tensor} tensor data.   
 
We next give a more precise description of the CoreID and SatID. 
(For readers new to tensor decomposition and tensor networks, the surveys \cite{kolda2009tensor,orus2014practical} are useful references.
See also the notation gallery in \cref{sec:notation}.)
Given a tensor $\t T\in\R^{n_1\times\cdots\times n_d}$, the CoreID problem of target rank $\v k = (k_1,\ldots,k_d)$ seeks the approximation (\cref{fig:cidandsid}, left)
\begin{equation}
    \label{eq:cid_def}
    \text{(CoreID)}~~~
    \min_{J_1,\ldots,J_d, \m X_1,\ldots,\X_d}\|\t T - \tucker(\t T_{J_1,\ldots,J_d}, \X_1,\ldots,\X_d)\|,
\end{equation}
where $J_i \sset [n_i]$ are index subsets of size $k_i$, indicating the selected entries for each index of $\t T$, and $\X_i \in\R^{n_i\times k_i}$ are unconstrained matrices. 
The function $\tucker(\cdot)$ computes the contracted tensor given its Tucker factors (see \cref{sec:notation}). 
The name CoreID comes from the fact the core tensor in the Tucker decomposition is a subtensor of the original tensor.
 
On the other hand, the SatID problem of target rank $\v k$ finds index subsets $J_i \sset \bigoplus_{j\neq i}[n_j]$ 
 of size $k_i$, corresponding submatrices $\m T_i = (\mat{i,\cdot}\t T\,)_{:, J_i} \in\R^{n_i\times k_i}$, 
 and an unconstrained core $\t C\in\R^{k_1\times\cdots\times k_d}$ that solves (\cref{fig:cidandsid}, right)
\begin{equation}
    \label{eq:sid_def}
    \text{(SatID)}~~~
    \min_{J_1,\ldots,J_d, \t C}\|\t T - \tucker(\t C, \m T_1,\ldots,\m T_d)\|.
\end{equation}
Here $\mat{i, \cdot} \t T$ denotes the matrix flattening of $\t T$ with the $i$-th axis being the row index (see \cref{sec:notation}). 
Unlike CoreID, in SatID the satellite nodes of the Tucker decomposition are comprised of entries from the original tensor, while the core is unrestricted.

\begin{figure}[!h]
    \centering
    \begin{minipage}{.45\textwidth}
        \centering
        \includegraphics[width=0.4\linewidth, valign=c]{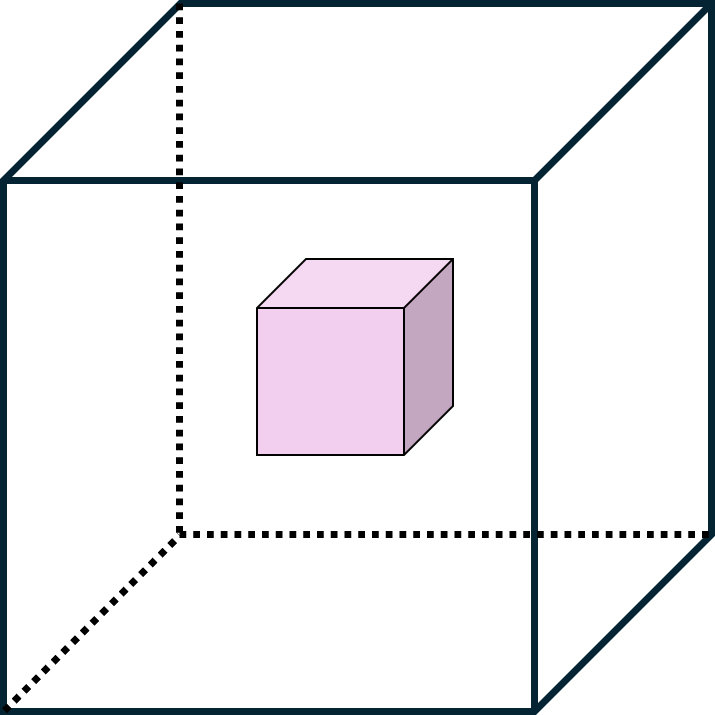}
        $~\approx~$
        \includegraphics[width=0.43\linewidth, valign=c]{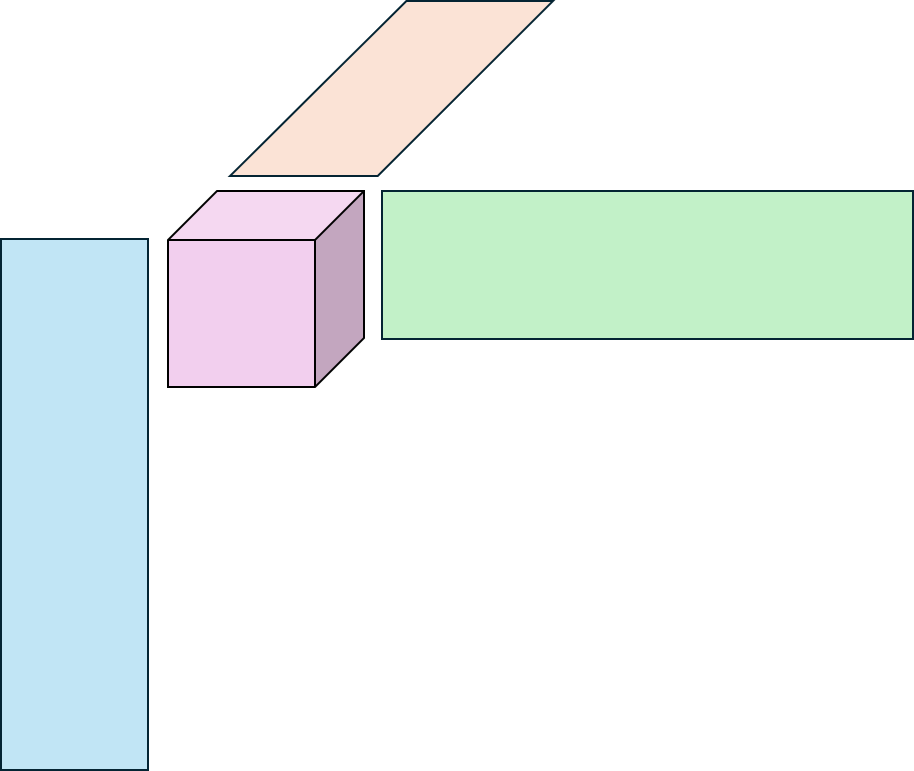}
    \end{minipage}     \hspace{0.03\textwidth}
    \begin{minipage}{.45\textwidth}
        \centering
        \includegraphics[width=0.4\linewidth, valign=c]{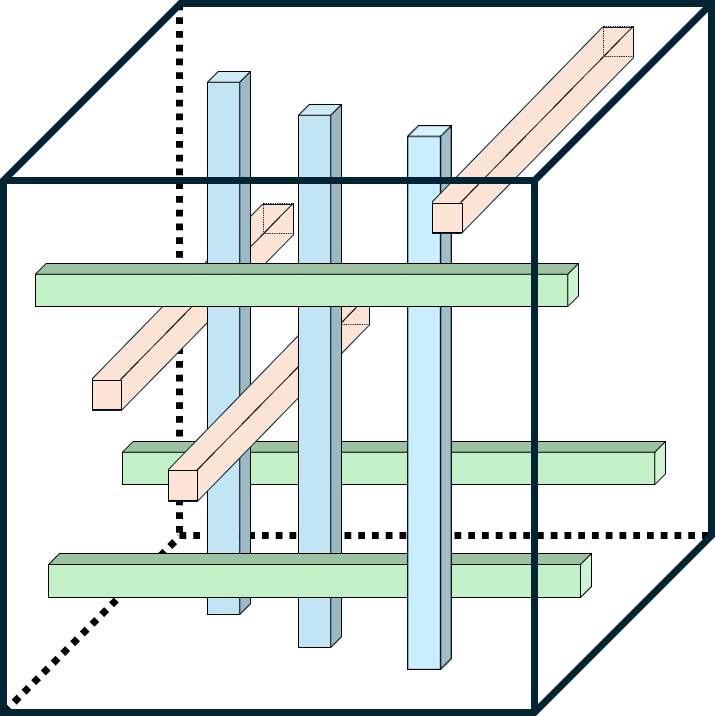}
        $~\approx~$
        \includegraphics[width=0.43\linewidth, valign=c]{illustrations/rhs.png}
    \end{minipage}

    \caption{
        Illustration of a 3rd-order CoreID (left) and a SatID (right).
        Both approximate the tensor on the left-hand side by a low rank Tucker format.
        In CoreID, the pink core (which need not be contiguous) is selected from the tensor; 3 satellite matrices or `nodes' are unrestricted and optimized to best approximate the tensor.
        In SatID, the blue, green, and orange satellite nodes are selected from the tensor (constructed by stacking the vectors selected from the tensor with corresponding colors); the pink core is unrestricted and optimized to best approximate the tensor.
    }
    \label{fig:cidandsid}
\end{figure}

While it is important to develop CoreID and SatID algorithms for a general tensor $\t T$, 
tensors in many applications are often sparse or structured.
For example, a moment tensor is often approximated as an empirical average of rank-one tensors. 
From an algorithm design perspective, specialized algorithms are needed to yield CoreID and SatID from these special tensor formats with optimal efficiency.
Some operations such as the Gram-Schmidt procedure can destroy sparsity or rank structure, and they should be avoided in the algorithm design.
For sparse tensors, extraction of a subtensor is economic, but for rank-structured tensors this can be expensive.
On the other hand, contraction between rank-structured tensors can be more efficient than contractions involving sparse tensor.
Therefore, one needs different CoreID and SatID algorithms for different types of tensors.

\subsection{Previous Works} \label{sec:previous}
The CoreID and SatID problems have been targeted by a number of previous works.
See e.g., \cite{minster2020randomized,saibaba2016hoid,drineas2007randomized,song2019relative}. 
The most common strategy is to flatten the tensor into different matrix views, and run matrix column selection algorithms to find index sets $J_1,\ldots,J_d$.
For simplicity, we refer to these matrix column selection subproblems as matrix ID problems.
The matrix column selection methods used in these works generally fall into the following categories. For a more detailed review please refer to the introduction of the recent work~\cite{fornace2024column}.

\textbf{Strong rank-revealing QR.}
First, \cite{minster2020randomized,saibaba2016hoid} developed algorithms for the CoreID and SatID problems (respectively) using strong rank-revealing QR (sRRQR) for the backbone matrix column selection.
The sRRQR algorithm was developed in \cite{gu1996efficient}, and it was introduced as a method for matrix interpolative decomposition in \cite{martinsson2011randomized}.
As computing the QR factorization of the matricized tensors are often too costly, this type of algorithms first apply a random sketch to the matricized tensor and then QR the sketched matrix.
One major issue of this strategy is that its error bound depends on the size of the tensor in a way that is especially problematic when applied to matricizations of higher-order tensors.
On an $n^d$ tensor, the ID error from the algorithm is only to bounded by $O(n^{d/2})$ times the optimal Tucker error of the same rank \cite{minster2020randomized}.

\textbf{Norm maximization, norm sampling, and uniform sampling.} Besides sRRQR, another type of matrix column selection approach is to select columns of a matrix $\A$ iteratively from a distribution with probability mass for the $i$-th column proportional to $\norm{\A_{:, i}}^\beta$. 
Once a column is selected, it is deflated from all columns of $\A$, similarly to the Gram-Schmidt procedure, and the next column is selected from the deflated matrix $\A$.
When $\beta \rightarrow +\infty$, this recovers the greedy column-pivoted QR decomposition.
This norm maximization approach is also related to the orthogonal matching pursuit framework \cite{tropp2007signal}. 
Taking $\beta = 2$ yields a randomized approach known as norm sampling or randomly pivoted QR \cite{drineas2007randomized,drineas2008relative,chen2023randomly,deshpande2006adaptive,deshpande2006matrix}, and taking $\beta = 0$ yields naive uniform sampling.
Using norm sampling for SatID of a general tensor was studied in \cite{drineas2007randomized}.
Compared to the sRRQR algorithm, the error bound from norm sampling is dimension-independent (i.e., independent of $n$).
However, the approach does not automatically adapt to exploit sparsity or rank structure in the input tensor, and a redesigned approach is needed for optimal scaling.

\textbf{Leverage score sampling.}
Another type of algorithm samples columns independently using leverage scores (see e.g., \cite{mahoney2011randomized}).
In \cite{song2019relative} the authors developed this approach for SatID on general tensors of order 3 (terming it the CURT decomposition). 
The authors found a novel way connecting theories of leverage scores sampling and SatID errors, and this allowed them to give the first SatID algorithm with an error of $1+\eps$ times the optimal CP (not Tucker!) error for any $\eps > 0$.
Specifically, according to the complete version of \cite[Theorem C.20]{song2019relative}, on a 3rd-order tensor the algorithm requires a SatID of rank $\cO(k^5/\eps^9)$ to guarantee an error of $(1 + \eps)^5$ relative to the optimal CP approximation error of rank $k$, assuming such an optimal CP approximation exists. 
This method can be extended to higher-order tensors at the cost of increasingly severe complexity in terms of $k$ and $\eps$.
For a tensor of order 4, \cite{song2019relative} suggests that a rank of $\cO(k^{17}/\eps^{25})$ is required to obtain an error of $(1 + \eps)^7$ relative to the optimal rank-$k$ CP. 

\textbf{Determinantal point process sampling.}
Determinantal point process (DPP) sampling (see, e.g., \cite{derezinski2021determinantal} for an overview) is an alternative approach which samples a subset $I$ of columns of matrix $\A$ all at once.
Writing $\K = \A^\top \A$,
the sampling probability is proportional to the determinant of the submatrix $\K_{I, I}$.
While the theoretical guarantees for DPP sampling are known to be nearly optimal (e.g., \cite{Guruswami2012-rr,derezinski2020improved,belabbas2009spectral}), algorithms for DPP sampling usually remain inefficient compared to alternative heuristic approaches such as greedy QR with column pivoting.

\textbf{Nuclear maximization.}
Two of the authors recently developed theory and efficient algorithms for column selection via nuclear maximization \cite{fornace2024column}, which can be viewed as greedily picking a column based on minimization of the posterior error.
Specifically, within the column selection iterations, for each remaining column $i$ the algorithm looks at the Frobenius norm of the matrix after column $i$ is deflated.
Then it selects column $i$ that minimizes the Frobenius norm.
In this sense, the basic algorithm is one-step optimal and matches isolated usage in prior literature (e.g., \cite{Altschuler2016-gd,ordozgoiti2018iterative,Farahat2011-zm}; see other references in \cite{fornace2024column}). 
In \cite{fornace2024column}, it is shown that the error of compares favorably to DPP sampling while also permitting scalable implementation. In particular, the matrix-free perspective of~\cite{fornace2024column} will be fundamental to some of our developments below.

\subsection{Our Contributions}

We develop scalable algorithms for both the CoreID and SatID problems for general, sparse, and CP-structured tensor data. 
Similar to previous attempts, we select index sets using matrix column selection on different views of the matricized tensor.
We use pivoted QR-based approaches (norm maximization, norm sampling, nuclear maximization) as our backbone matrix ID algorithm.
The reason for choosing QR-based approaches is that they lead to efficient algorithm design and dimension-independent error bounds relative to the Tucker error. Moreover, these methods can be adaptive and a prescription of the target rank is not necessary.
To alleviate the cost of QR deflation, efficient and specialized randomized sketching-based algorithms are designed based on the structure of the target tensor.
More specifically, we highlight our main contributions below:
\begin{itemize}
    \item \bf{Scalable algorithms for general, sparse, and CP-structured tensor data.} 
    Using sketching techniques from randomized numerical linear algebra, we give scalable algorithms for CoreID and SatID.
    Specialized algorithms are given for CoreID and SatID for sparse and CP tensors. 
    For selecting each index set $J_i$, $i = 1,\ldots,d$,
    our algorithms have a complexity linear in the number of nonzeros for sparse tensors, and linear in the number of rank-1 terms for CP tensors. In particular, for an efficient structure-exploiting implementation of SatID using any pivoted QR approach, we must avoid computing the full vector of column scores (whether norm-based or nuclear~\cite{fornace2024column}), which is generally of size $\cO(n^{d-1})$ where $n=\max_j \{n_j\}$.
    To this end, we develop a marginalized sampling method for SatID that preserves optimal complexity in both the sparse and rank-structured settings and moreover ensures that the algorithm requires essentially no extra storage.
    \item \bf{Improved dimension-independent error bounds.}
    Previously, to the best of our knowledge, tensor CoreID methods have exclusively used sRRQR \cite{minster2020randomized}, where the reconstruction error is only guaranteed to be $\cO\left(\prod_{j = 1}^d n_j^{1/2}\right)$
    times worse than the optimal Tucker approximation error of the same rank.
    We propose a new adaptive approach on CoreID that leads to dimension-independent error bounds.
    Specifically, the CoreID reconstruction error bound only depends on the errors of $d$ matrix ID problems on different matricized views of the tensor  (\cref{thm:err_cid_seq}).
    When random sketching is used, we also extend these new error bounds to the sketched version of the algorithm. 
    For SatID, \cref{prop:sid_err} states that the reconstruction error can be bounded by the sum of the matrix ID errors, which is again dimension-independent.
\end{itemize}

\subsection{Notation and Outline} \label{sec:notation}
We use bold, upper case calligraphic letters $\t T,\t A$ for tensors (usually of dimension at least 3), bold upper case letters $\m T, \m A$ for matrices, and bold lowercase letters $\v a, \v u$ for vectors.
For a tensor $\t T$ of order $d$ and subsets $I \sset [d]$, $\mat{I,\, [d]\setminus I}\t T$ gives the $\prod_{i\in I}n_i \times \prod_{j \in [d]\setminus I}n_j$ matrix flattening of $\t T$, where the rows are indexed by axes (modes) in $I$ and columns are indexed by axes in $[d]\setminus I$.
For convenience, $\mat{I,\, [d]\setminus I}$ is also written as $\mat{I, \cdot}$ or $\mat{\cdot, \,[d]\setminus I}$.
If $I = \set{i}$ or $[d]\setminus I = \set{i}$ is a singleton, we use the shortcut $\mat{i, \cdot} \t T$ or $\mat{\cdot, i}\t T$ instead. 
The operator $\times_j$ denotes the mode $j$ contraction (see \cite{kolda2009tensor} for details).
\begin{equation*}
    (\t T\times_j \x)_{i_1,\ldots,i_{j-1},i_{j+1},\ldots,i_d} 
    =
    \sum_{i_j} \t T_{i_1,\ldots,i_j,\ldots,i_d} \x_{i_j}.
\end{equation*}

A good review of tensor networks can be found in \cite{orus2014practical}.
We use $\cp(\A_1,\ldots,\A_d)$ to denote the contracted tensor for a CP network.
That is,
\begin{equation*}
    \cp(\A_1,\ldots,\A_d)_{i_1,\ldots,i_d} = \sum_k (\A_1)_{i_1, k} (\A_2)_{i_2, k}\cdots(\A_d)_{i_d, k}.
\end{equation*}
This can be viewed as a high order generalization of matrix rank decomposition.
Sometimes there is a mixing weight $\w$ in CP network, and we denote
\begin{equation*}
    \cp(\w, \A_1,\ldots,\A_d)_{i_1,\ldots,i_d} = \sum_k w_k \cdot (\A_1)_{i_1, k} (\A_2)_{i_2, k}\cdots(\A_d)_{i_d, k}.
\end{equation*}
These two definitions are compatible with each other if one views $\w$ as a matrix with one row.
The Tucker contraction of core $\t C$ and satellite nodes $\A_1,\ldots,\A_d$ is denoted by
\begin{equation*}
    \tucker(\t C, \A_1,\ldots,\A_d)_{i_1,\ldots,i_d}
    =
    \sum_{j_1,\ldots,j_d}\t C_{j_1,\ldots,j_d}(\A_1)_{i_1,j_1}\cdots(\A_d)_{i_d, j_d}.
\end{equation*}
This can be viewed as a higher-order SVD (indeed, the matrices $\A_i$ can be made into isometries). 

For a matrix $\A$, $\col(\A)$ and $\row(\A)$ denote the column and row spaces of $\A$ respectively.
For the purpose of this paper, unless otherwise specified, the norm $\|\cdot\|$ refers to the $\ell_2$ or Frobenius norm (for vector, matrix, and tensor inputs).
For a positive integer $r$, $\A^{(r)}$ denotes the best rank $r$ approximation to $\A$ in the Frobenius metric.
For a tensor $\t T$ and a multilinear rank $\v r$, $\t T^{(\v r)}$ denotes the best rank $\v r$ Tucker approximation to $\t T$.

The notation $\nnz(\cdot)$ denotes the number of nonzero entries in the input array. 
For two numbers $C_1$ and $C_2$, notation $C_1\asymp (1\pm\eps)C_2$ indicates that $(1-\eps)C_2 \leq C_1\leq(1 + \eps)C_2$, which is often convenient in sketching theory. 

For terminology, we refer to the matrix column selection problem as the matrix ID problem.
We use the name QR-based method to refer to three matrix column selection algorithms: norm maximization (greedy column-pivoted QR), norm sampling (randomly pivoted QR), and nuclear maximization.
These method select columns iteratively and deflate the selected column from the matrix, and hence the name.
See \cref{sec:previous} for more details.
We say that a CoreID or SatID error bound is dimension-independent to mean that it only depends on the matrix ID errors and in particular has no explicit dependence on the tensor shape $(n_1,\ldots,n_d)$, except via the number $d$.

The rest of the paper is organized as follows.
In \cref{sec:cid} and \cref{sec:sid} respectively, we 
give details to the algorithms and theoretical analyses for CoreID and SatID.
In \cref{sec:exp}, we illustrate the performance of our algorithms on synthetic and real-world datasets.
Many proofs and additional details are deferred to appendices.

\section{Core Interpolative Decomposition} \label{sec:cid}

In this section, we give full details of CoreID algorithms.
In \cref{sec:cid_from_matrix}, we first review existing methods that compute CoreID from sequential matrix column selection (which we simply refer to as matrix ID) on flattened tensors.
As the accuracy guarantee of the existing algorithm depends unfavorably on the size of the tensor, we then propose a new adaptive sequential approach to assemble the matrix ID results, leading to an improved dimension-independent error bound.
A few examples are provided to show how the bound can be used when using different matrix ID algorithms.
The corresponding baseline algorithm using this adaptive sequential approach for CoreID is \cref{alg:basic_cid}.
To make this algorithm more scalable, in \cref{sec:rand_cid}, we propose a sketching-based randomized algorithm (\cref{alg:sketch_cid}) on top of the adaptive sequential approach to accelerate the baseline algorithm.
Relevant theoretical guarantees are also established.
Finally in \cref{sec:cid_cp} and \cref{sec:cid_sparse} respectively, we specialize the sketching-based algorithm to CP tensors and sparse tensors.
The resulting algorithms are respectively \cref{alg:cid_cp} and \cref{alg:cid_sparse}.

Contrary to SatID, which is a natural generalization of CUR decomposition (e.g., \cite{mahoney2009cur})
of matrices to tensors, CoreID is a problem unique to tensors of order $d \geq 3$.
For matrices, the CoreID is computing a decomposition $\A = \m X \m C \m Y$ where $\m C$ is a $k_1\times k_2$ submatrix of $\A$. 
In this case, since we are allowed to choose $\m X$ and $\m Y$ freely, the optimal solution will match the SVD of $\A$ of rank $\min\set{k_1, k_2}$.
For $d \geq 3$, however, the optimal CoreID is not equivalent to the optimal Tucker decomposition.
The reason is that given a fixed CoreID core $\t C$ and the optimal Tucker decomposition $\t T \approx \tucker(\t C^*,\Y_1^*,\ldots,\Y_d^*)$, in general it is not possible to find satellites $\X_1,\ldots,\X_d$ such that $\tucker(\t C,\X_1,\ldots,\X_d) = \tucker(\t C^*,\Y_1^*,\ldots,\Y_d^*)$.

\subsection{The Adaptive Sequential Approach for CoreID} \label{sec:cid_from_matrix}

To motivate our algorithm, consider the approach computing a CoreID on tensor $\t T$ by flattening it into matrices $\A_i = \mat{\cdot, i}\t T$, $i = 1,\ldots,d$, selecting columns from each $\A_i$ to find index sets $J_i$, and then finding satellite nodes $\X_i$ by solving
\begin{equation*}
    \X_i = \argmin_{\X}\|\A_i - (\A_i)_{:, J_i} \X\|.
\end{equation*} 
We then simply deliver $\X_i$ and $J_i$ as our solution to CoreID.
We call this approach the \it{independent approach}, as calculating the pairs ($\X_i, J_i$) are independent over $i = 1,\ldots,d$.
Despite some potential benefits of the independent approach in computational efficiency and interpretability, it can lead to arbitrarily bad reconstruction error even if all matrix ID errors are small.
We give an explicit example to illustrate this in \cref{app:bad_indep}.

To avoid this issue, one can use a \it{sequential approach} to compute the CoreID, in which the algorithm uses the selected $(J_i,\X_i)$, $i = 1,\ldots,s-1$ to find the next tuple $(J_s,\X_s)$. 
In \cite{minster2020randomized}, the authors proposed the following sequential approach, which we refer to as the \it{simple sequential approach}.

\begin{enumerate}[label = (\roman*), itemsep = -1ex, partopsep = 1ex,parsep = 1ex]
    \item (\bf{Simple sequential approach}) Continue the ID process on the remaining core tensor.
    That is, find $J_2$ by running selection on $\mat{\cdot, 2} \t T_{J_1,\ldots}$ and set 
    \begin{equation*}
        \X_2=\argmin_{\X} \left\|\t T_{J_1,\ldots} - \t T_{J_1, J_2,\ldots}\times_2 \X\right\|.
    \end{equation*}
\end{enumerate}

\noindent
The main issue of (i) is that the final reconstruction error bound will depend on the operator norms $\|\X_1\|_2,~\|\X_2\|_2$, etc, because any reconstruction error made to $\t T_{J_1,\ldots}$ in later matrix ID solves will get magnified by $\X_1$.
For this reason, the algorithm in \cite{minster2020randomized} uses sRRQR as the matrix ID algorithm, as many other methods do not come with any (non-trivial) bound on $\|\X_i\|_2$.
Consequently, due to the shape-dependent bound on $\|\X_i\|_2$ from the theory of sRRQR, the reconstruction error bound given in \cite{minster2020randomized} is $\Omega(\sqrt{N})$ times larger relative to the optimal Tucker error, where $N$ is the number of elements in the tensor.

Here we propose using another sequential approach in CoreID, which we term \it{adaptive sequential approach}. 
\begin{enumerate}[label = (\roman*), itemsep = -1ex, partopsep = 1ex,parsep = 1ex, start = 2]
    \item (\bf{Adaptive sequential approach}) Continue the ID process on the reconstructed tensor. That is, find $J_2$ by running selection on $\mat{\cdot, 2}(\t T_{J_1,\ldots}\times_1 \X_1)$ and set 
    \begin{equation*}
        \X_2=\argmin\nolimits_{\X}\left\|\t T_{J_1,\ldots}\times_1 \X_1 - \t T_{J_1,J_2,\ldots}\times_1 \X_1 \times_2\X \right\|.
    \end{equation*}
\end{enumerate}

\noindent
We will see below that with the adaptive sequential approach (ii), we can get a dimension-independent error bound on the reconstruction error (see \cref{sec:notation}) for many popular matrix ID algorithms, i.e., the error only depends on the matrix ID errors and is independent of the tensor size.

To make the notation simple, we introduce the default notation for \cref{sec:cid} that $\t T_i$ denotes the reconstruction tensor
\begin{equation*}
    \t T_i = \t T_{J_1,\ldots,J_{i-1},\ldots}\times_1\X_1\times_2\cdots\times_{i-1}\X_{i-1}
\end{equation*}
from which we select the $i$th index set $J_i$, and $\A_i = \mat{\cdot, i}\t T_i$ denotes the corresponding matrix.

Before the analysis, first we outline the baseline algorithm using the adaptive sequential approach below in \cref{alg:basic_cid} and analyze its complexity.
The only difference compared to approach (ii) above is that in \cref{alg:basic_cid} we merge the $R$ factor of the QR decomposition of $\X_i$ into the remaining core to make the computation more efficient.

\begin{algorithm}
    \textbf{Input:} tensor $\t T$ of order $d$, target rank $(k_1,\ldots,k_d)$, matrix ID algorithm $\textsc{MatrixID}$.\\
    \textbf{output:} Index sets $J_1,\ldots,J_d$, satellite nodes $\X_1,\ldots,\X_d$.

    \vspace*{1ex}
    \begin{algorithmic}[1]
        \caption{Baseline CoreID for a general tensor}
            \label{alg:basic_cid}
            \Function{BasicCoreID}{}
            \For{$i = 1,\ldots,d$}
                \State $\m A_i \gets \mat{\cdot, 1}\t T$ if $i=1$ else $\mat{\cdot, i}(\t T\times_1\L_1\times_2\cdots\times_{i-1} \L_{i-1})$
                \State Solve matrix problem $J_i, \X_i \gets \textsc{MatrixID}(\m A_i, k_i)$
                \State Compute a row-wise QR decomposition for short-wide matrix, $\X_i = \L_i \Q_i$
            \EndFor
        \EndFunction
        \\
        \Return Index sets $J_1,\ldots,J_d$, satellite nodes $\X_1,\ldots,\X_d$.
    \end{algorithmic}
\end{algorithm}

To analyze the complexity of the basic algorithm, we assume for simplicity that $k_1 = \ldots = k_d = k$, and the shape of $\t T$ is $n\times\ldots\times n$.
The dominant contribution to the cost is computing the matrix ID, which depends on the algorithm used.
For norm maximization and norm sampling, the complexity is generically $\cO(k n^d)$.
The cost is the same for nuclear maximization with randomized score estimation assuming modest sketching size. 
Alternatively, non-randomized nuclear maximization yields $\cO(k n^2 + \theta)$ cost, where $\theta$ is the cost of computing $\A_i^T \A_i$, which is $\cO(n^{d+1})$ in a general dense setting. 

In the rest of this section, we show that the adaptive sequential approach indeed leads to dimension-independent error bounds.
Recall from \cref{sec:notation} that $\A^{(r)}$ and $\t T^{(\v r)}$ are respectively the best rank $r$ matrix approximation to $\A$ and the best rank $\v r$ Tucker approximation to tensor $\t T$.
We begin by introducing the following notion.
\begin{definition}\label{def:phiacc}
    A matrix ID algorithm is said to be $\Phi$-accurate if it has the following property.
    For any matrix $\A$ and any reference rank $r$,
    if $k\geq r$ columns of $\A$ are selected by the algorithm, resulting in an index set $J$ and satellite node $\X$,
    then for all such $k$ it holds that
    \begin{equation}\label{eq:phi_bnd}
        \|\A - \A_{:, J}\X\| \leq \Phi(k, r, \|\A - \A^{(r)}\|).
    \end{equation}
    If the matrix ID algorithm is randomized, we say the algorithm is $\Phi$-accurate in expectation if 
    \begin{equation}\label{eq:phi_bnd_exp}
        \E\|\A - \A_{:, J}\X\| \leq \Phi(k, r, \|\A - \A^{(r)}\|).
    \end{equation}
    We say the algorithm is relatively $\Phi$-accurate
    if instead of \eqref{eq:phi_bnd} we have
    \begin{equation}
        \frac{\|\A - \A_{:, J}\X\|}{\|\A\|} \leq \Phi\left(k, r, \frac{\|\A - \A^{(r)}\|}{\|\A\|}\right).
    \end{equation}
    We naturally require that $\Phi$ is non-decreasing in the first and last argument.
\end{definition}
Now we state the error bound for the adaptive sequential approach.

\begin{restatable}{theorem}{ciderrorbnd}
    \label{thm:err_cid_seq}
    Consider running \cref{alg:basic_cid} with target rank $\v k$, $k_i\geq r_i$.
    Denote $\t T^{[\v k]}$ the reconstructed tensor.
    Suppose that the matrix ID algorithm is $\Phi$-accurate, and denote
    for brevity $\Phi_i(\cdot) = \Phi(k_i, r_i, \cdot)$.
    Suppose without loss of generality that the processing order of the modes is $1\rightarrow 2\rightarrow\cdots\rightarrow d$.
    Then
    \begin{equation}\label{eq:cid_err_phi}
        \|\t T - \t T^{[\v k]}\| \leq \Phi_d \circ (I + \Phi_{d-1}) \circ\cdots\circ (I + \Phi_1)(\|\t T - \t T^{(\v r)}\|).
    \end{equation}
    Here $I$ is the identity map.
    In particular, if $\Phi_i(x) = C_i x$ for all $i$, then
    \begin{equation}\label{eq:cid_err_c}
        \|\t T - \t T^{[\v k]}\| \leq C_d \cdot \left(\prod_{i = 1}^{d-1} (1 + C_i)\right)\|\t T - \t T^{(\v r)}\|.
    \end{equation}
    If the matrix ID algorithm is relatively $\Phi$-accurate,
    then bounds \eqref{eq:cid_err_phi} and \eqref{eq:cid_err_c} hold with absolute errors replaced with relative errors.
    If the matrix ID algorithm is $\Phi$-accurate in expectation and $\Phi$ is concave in the last argument,
    then bounds \eqref{eq:cid_err_phi} and \eqref{eq:cid_err_c} hold with the left-hand sides replaced with the expected error.
\end{restatable}

\begin{proof}
    See \cref{app:cid_err_bnd}.
\end{proof}

\begin{remark}
    The bound in \cref{thm:err_cid_seq} displays an exponential dependence on $d$.
    This is not as good as the HOSVD error, which is proved to be $\sqrt{d}$-suboptimal \cite{vannieuwenhoven2012new}.
    The missing property preventing such a result here is the lack of orthogonality in the satellite nodes.
    Nonetheless, for fixed $d$ the bound is independent of the shape $(n_1, \ldots, n_d)$  when the matrix dependence of $\Phi$ is dimension-independent.
    This is a substantial improvement over the simple sequential approach (i), cf. \cite{minster2020randomized}.
    Improved reconstruction accuracy of the adaptive sequential approach is also demonstrated through numerical examples in \cref{sec:exp}.
\end{remark}

We provide three worked examples to illustrate how \cref{thm:err_cid_seq} can be used together with error bounds for matrix ID algorithms.
In \cref{exp:err_cid_dpp}, we calculate the expected CoreID error when determinantal point process sampling is used to perform matrix ID.
In \cref{exp:err_cid_nuc}, we calculate the CoreID error in relative terms when nuclear maximization is used to perform matrix ID, for which a relative error guarantee is available.
Finally in \cref{exp:err_cid_rpchol}, we consider using the norm sampling algorithm as the matrix ID solver and show how \cref{thm:err_cid_seq} can be used when the matrix ID error bound is stated as an $(r,\eps)$ bound (e.g. \cite{derezinski2021determinantal,chen2023randomly}).
In all examples, the resulting bounds are dimension-independent.

\begin{example} \label{exp:err_cid_dpp}
    (CoreID error with DPP)
    Following the notations in \cref{thm:err_cid_seq}, let $\t T^{[\v k]}$ denote the output of \cref{alg:basic_cid} using determinantal point process sampling. 
    From \cite{derezinski2021determinantal}, we know $k$-DPP is $\Phi$-accurate in expectation with 
    \begin{equation*}
        \Phi(k, r, v) = \sqrt{1 + \frac{r}{k - r + 1}} \  v,
    \end{equation*}
    which is trivially concave in $v$.
    Thus, if in particular we take $k_i = m r_i$ then \cref{thm:err_cid_seq} gives
    \begin{equation*}
        \E \| \t T - \t T^{[\v k]} \| \leq \sqrt{1 + \frac{1}{m - 1}} \  \left(1 + \sqrt{1 + \frac{1}{m - 1}}\right)^{d-1} \  \| \t T - \t T^{(\v r)} \|,
    \end{equation*}
    which is again dimension-independent.
    As $m \rightarrow \infty$, the above bound goes to $2^{d-1} \|\t T - \t T^{(\v r)}\|$.
\end{example}

\begin{example} \label{exp:err_cid_nuc}
    (CoreID error with nuclear maximization)
    Using the same notations as in the example above, define
    \begin{equation*}
        \Phi(k, r, v) = \sqrt{ \min_{s \in \set{r+1,r+2,\ldots,k}} e^{-k/s} + (1-e^{-k/s})\left(1 + \frac{r}{s-r+1}\right)v^2},
    \end{equation*}
    and $\Phi_i(v) = \Phi(k_i, r_i, v)$.
    Since nuclear maximization is relatively $\Phi$-accurate~\cite{fornace2024column}, if we take in particular $k_i = m_i r_i$ with $m_i \geq 2$ and instead specify $s = 2r_i$ in each $\Phi_i$, then \cref{thm:err_cid_seq} gives (relaxing the bound by removing the square root for simplicity)
    \begin{equation*}
        \Phi_i(v) \leq e^{-m/2} + 2v,
    \end{equation*}
    and hence recursion computation gives
    \begin{equation*}
        (1 + \Phi_{d-1})\circ\cdots\circ(1 + \Phi_1)(v) \leq
        3^{d-1}v + \frac{1}{2}(3^{d-1} - 1)e^{-m/2}.
    \end{equation*}
    Thus,
    \begin{equation*}
        \frac{\|\t T - \t T^{[\v k]}\|}{\|\t T\|} 
        \leq
        2\cdot 3^{d-1}\cdot \frac{\|\t T - \t T^{(\v r)}\|}{\|\t T\|} + 3^{d-1}e^{-m / 2},
    \end{equation*}
    which is dimension-independent.
    The optimal choice of $s$ that minimizes $\Phi_i$ depends on the relative Tucker error, but we do not consider optimizing over $s$ here for simplicity.
\end{example}

\begin{example} \label{exp:err_cid_rpchol}
    (CoreID error with norm sampling)
    Using the same notations as in the example above.
    If instead $\t T^{[\v k]}$ is the output of \cref{alg:basic_cid} using norm sampling,
    then from \cite{chen2023randomly}, for norm sampling of rank $k$ applied to any matrix $\A$
    \begin{equation*}
        \E\|\A - \A_{:, J}\X\|^2 \leq (1 + \eps) \cdot \|\A - \A^{(\v r)}\|^2
    \end{equation*}
    provided that
    \begin{equation*}
        k \geq \frac{r}{\eps} + r\log\left(\frac{1}{\eps\eta}\right),
    \end{equation*}
    where $\eta = \|\A - \A^{(r)}\|^2 / \|\A\|^2$.
    Since this bounds $\E\|\A - \A_{:, J}\X\|$ linearly in $\|\A - \A^{(r)}\|$, the concavity assumption on $\Phi$ is satisfied.
    If we choose $k_i$ such that for each $i = 1,\ldots,d$
    \begin{equation*}
        k_i \geq \frac{r_i}{\eps} + r_i\log\left(\frac{1}{\eps\eta_i}\right),
    \end{equation*}
    where $\eta_i = \|\A_i - \A_i^{(r_i)}\|^2 / \|\A_i\|^2$,  
    then \cref{thm:err_cid_seq} gives 
    \begin{equation*}
        \E\|\t T - \t T^{[\v k]}\| \leq (1 + \eps)^{1/2} (2 + \eps)^{(d-1)/2} \cdot \|\t T - \t T^{(\v r)}\|.
    \end{equation*}
    This is again dimension-independent.
\end{example}

The above examples demonstrate how to translate existing matrix ID error guarantees into tensor ID error guarantees using \cref{thm:err_cid_seq}.
In particular, by using the proposed adaptive sequential approach, if the matrix ID error is dimension-independent, then the tensor CoreID error is dimension-independent as well.

\subsection{Accelerating \cref{alg:basic_cid} with Random Sketching}\label{sec:rand_cid}

As pointed out in \cref{sec:cid_from_matrix}, running \cref{alg:basic_cid} requires $\Omega(k n^d)$ complexity per mode, which can be prohibitive for large tensor sizes, and the deflation process within the matrix ID can destroy structure in the tensor (such as being low-rank or sparse).
To make \cref{alg:basic_cid} more scalable, we compress matrices $\A_i\in\R^{n^{d-1}\times n}$ 
to much smaller matrices $\B_i \in \R^{m\times n}$ using a random sketching matrix $\S_i\in\R^{m\times n^{d-1}}$.
This not only makes the subsequent column selection much faster but also enables exploitation of the tensor structure via the design of the random sketch $\S_i$.
At the end of this section, we provide analyses for how sketching affects the CoreID error bound when nuclear maximization or norm sampling is used as the matrix ID algorithm.

Our proposed approach works for all QR-based matrix ID algorithms, such as norm maximization, norm sampling, and nuclear maximization.
Without loss of generality, let us consider the selection of index set $J_1$.
For QR-based matrix ID, we select columns into $J_1$ one by one.
Let $\A = \mat{\cdot, 1}\t T$ be the flattened matrix.
Suppose we have selected column indices $I$ into $J_1$, and we are selecting the next column.
Let $\Q_I$ be an orthonormal basis of the space spanned by the selected columns $\A_{:, I}$.

For norm sampling and norm maximization, we need to compute the norms
\begin{equation}\label{eq:qr_score}
    d^{(I)}_i := \|\A_{:, i} - \Q_I\Q_I^\top\A_{:, i}\|^2 = \min_{\x}\|\A_{:, I} \x - \A_{:, i}\|^2,~~i \in [n].
\end{equation}
Then norm sampling selects the next column by sampling from a probability density proportional to $\v d^{(I)} = (d_i^{(I)})$, while norm maximization selects $\argmax_i d_i^{(I)}$.

For nuclear maximization, the algorithm looks at the reduction of posterior reconstruction errors
\begin{equation}\label{eq:nuc_score}
    d_i^{(I)} := \|\A - \Q_{I}\Q_{I}^\top \A\|^2 - \|\A - \Q_{I\union\set{i}}\Q_{I\union\set{i}}^\top \A\|^2
    =
    \min_{\X}\|\A_{:, (I\union\set{i})^c}\X - \Q_{I^c}\Q_{I^c}^\top\A\|^2,~~i \in [n],
\end{equation}
and augments $I$ with $\argmax_i d_i^{(I)}$.

We see that in all algorithms, the scores $\v d^{(I)} = (d_i^{(I)})$ are computed by
solving least squares problems involving columns of $\A$.
This motivates the use of a sketch matrix to compute the scores.
To derive the randomized algorithm, we introduce the following notion from random sketching theory (see, e.g., \cite{woodruff2014sketching}). 

\begin{definition}
    [$\eps$-SE]
    \label{def:se}
    Given $\eps > 0$ and a subspace $V$, a matrix $\S$ is said to be an $\eps$-subspace embedding ($\eps$-SE) of $V$ if for all $\x \in V$, we have $\|\S\x\|^2 \asymp (1 \pm \eps) \|\x\|^2$ (see \cref{sec:notation} for the definition of $\asymp$). 
    We say $\S$ is an $\eps$-SE of a matrix $\A$ if $\S$ is an $\eps$-SE of the subspace $\col(\A)$.
\end{definition}

There are several standard techniques for constructing an $\eps$-SE.
Let $\A \in\R^{N\times n}$ be the matrix from which we intend to select columns.
A simple choice of sketch matrix $\S \in\R^{m\times N}$ is a random Gaussian matrix. 
Then with high probability $\S$ will be an $\eps$-SE if $m \sim \eps^{-2} \rank(\A)$ (see, e.g., \cite{woodruff2014sketching}). 
In addition, when $\A$ is close to a matrix of rank $r \ll n$, it is often sufficient to take $m \sim \eps^{-2} r$ instead of $\eps^{-2} \rank(\A)$ as explained in \cite{cohen2015optimal}. 
Alternatively, one can sketch with the fast JL transform (FJLT) operator \cite{ailon2006approximate} or the count sketch \cite{pagh2013compressed} operator.
Later, we will also use tensor-structured sketches $\S$ in order to more efficiently compute IDs for sparse and CP tensors.
For different designs of $\S$, the bounds on the sketch dimension $m$ may vary.
Nonetheless, $m$ can often be taken independent of $N$.
When $\A$ is constructed as the matricization of a tensor, usually $N\gg n$, and thus the compressed matrix $\B = \S\A$ is much smaller than $\A$.
We refer the readers to \cite{woodruff2014sketching,martinsson2020randomized} for further discussions.

The theorems \cref{thm:pert_rpchol} and \cref{thm:pert_nuc} stated later in this section show that we can run QR-based matrix ID algorithms (nuclear maximization, norm maximization, norm sampling, see \cref{sec:notation}) on the smaller matrix $\B$ instead of $\A$.
As suggested by these theorems, we summarize the whole procedure of sketched QR-based matrix ID algorithm in \cref{alg:sketch_cid} below.
This algorithm can be used to replace Line 4 in \cref{alg:basic_cid}.
In subsequent sections, we will tailor the design of $\S$ to sparse and CP tensors.

\begin{algorithm}
    \textbf{Input:} matrix $\A$, target rank $k$, matrix ID solver $\textsc{MatrixID}$ using QR-based methods. \\
    \textbf{output:} index set $J$, satellite node $\X$.
    
    \vspace*{1ex}
    \begin{algorithmic}[1]
        \caption{Sketched QR-based matrix ID}
        \label{alg:sketch_cid}
            \Function{SketchedMatrixID}{}
            \State Generate an $\eps$-SE sketch $\S$ of $\A$
            \State Compute $\B = \S\A$
            \State $J, \X \gets \textsc{MatrixID}(\B, k)$
            \EndFunction
        \\
        \Return index set $J$ and node $\X$.
    \end{algorithmic}
\end{algorithm}

Note that in the entire selection process of $J_1$, we only need to generate a single $\S$, and we pre-compute $\B = \S\A$ before running any selection. 
Then we run the selection on a much smaller matrix $\B$ to obtain $J_1$. 
This is crucial.
For structured and sparse tensors, it is not an option to 
generate matrices $\S$ anew each time the set $I$ gets augmented and only require $\S$ to preserve the norms of columns $\A_{:, i} - \Q_I\Q_I^\top\A_{:, i}$, $i = 1,\ldots,n$.
This is because the deflation operation $\A_{:, i} - \Q_I\Q_I^\top\A_{:, i}$ will destroy sparsity and tensor structure, making the algorithm prohibitively slow. 
In addition, the pre-application of $\S$ avoids generating and applying the sketch operator to the large tensor each time a column is selected.

To analyze the complexity of \cref{alg:sketch_cid}, suppose that the sketch dimension is $m$ and the target rank is $k$. Then after the sketch step, the complexity of the remaining steps is the same as the complexity of the $\textsc{MatrixID}$ algorithm applied to a matrix of size $m \times n$.
In theory, the sketch dimension $m$ depends on the design of $\S$ as well as other parameters such as $n$, $k$, and the (stable) rank of $\A$ \cite{cohen2015optimal}.
In practice, $m$ is often treated as a hyperparameter, not necessarily following the theoretical bounds.

How will the sketch affect the reconstruction error? 
To ease the discussion, we introduce the following notations.
Fix the matrix ID rank $k$.
Let $(J, \X^*)$ be the matrix ID output on the sketched matrix $\S\A$.
Given a reference rank $r \leq k$, we seek $C \geq 0$ such that
\begin{equation}
\label{eq:generalbound}
    \|\A_{:, J}\X^* - \A\|^2 \leq C \|\A - \A^{(r)}\|^2.
\end{equation}
We present two strategies to deduce such a bound, which will be applied respectively in the cases where norm sampling and nuclear maximization are applied for the $\textsc{MatrixID}$ step.

The first strategy for establishing \eqref{eq:generalbound} uses the following chain of inequalities:
\begin{equation}\label{eq:rpchol_approach}
    \|\A_{:, J}\X^* - \A\|^2 \xrightarrow{\leq} \min_{\X}\|\A_{:, J}\X - \A\|^2 \xrightarrow{\leq} \|\A - \A^{(r)}\|^2,
\end{equation}
where in each step (indicated $\xrightarrow{\leq}$) an inequality must be shown to hold up to a suitable constant factor.
In \eqref{eq:rpchol_approach}, the first step will follow provided that our sketch defines a SE of $\A$, as we shall explain at greater length in the proof of \cref{thm:pert_rpchol} below. 
Specifically, the bound for this step is achieved in \cref{lem:step1}, which shows the preconstant to be $(1-\delta)^{-1}$, under the assumption that $\S$ is a $\delta$-SE for $\A$. 
This result may be of some independent interest, as a more straightforward argument yields only the preconstant $(1+\delta)/(1-\delta)$.
The second step is harder, and it requires an analysis of how a perturbation of the scores defined $d_i^{(I)}$ affects the matrix ID error. In the case of norm sampling, where the scores are defined by \eqref{eq:qr_score}, we are able to control the effect of such a perturbation, as we shall show in the proof of Theorem \ref{thm:pert_rpchol} below.

The second strategy for establishing \eqref{eq:generalbound} is more general and applies in particular to the case of nuclear selection. (The reason for including the previous strategy is that the analysis of the perturbation of the norm sampling scores will be used later in \cref{thm:pert_cpsid}.)  The structure of the second strategy is given by the following chain:
\begin{equation}\label{eq:nuc_approach}
    \|\A_{:, J}\X^* - \A\|^2 \xrightarrow{\leq} \min_{\X}\|\S\A_{:, J}\X - \S\A\|^2 \xrightarrow{\leq} \|\S\A - (\S\A)^{(r)}\|^2 \xrightarrow{\leq} \|\A - \A^{(r)}\|^2.
\end{equation} 
The first and the last step are based on the SE assumption for the sketch (via arguments going beyond direct application of the definition of SE, as outlined in the proof of \cref{thm:pert_nuc} below), and the middle step uses the available matrix ID error bounds without sketching.

Based on these proof strategies, we prove the following error bounds for sketched norm sampling (using the first strategy \eqref{eq:rpchol_approach}) and sketched nuclear maximization (using the second strategy \eqref{eq:nuc_approach}).

\begin{restatable}{theorem}{rpqrError}\label{thm:pert_rpchol}
    Fix any $\eps, \gd \in (0, 1)$ and a target rank $k$.
    Suppose that in \cref{alg:sketch_cid} the norm sampling method is used for $\textsc{MatrixID}$, and $\m S$ is a $\delta$-SE of $\A$.
    Let $(J, \X^*)$ be the output of \cref{alg:sketch_cid}.
    Then for a reference rank $r$,
    \begin{equation}\label{eq:pert_rpchol1}
        \E\min_{\X}\|\A - \A_{:, J}\X\|^2 \leq (1 + \eps) \|\A - \A^{(r)}\|^2
    \end{equation}
    provided that 
    \begin{equation*}
        k \geq \frac{1+\gd}{1-\gd}\cdot \left(\frac{r}{\eps} + r\log\left(\frac{1}{\eps\eta}\right)\right),
    \end{equation*}
    where $\eta = \|\A - \A^{(r)}\|^2 / \|\A\|^2$. 
    Consequently, the matrix ID error satisfies 
    \begin{equation}\label{eq:pert_rpchol2}
        \E\|\A - \A_{:, J}\X^*\|^2 \leq (1 + \eps)(1 - \delta)^{-1} \|\A - \A^{(r)}\|^2.
    \end{equation}
\end{restatable}

\begin{proof}
    See \cref{app:cid_rpqr}. 
\end{proof}

A key component of the proof of \cref{thm:pert_rpchol} is the following lemma on the error bound of norm sampling when the scores are perturbed. This may be of independent interest, and in particular we use it again in the proof of \cref{thm:pert_cpsid}.

\begin{restatable}{lemma}{approxRpqr}\label{lem:approx_rpchol}
    Suppose norm sampling is used to compute the matrix ID of $\A$.
    Fix any $\eps\in (0, 1)$ and a target rank $k$.
    Using exact norm sampling,
    denote by $\pr_I(i)$ the probability that column $i$ is sampled given that the subset of columns $I$ is already selected.
    Suppose we instead sample columns from an approximate distribution $\wt{\pr}_I(i)$, such that for some $\beta \in (0, 1)$, $\wt{\pr}_I(i) \geq \beta\pr_I(i)$ for all $I, i$.
    Let $J$ be the set of $k$ column indices selected by following $\wt{\pr}$.
    Then
    \begin{equation} \label{eq:pert_rpchol11}
        \E\min_{\X}\|\A - \A_{:, J}\X\|^2 \leq (1 + \eps) \|\A - \A^{(r)}\|^2
    \end{equation}
    provided that 
    \begin{equation*}
        k \geq \beta^{-1}\cdot \left(\frac{r}{\eps} + r\log\left(\frac{1}{\eps\eta}\right)\right),
    \end{equation*}
    where $\eta = \|\A - \A^{(r)}\|^2 / \|\A\|^2$.
\end{restatable}

\begin{proof}
    See \cref{app:cid_rpqr}.
\end{proof}

The next result is for the matrix ID error bound using a sketch matrix in nuclear maximization.

\begin{restatable}{theorem}{nucError}\label{thm:pert_nuc}
    Fix any $\gd \in (0, 1)$ and a target rank $k$.
    Suppose that in \cref{alg:sketch_cid} the nuclear maximization method is used for $\textsc{MatrixID}$, and $\m S$ is a $\delta$-SE of $\A$.
    Let $(J, \X^*)$ be the output of \cref{alg:sketch_cid}.
    Then for a reference $r$, in the $\eps \rightarrow 0$ limit
    \begin{equation}
        \min_{\X}\|\A - \A_{:, J}\X\|^2 \leq (1 + \eps) \|\A - \A^{(r)}\|^2
    \end{equation}
    provided that
    \begin{equation*}
        k\geq C \cdot \frac{1+\gd}{1-\gd} \left(\frac{r}{\varepsilon }+r-1\right) \left(\log \left(\frac{1+\delta}{1-\delta} \ \nu^{-1}\right)+\log \left(\varepsilon^{-1}-r^{-1}+1\right)\right),
    \end{equation*}
    where $C$ is a universal constant and $\nu = \|\A - \A^{(r)}\|^2 / \|\A^{(r)}\|^2$.
    Consequently, the matrix ID error satisfies:
    \begin{equation}
        \|\A - \A_{:, J}\X^*\|^2 \leq (1 + \eps)\frac{1 + \delta}{1-\delta} \|\A - \A^{(r)}\|^2.
    \end{equation}
\end{restatable}

\begin{proof}
See \cref{app:cid_nuc}.
\end{proof}

In the theorems above, if we do not apply any sketch and set $\delta = 0$, we recover the corresponding guarantees in the literature \cite{chen2023randomly, fornace2024column} for the exact version of the algorithms.
With a sketch that is a $\delta$-SE, the matrix ID error bounds of sketched nuclear maximization and sketched norm sampling are $(1 + \cO(\delta))$ times worse than the error without a sketch, provided that the number of selected columns $k$ is only $\frac{1 + \delta}{1-\delta}$ times larger than the number used without sketching. 
This justifies the use of a sketch.

\subsection{CoreID for CP Tensors}\label{sec:cid_cp}

In order to sketch a tensor in the CP format, we use
the Kronecker FJLT (KFJLT) \cite{malik2020guarantees,jin2021faster} or tensor sketch \cite{pagh2013compressed,pham2013fast} as our sketch matrix $\S$ as they can be applied to CP tensors efficiently. 
Using \cref{alg:sketch_cid} we can efficiently select $J_1$ and compute $\X_1$.

The only remaining issue is how to merge the $\L$ matrices in \cref{alg:basic_cid} and selection from mode $s$ when $s > 1$. 
Suppose the given tensor is $\cp(\m H_1,\ldots,\m H_d)$ and that we have selected index sets $J_1,\ldots,J_s$ with corresponding satellite nodes $\m X_1,\ldots,\m X_s$.
Let $\X_i = \L_i\Q_i$ be the row-wise QR.
Using \cref{alg:basic_cid}, in the next step we need to select columns from
\begin{equation*}
    \A_{s+1} = \mat{\cdot, s+1}(\t T_{J_1,J_2,\ldots,J_s,\ldots} \times_1 \L_1\times_2\cdots\times_s \L_s).
\end{equation*} 
Fortunately, this is simply another CP tensor:
\begin{equation*}
    \A_{s+1} = \mat{\cdot, s+1}\cp(\L_1^\top(\H_1)_{J_1,:}, \ldots,\L_s^\top (\H_s)_{J_s,:}, \H_{s+1},\ldots,\H_d),
\end{equation*}
meaning that we can still efficiently sketch the columns of this flattened matrix.
Our full algorithm for CP CoreID is outlined in \cref{alg:cid_cp}. 

\begin{algorithm}
    
    \textbf{Input:} CP factors $\m H_1,\ldots,\m H_d$, target rank $\v k = (k_1,\ldots,k_d)$, matrix ID solver $\textsc{MatrixID}$ using QR-based methods.\\
    \textbf{output:} index sets $J_1,\ldots,J_d$, satellite nodes $\X_1,\ldots,\X_d$.
    
    \vspace*{1ex}
    \begin{algorithmic}[1]
        \caption{CoreID for CP tensors}\label{alg:cid_cp}
        \Function{CP-CoreID}{}
        \State Denote $\t T := \cp(\m H_1,\ldots,\m H_d)$
        \For{$i = 1,\ldots,d$}
        \State Generate a KFJLT or tensor sketch operator $\S$ that is a $\delta$-SE of $\mat{\cdot,i}\t T$
        \State $\B\gets\S\cdot\mat{\cdot,i}\t T$
        \State $J_i, \m X_i \gets \textsc{MatrixID}(\B, k_i)$
        \State Compute row-wise QR, $\X_i = \L_i \Q_i$
        \State $\m H_i \gets \m L_i^\top (\m H_i)_{J_i, :}$
        \EndFor
        \EndFunction
        \\
        \Return index sets $J_1,\ldots,J_d$ and nodes $\X_1,\ldots,\X_d$.
    \end{algorithmic}
\end{algorithm}
\noindent
To analyze the complexity, we assume all $\m H_i$ are of shape $n\times p$, $p \geq n$ (otherwise one can reduce to $p = n$ by passing to a subspace), $k_i = k$, and take $d$ to be an $\cO(1)$ constant.
Let $m$ be the sketch dimension.
Both the KFJLT and tensor sketch run in $\wt{\cO}((m + n)p)$ time.
In total $d$ sketches are needed, and after sketching, the complexity is equal to that of $\textsc{MatrixID}$, which typically runs in $\cO(kmn)$ time.
After selection on each mode, we compute $\L$ and update the $\m H$ matrix in $\cO((n + p) k^2)$ time.
Hence, the total time complexity of \cref{alg:cid_cp} is $\cO(pk^2 + kmn) + \wt{\cO}((n + m)p)$.

\subsection{CoreID for Sparse Tensors}\label{sec:cid_sparse}

Suppose the selection order over the modes is (without loss of generality) $1\rightarrow2\rightarrow\cdots\rightarrow d$.
For selection over the first mode on a sparse tensor $\t T$, we use a count sketch operator \cite{clarkson2017low} {as an $\eps$-SE} to sketch the columns of $\A = \mat{\cdot, 1} \t T$. 
However, the situation becomes more complicated after column selection for the first mode.
Indeed, suppose the selected index sets and satellite nodes are respectively $J_i$ and $\X_i$, $i \leq s$.
Then when selecting from mode $s+1$, we need to sketch the columns of the following flattening
\begin{equation*}
    \A_{s+1} = \mat{\cdot,s+1}\left(
        \t T_{J_1,\ldots,J_s,\ldots}\times_{1,\ldots,s}(\X_1\otimes\ldots\otimes\X_s) 
    \right).
\end{equation*}
So $\A_{s+1}$ has a ``tensor product times sparse matrix" structure.
To sketch such a matrix efficiently, we use a composition of KFJLT, count sketch, and FJLT.

Specifically, we define the sketch network $G_{s+1}$ in tensor diagram notation shown in \cref{fig:sparse_cid}.
\begin{figure}[h!]
    \centering
    $G_{s+1}(\S_1,\S_2,\S_3,\t T,\X_1,\ldots,\X_s) \,=\,$
    \includegraphics[width = 0.45\textwidth, valign = c]{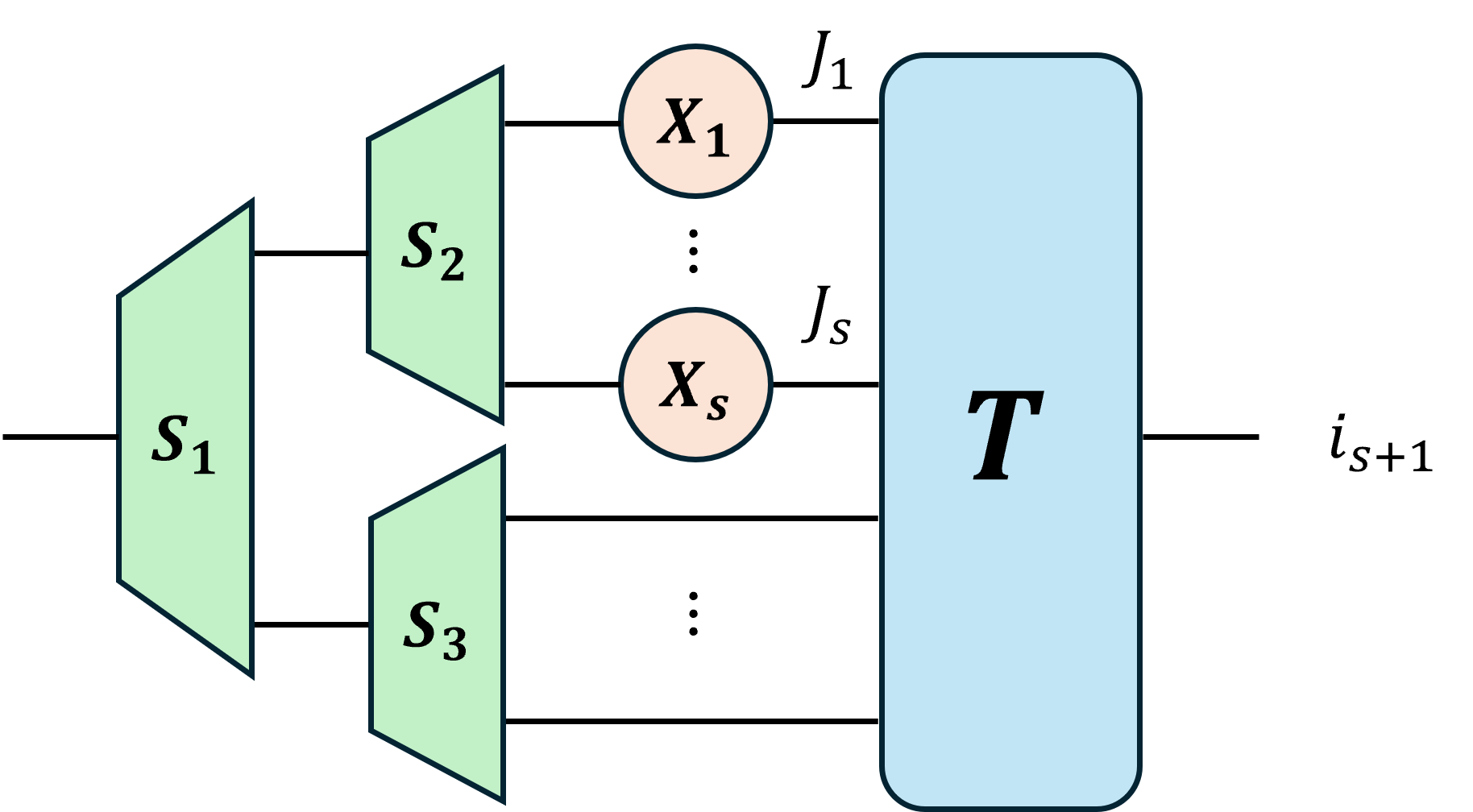}.
    \caption{Tensor diagram for the sketch operation on sparse tensor $\t T$. A connecting edge means contraction. See \cite{orus2014practical} for more introductory details on tensor diagrams. 
    In the diagram, $\S_2$ is a KFJLT or tensor sketch operator, $\S_3$ is a count sketch operator, and $\S_1$ is a general unstructured sketch operator (e.g. FJLT or Gaussian). See main text for details.
    }
    \label{fig:sparse_cid}
\end{figure}
For simplicity let us assume $\t T$ has a shape of $n^d$ and the target rank is $k_1 = \ldots = k_d = k$.
We choose $\S_1 \in\R^{m_1\times m_2m_3}$ to be an FJLT operator, $\S_2\in\R^{m_2\times k^s}$ a KFJLT operator, and $\S_3\in\R^{m_3\times n^{d-s-1}}$ a count sketch operator.
The sketch $\S_1$ is optional in that it can be omitted if $m_2m_3$ is sufficiently small.
When $s = 0$ (selecting over the first mode), KFJLT $\S_2$ is not used, and when $s = d-1$ (selecting over the last mode), count sketch $\S_3$ is not used.
According to \cite{woodruff2014sketching,ahle2020oblivious,haselby2023modewise}, fixing $\eps$ as a constant, given that the matrix of flattened tensor $\mat{\cdot, i}\t T$ has rank $r \leq k$ (which holds in particular if $\t T$ has a Tucker rank $r$) 
the composed sketch operator $\S_1(\S_2\otimes\S_3)$ is an SE of $\A_{s+1}$ with high probability if $m_1 = \wt\cO(r)$, $m_2 = \wt\cO(r^2)$, and $m_3 = \cO(r^2)$.
In practice, if the matricized tensor is close to being rank $r$, these sketching dimensions also work well as explained in \cite{cohen2015optimal}.

To contract $G_{s+1}$, we first compute $\S_3 \t T$, which takes $\cO(\nnz(\t T))$ time.
Then since $\S_2(\X_1\otimes\cdots\otimes \X_s)$ contains $m_2$ rank-1 tensors as rows, it takes $\cO\left(m_2\min\{s\cdot\nnz(\t T), k^d\}\right)$ time to contract it with $\S_3\t T$,
assuming the Fourier transform on each $n\times k$ matrix $\X_i$ is negligible in comparison.
Finally, the optional FJLT operator $\S_1$ can be applied in $\wt\cO(m_2m_3n)$ time.
So the total complexity is 
$\wt\cO(m_2m_3 n) + \cO(m_2 \min\set{k\nnz(\t T),~k^d})$.
We outline this algorithm in \cref{alg:cid_sparse}.

\begin{algorithm}

    \textbf{Input:} Sparse tensor $\t T$, target rank $\v r = (r_1,\ldots,r_d)$, matrix ID solver $\textsc{MatrixID}$ using norm-based sampling.\\
    \textbf{output:} index sets $J_1,\ldots,J_d$, satellite nodes $\X_1,\ldots,\X_d$.

    \vspace*{1ex}
    \begin{algorithmic}[1]
        \caption{CoreID for sparse tensors}
        \label{alg:cid_sparse}
            \Function{Sparse-CoreID}{}
                \For{$i = 1,\ldots,d$}
                    \State Generate sketch operators $\S_1$, $\S_2$, $\S_3$ as needed
                    \State Contract $\B\gets G_{i}(\S_1,\S_2,\S_3,\t T,\L_1,\ldots,\L_{i-1})$ as described above
                    \State $J_i, \m X_i \gets \textsc{MatrixID}(\B, r_i)$
                    \State Compute row-wise QR, $\X_i = \L_i \Q_i$
                \EndFor
            \EndFunction
        \\
        \Return index sets $J_1,\ldots,J_d$ and nodes $\X_1,\ldots,\X_d$.
    \end{algorithmic}
\end{algorithm}

\section{Satellite Interpolative Decomposition} \label{sec:sid}

In the last section we looked at the CoreID problem, in which we approximate the tensor with a Tucker decomposition where the core tensor is required to be a subtensor of the original tensor.
As pointed out before, CoreID is a problem unique to tensors of order $d \geq 3$.
As opposed to CoreID, SatID is a direct higher order generalization of matrix CUR decomposition.
Specifically, we seek for a Tucker approximation to $\t T$ in which the satellite nodes $\m T_i$ are required to be a subset of $k_i$ columns of the matricizations of $\t T$.
That is,
\begin{equation*}
    \m T_i = \left(\mat{i, \cdot}\t T\right)_{:, J_i}\in \R^{n_i \times k_i}
\end{equation*}
Let $\t C \in \R^{k_1\times\ldots\times k_d}$ be an unrestricted core tensor. In SatID we seek to solve
\begin{equation*}
    \min_{\t C, J_1,\ldots,J_d} \|\t T - \tucker(\t C,\m T_1,\ldots,\m T_d)\|.
\end{equation*}

Recall that for the CoreID problem, despite the benefits of the independent approach in computation and interpretation, we resorted to an adaptive sequential approach to control the error. 
However, the sequential approach is not necessary for SatID.
It is not hard to show that for SatID the reconstruction error using the independent approach can be well-controlled using the matrix ID errors.
We state this result below in \cref{prop:sid_err}.

\begin{restatable}{proposition}{satIdError}
    \label{prop:sid_err}
    For each $i$, let $J_i$ be set of selection.
    Denote the error in the matrix column selection on mode $i$ as
    \begin{equation*}
        \eps_i = \min_{\X}\|\mat{i, \cdot}\t T - (\mat{i,\cdot}\t T)_{:, J_i}\cdot \X\|.
    \end{equation*} 
    Then for these index sets $J_1,\ldots,J_d$, the error in SatID is bounded by
    \begin{equation*}
        \min_{\t C} \|\t T - \tucker(\t C,\m T_1,\ldots,\m T_d)\| 
        \leq
        \sum_{i=1}^d \eps_i,
    \end{equation*}
    where $\m T_i = \left(\mat{i, \cdot}\t T\right)_{:, J_i}\in \R^{n_i \times r_i}$.
    If the selection algorithm is random and
    \begin{equation*}
        \eps_i = \E \min_{\X}\|\mat{i, \cdot}\t T - (\mat{i,\cdot}\t T)_{:, J_i}\cdot \X\|.
    \end{equation*}
    Then similarly
    \begin{equation*}
        \E \min_{\t C} \|\t T - \tucker(\t C,\m T_1,\ldots,\m T_d)\| \leq \sum_{i=1}^d \eps_i.
    \end{equation*}
\end{restatable} 

\begin{proof}
    See \cref{app:sid_errbnd}.
\end{proof}

Consequently, we adopt the simple independent approach to compute SatID, with the baseline algorithm outlined in \cref{alg:sid}.

\begin{algorithm}

    \textbf{Input:} tensor $\t T$, target rank $\v k = (k_1,\ldots,k_d)$, matrix ID algorithm $\textsc{MatrixID}$.\\
    \textbf{output:} core node $\t C$, satellite nodes $\m T_1,\ldots,\m T_d$.

    \vspace*{1ex}
    \begin{algorithmic}[1]
        \caption{Skeleton for tensor SatID algorithm}
            \label{alg:sid}
            \Function{BasicSatID}{}
            \For{$i = 1,\ldots,d$}
                \CommentIL{select on the $i$th mode}
                \State $J_i, \m T_i \gets \textsc{\textsc{MatrixID}}(\mat{i, \cdot}\t T, k_i)$ 
            \EndFor
            \State Solve the core $\t C \gets\argmin_{\t X}\|\t T - \tucker(\t X,\m T_1,\ldots,\m T_d)\|$
            \EndFunction
        \\
        \Return $\t C$, $\m T_1,\ldots,\m T_d$.
    \end{algorithmic}
\end{algorithm}

Compared to CoreID, controlling the error in SatID is straightforward, but the computation of SatID can seem less tractable: on an $n^d$ tensor, if a QR-based algorithm is used when selecting $J_1$, one needs to compute $n^{d-1}$ scores for the $n^{d-1}$ columns in $\mat{1,\cdot}\t T$.
This is far from optimal if the tensor is structured and in particular may require more storage than the structured representation of the tensor itself!

To solve this problem, we propose using norm sampling in matrix ID with an efficient marginalization trick to sample columns efficiently according to the $n^{d-1}$ scores, without forming them all explicitly.
In \cref{sec:margin}, we introduce this marginalization trick for general tensors.
In \cref{sec:sid_cp}, we detail how this marginalization trick and random sketching can be used to compute SatID for CP tensors.
The resulting approach is summarized in \cref{alg:sid_cp}.
Lastly in \cref{sec:sid_sparse}, we introduce the algorithm for computing SatID for sparse tensors.
For sparse tensors, there are only $\cO(\nnz (\t T))$ nonzero scores to compute, which means it is often acceptable to directly compute and store the scores. 
Thus, in \cref{sec:direct} we give an efficient direct algorithm (\cref{alg:sid_sparse_direct}) where the scores are directly computed.
In \cref{sec:nnzalg}, we show how the marginalization trick plus sketching can improve the complexity of \cref{alg:sid_sparse_direct} for large and high-order tensors.
Specifically, the complexity of this algorithm is $\cO(\nnz(\t T) + m n^2 k^2)$ for an $n^d$ tensor with target rank $k$ (\cref{prop:sparse_sid_complexity}), where $d$ is regarded as an $\cO(1)$ constant and $m$ is the sketch dimension.

\subsection{Marginalized Norm Sampling for SatID} \label{sec:margin}

To make the discussion clearer, we assume the tensor has a size of $n^d$.
Notice that for index set $J_i$, a selected column index $\v b = (b_1,\ldots,b_{i-1},b_{i+1},\ldots,b_d)$ is a tuple of length $d-1$, with $b_k \in [n]$.
The space from which $\b$ is sampled is rather large, of size $n^{d-1}$.
To accelerate this sampling, we sample the entries of $\b$ one by one (or `autoregressively') from suitable conditional marginal  distributions. 
In this way we avoid computing the entire joint distribution.

Since the independent approach is used, we consider without loss of generality the selection of $J_d$, the subset for the $d$-th (final) mode.
Suppose we have selected a set of column indices $I$ into $J_d$, and we are selecting the next column.
For simplicity, we let $\v i$ denote the multi-index $(i_1,\ldots,i_{d-1})$ for columns in $\A := \mat{d, \cdot}\t T$.
Recall that to sample columns in norm sampling, we need to compute
\begin{equation*}
    d^{(I)}_{\v i} := \|\A_{:, \v i} - \Q_I\Q_I^\top\A_{:, \v i}\|^2 = \min_{\x}\|\A_{:, I} \x - \A_{:, \v i}\|^2,
\end{equation*}
where, as before, $\Q_I$ is an orthonormal basis of the space spanned by selected columns $\A_{:, I}$.
Let $X = (X_1,\ldots,X_{d-1})$ be the random variable denoting the index of the next selected column of $\A$. 
Then norm sampling requires us to sample $X$ from the law 
\begin{equation*}
    \pr(X = (i_1\ldots i_{d-1})) ~\propto~ d^{(I)}_{i_1\ldots i_{d-1}}.
\end{equation*}
The basic idea is to first sample $i_1$ from its marginal distribution, and then sample $i_2$ given $i_1$ and continue until $X$ is fully sampled.
To this end note that 
\begin{equation*}
    \pr(X_1 = i_1) ~\propto~ \sum_{i_2,\ldots,i_{d-1}} d^{(I)}_{i_1\ldots i_{d-1}}. 
\end{equation*}
To figure out the sum, we compute $\Q_I \in\R^{n\times |I|}$ explicitly and derive
\begin{align}
    \pr(X_1 = i_1) 
    &~\propto~
    \sum_{i_2,\ldots,i_{d-1}} \sum_{a,b}\A_{i_1i_2\ldots i_{d-1}, a}(\m I - \Q_I\Q_I^\top)_{a,b}\A_{i_1i_2\ldots i_{d-1}, b}\nonumber\\
    &=
    \Bigl((\mat{\cdot, 1}\t T)^\top \mat{\cdot, 1}\t T\Bigr)_{i_1,i_1}
    -
    \Bigl((\mat{\cdot, 1}(\t T \times_d \m Q_I))^\top \mat{\cdot, 1}(\t T\times_d \m Q_I)\Bigr)_{i_1,i_1}\nonumber\\
    &=
    \|(\mat{\cdot, 1}\t T)_{:, i_1}\|_2^2 - \|(\mat{\cdot, 1}(\t T\times_d \m Q_I))_{:, i_1}\|^2_2\nonumber \\
    &= 
    \|(\mat{\cdot, 1}(\t T\times_d \m Q_{I^{c}}))_{:, i_1}\|^2_2 \label{eq:margin_norms}
\end{align}
\cref{fig:marg} gives a tensor diagram illustration of this marginalization process.
Thus, to sample $i_1$, it suffices to compute the column norms of $\mat{\cdot, 1}(\t T\times_d \m Q_{I^c})$.
These norms can be efficiently updated as $\Q_I$ gets augmented every time a new column is added to $J_d$.
On top of everything, sketching methods can be applied to approximate the norms.
Once $i_1$ is sampled, we can repeat this process for $i_2$ on the sliced tensor $\t T_{i_1,\ldots}$ until we finish sampling $X$.
In this way, all $n^{d-1}$ scores need not be computed or considered directly.

\begin{figure}
    \centering
    \begin{minipage}[b]{.18\textwidth}
        \centering
        \includegraphics[height=12ex]{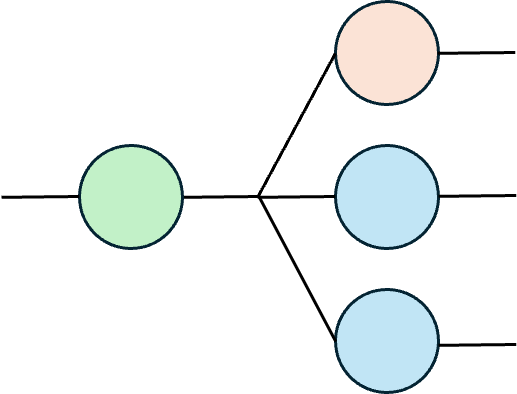}\\
        (a)
    \end{minipage}
    \hspace{0.03\textwidth}
    \begin{minipage}[b]{.24\textwidth}
        \centering
        \includegraphics[height=12ex]{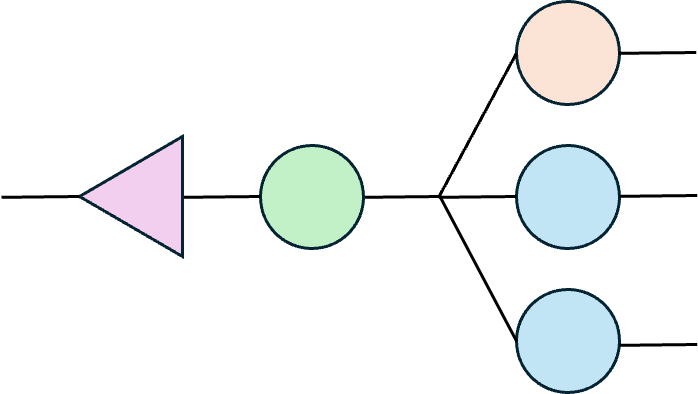}\\
        (b)
    \end{minipage}%
    \hspace{0.03\textwidth}
    \begin{minipage}[b]{.48\textwidth}
        \centering
        \includegraphics[height=12ex]{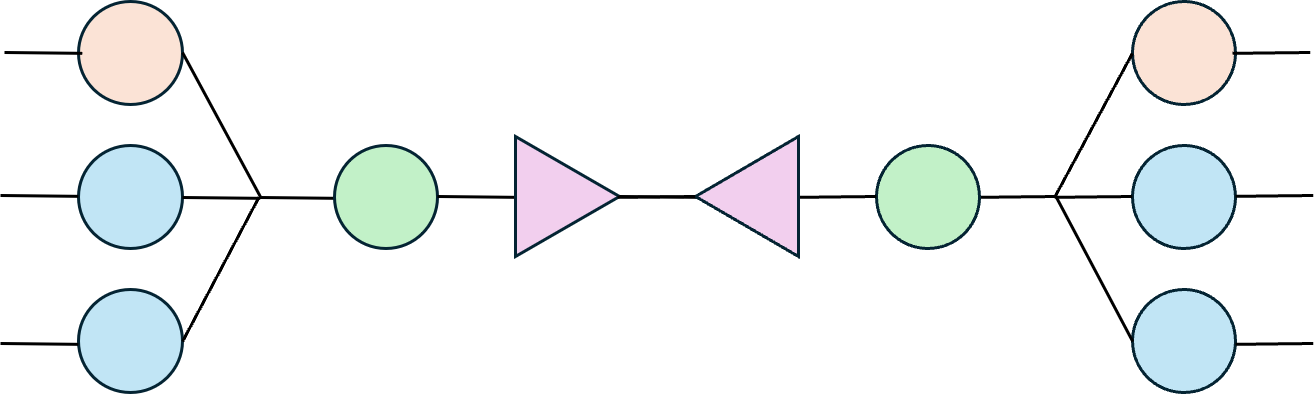}\\
        (c)
    \end{minipage}\\
    \vspace{2ex}
    \centering
    \begin{minipage}[b]{.45\textwidth}
        \centering
        \includegraphics[height=14ex]{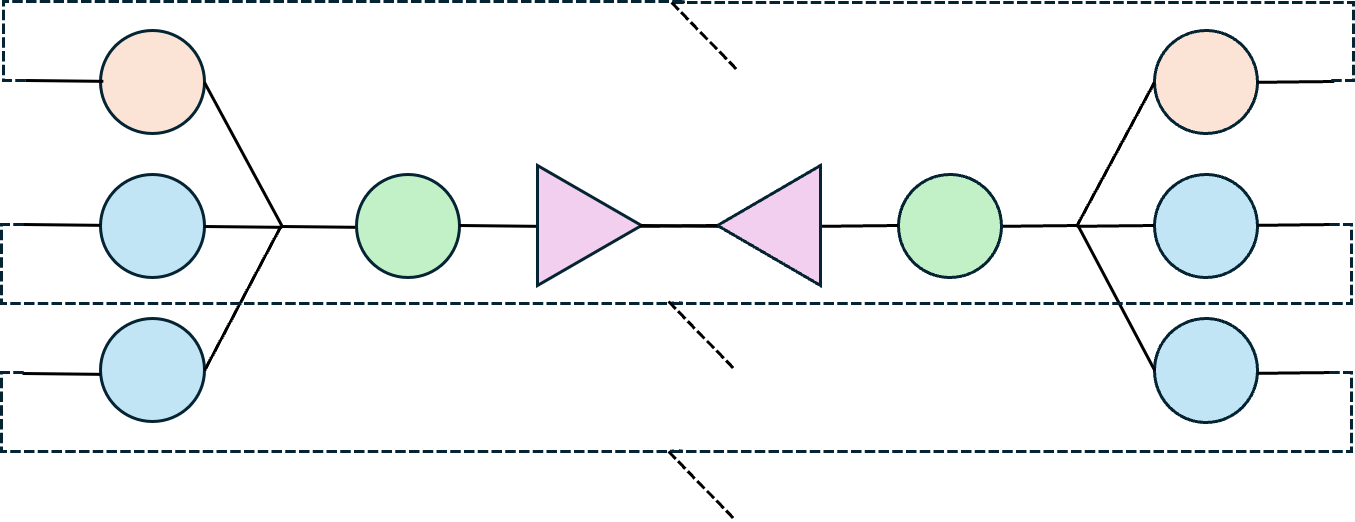}\\
        (d)
    \end{minipage}
    \begin{minipage}[b]{.45\textwidth}
        \centering
        \vspace{-2ex}
        \includegraphics[height=12ex]{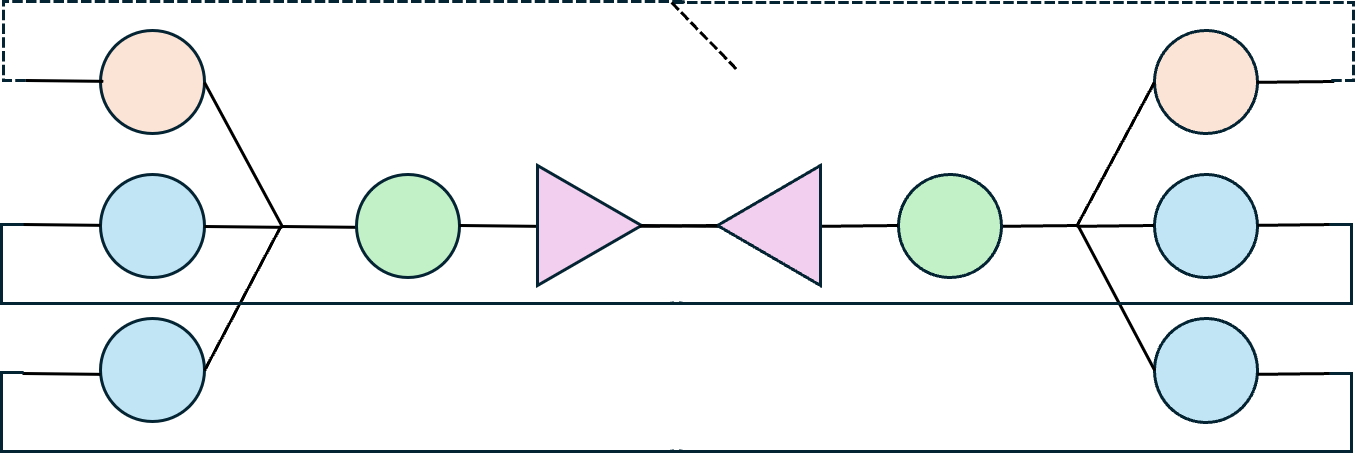}\\
        \vspace{2ex}
        (e)
    \end{minipage}\\
    \vspace{2ex}
    \centering
    \begin{minipage}[b]{.36\textwidth}
        \centering
        \includegraphics[height=12ex]{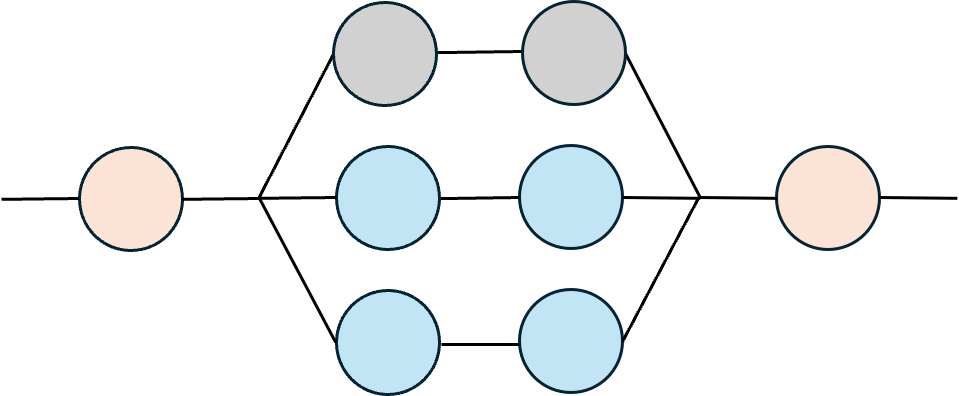}\\
        (f)
    \end{minipage}
    \begin{minipage}[b]{.48\textwidth}
        \centering
        \includegraphics[height=12ex]{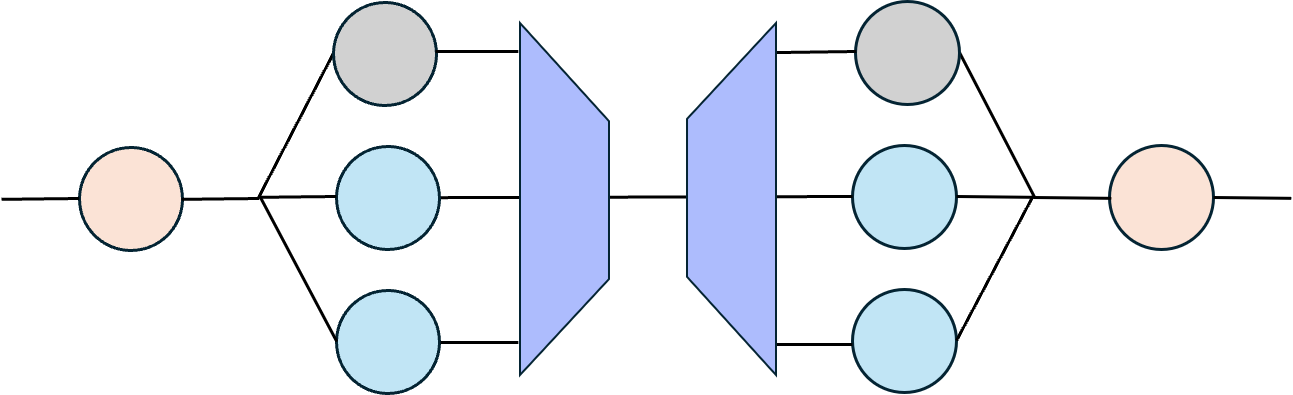}\\
        (g)
    \end{minipage}
    \caption{
        Illustration of the marginalization trick in tensor diagram.
        (a) The tensor being decomposed in this example is given as a CP tensor of order 4.
        The factors are the 4 circle nodes.
        (b) We are performing selection for the green mode, adding to $I$ a new column index $\b$, which is a triplet indexing the other 3 modes (starting with the orange mode). 
        The diagram represents the matrix $\Q_{I^c}^\top \A$.
        The pink triangle represents the $\Q_{I^c}$ matrix.
        (c) To compute the column norms, we need the diagonal of $\K = (\Q_{I^c}^\top\A)^\top (\Q_{I^c}^\top\A)$. $\K$ is depicted in this diagram.
        (d) The column norms (i.e, the diagonal of $\K$) are represented by the tensor in this diagram.
        (e) There are $n^3$ scores, and we want to avoid forming them all explicitly. The marginalization trick samples the 3 indices in $\b$ one by one (or `autoregressively') by computing the conditional marginal distribution for each index. We start with the orange index, whose marginal distribution is given by summing over the outgoing legs from the blue nodes in (d).
        (f) The contraction of the green and the pink nodes gives the grey node, which is efficiently updated every time $I$ gets augmented. The marginal distribution of the orange index in simply the diagonal of the matrix in (f). 
        (g) This can be computed efficiently by inserting a `stochastic resolution of the identity' $\S^\top \S$ (where $\S$ is a KFJLT sketch matrix, corresponding to the darker blue nodes) along the vertical midline of the diagram, to avoid the $\cO(p^2)$ cost where $p$ is the number of rank-1 tensors in the CP tensor. 
        Once the orange index in $\b$ is sampled, we repeat this process on 2 other modes (where in each stage the order of the tensor is reduced by 1).
    }
    \label{fig:marg}
\end{figure}

\subsection{SatID for CP Tensors} \label{sec:sid_cp}

Consider $\t T = \cp(\H_1,\ldots,\H_d)$, where $\H_i \in \R^{n\times p}$ are given explicitly but $\t T$ itself is not explicitly available.
We want to convert $\t T$ into SatID format without ever incurring a computation or storage cost of order $\Omega(\min(p^2, n^{d-1}))$.
To achieve this, we combine the marginalization approach and matrix sketching to efficiently carry out the norm sampling algorithm.
Specifically, in each column selection iteration, $I$ is given and we find an sketch $\S$ and approximate the norms in \eqref{eq:margin_norms} for finding $\pr(X_1 = i_1)$.
Then after sampling $i_1$ we repeat this for $i_2$ on $\t T_{i_1,\ldots}$ until a new column index is sampled, and augment $\Q_I$.
Function \textsc{SampleRow} in \cref{alg:sid_cp} below summarizes the proposed approach (with many steps done in an implicit but efficient manner).
We use the KFJLT operator $\S$ for efficient application to the CP tensor.
\cref{alg:sid_cp} can then be used as the \textsc{MatrixID} subroutine (Line 3) in \cref{alg:sid} for CP tensors. A graphical depiction of the steps of \cref{alg:sid_cp} is provided in Figure~\ref{fig:marg}.

\begin{algorithm}[!h]
    \textbf{Input:} CP factors $\H_1,\ldots,\H_d$, target rank $k$.\\
    \textbf{output:} Sampled index set $J_d$ and corresponding $\m T_d$ for node $d$.

    \vspace*{1ex}
    \begin{algorithmic}[1]
        \caption{norm sampling on mode $d$ with marginalization and sketch for CP tensors}
        \label{alg:sid_cp}
            \Function{SampleRow}{}
                \State Initial CP weights $\w \gets \ind{p}$
                \State $J_d = \{\}$
                \State $\m T_d = 0_{r\times n}$
                \State $\wh{\H}_d \gets \H_d$ \CommentIL{A copy of $\H_d$, will be modified at runtime}
                \For{$t = 1,\ldots,k$} \CommentIL{Compute the $t$th column}
                    \State $\b \gets (\text{Null},\ldots,\text{Null})$ of size $d-1$, storing the index of the sampled row
                    \For{$s = 1,\ldots,d-1$} \CommentIL{Compute index $b_s$}
                        \State Generate sketch operator $\S$
                        \State Sample $i\gets \sampleIndex(\w, \H_s,\ldots, \H_{d-1},\wh{\H}_d, \S)$
                        \State $b_s\gets i$
                        \State Update $\w \gets (\H_s)_{i,:} \ast \w$ \CommentIL{Effectively taking the slice $\t T \gets \t T_{b_1\ldots b_s\ldots}$}
                    \EndFor
                    \State $J_d \gets J_d \union \{\b\}$
                    \State Save the sampled row $(\m T_d)_{:, t}\gets \H_d \cdot \w$
                    \State Compute the projection direction $\v q \gets \wh{\H}_d \cdot \w$
                    \State Deflate vector $\v q$ by setting $\wh{\H}_d\gets \wh{\H}_d - \v q\v q^\top \wh{\H}_d$ 
                \EndFor
            \EndFunction
        \\
        \Return Index set $J_d$, satellite matrix $\m T_d$
    \end{algorithmic}

    \vspace*{1ex}

    \begin{algorithmic}[1]
        \Function{SampleIndex}{$\w$, $\A_1,\ldots,\A_\ell$, $\S$}
        \CommentIL{Sampling an index from the first mode}
            \State Apply sketch $\M \gets \S \mat{\cdot, 1}\cp(\w, \A_1,\ldots,\A_\ell) \in \R^{m \times n}$
            \State Compute the score $d_i \gets \norm{\M_{:, i}}^2$ for each column of $\m M$ 
            \State Sample $i \in [n]$ following $\pr(i) \,\propto\, d_i$
        \EndFunction
        \\
        \Return Index $i$
    \end{algorithmic}
\end{algorithm}

For computing the core tensor $\t C$, since the tensor is already in CP format, after computing $J_1,\ldots,J_d$ and satellite notes $\m T_1,\ldots,\m T_d$, we are left with a core in CP format of shape $(k_1,\ldots,k_d)$ given by
\begin{equation*}
    \t C = \t T \times_1\m T_1^\dagger \times_2\cdots\times_d \m T_d^\dagger
    =
    \cp(\m T_1^\dagger \m H_1,\ldots,\m T_d^\dagger\m H_d).
\end{equation*}
This can be used as the final result, or one can contract it into an unstructured core, or CP rank pruning methods can be applied to reduce the rank \cite{sherman2020estimating,malik2020fast}.

Now we analyze the complexity of \cref{alg:sid_cp}. 
We assume the CP factors $\H_i$ are of shape $n\times p$, and $d$ is an $\cO(1)$ constant.
In \textsc{SampleIndex}, matrix $\M$ is computed in $\wt{\cO}((n + m)p)$ time, and the scores are computed in $\cO(mn)$ time, where $m$ denotes the sketch dimension. 
\textsc{SampleIndex} is repeated $d-1 = \cO(1)$ time for each column, so the complexity for generating a column index $\v b$ is $\cO(mn) + \wt{\cO}((n + m)p)$.
Other operations are all $\cO(np)$, and hence the total complexity for SatID using \cref{alg:sid_cp} as the matrix ID routine is
$k\cdot (\cO(mn) + \wt{\cO}((n + m)p))$.  

Regarding the choice of $m$, as opposed to the CoreID problem where $\S$ is a subspace embedding of the flattened matrix $\A$, here it is sufficient for $\S$ to preserve the norms of $n$ columns of $\A$.
The following definition is handy (see e.g., \cite{woodruff2014sketching}).
\begin{definition}
    [$\eps$-JLT]\label{def:jlt}
    Given a set of vectors $\set{\x_1,\ldots,\x_p} \sset \R^N$ and $\eps > 0$, we say that a matrix $\S\in\R^{m\times N}$ is an $\eps$-JL transform ($\eps$-JLT) on these vectors if for all $i\in[p]$, $\|\S\x_i\|^2 \asymp (1 \pm \eps) \|\x_i\|^2$.
\end{definition}
Evidently, being an $\eps$-JLT on the columns of $\A$ is a relaxed condition compared to being an $\eps$-SE of $\A$.
Let $\S$ be a KFJLT operator of order $c$.
The optimal known bound for $m$ in order for $\S$ to be an $\eps$-JLT on the $n$ columns is $m = \wt{\cO}(\eps^{-2} (\log n)^{c})$ \cite{bamberger2022johnson}.

Finally, we combine \cref{prop:sid_err} and \cref{thm:pert_rpchol} to bound the SatID reconstruction error using \cref{alg:sid_cp} as the matrix ID routine. 
Suppose that at marginalization level $\ell$ (where $\ell = 1,\ldots,d-1$), the sketch matrix $\S_\l$ is an $\eps$-JLT on the columns of $\mat{\cdot, \l}\t T_{i_1\ldots i_{\l-1}\ldots}$. That is for all $j \in [n]$,
\begin{equation*}
    \|\S_\l(\mat{\cdot, \l}\t T_{i_1\ldots i_{\l-1}\ldots})_{:, j}\|^2 
    \asymp 
    (1\pm\eps)
    \|(\mat{\cdot, \l}\t T_{i_1\ldots i_{\l-1}\ldots})_{:, j}\|^2.
\end{equation*}
Then, since the unsketched sampling probability $\pr(X_\ell = j)$ is proportional to the right-hand side, and the sketched sampling probability $\wt{\pr}(X_\ell = j)$ is proportional to the left-hand side, we have the following relation for all $\eps \in (0, 1)$
\begin{equation*}
    \wt{\pr}(i_\l \,|\, i_1,\ldots,i_{\l-1}) \asymp \frac{1\pm \eps}{1\mp \eps}\pr(i_\l \,|\, i_1,\ldots,i_{\l-1}).
\end{equation*}
Hence, the joint sampling distribution has the relation
\begin{equation*}
    \wt{\pr}(i_1,\ldots,i_{d-1}) \asymp \left(\frac{1\pm \eps}{1\mp \eps}\right)^{d-1}\pr(i_1,\ldots,i_{d-1}).
\end{equation*}

Thus, a direct application of \cref{lem:approx_rpchol} 
gives the following bound on each (without loss of generality we only bound the $d$-th) matrix ID error in CP SatID.

\begin{restatable}{theorem}{cpsid}\label{thm:pert_cpsid}
    Fix any $\eps, \gd \in (0, 1)$ and a target rank $k$.
    Suppose that in \cref{alg:sid_cp}, in each call to $\textsc{SampleIndex}$ the matrix $\S$ is a $\delta$-JLT on the columns of the matrix to which it is applied.
    Let $(J, \m T)$ be the output of \cref{alg:sid_cp}.
    Let $\A = \mat{d, \cdot}\t T$.
    Then for a reference rank $r$,
    \begin{equation}\label{eq:pert_cpsid}
        \E\min_{\X}\|\A - \m T\X\|^2 \leq (1 + \eps) \|\A - \A^{(r)}\|^2
    \end{equation}
    provided that 
    \begin{equation*}
        k \geq \left(\frac{1+\gd}{1-\gd}\right)^{d-1}\cdot \left(\frac{r}{\eps} + r\log\left(\frac{1}{\eps\eta}\right)\right),
    \end{equation*}
    where $\eta = \|\A - \A^{(r)}\|^2 / \|\A\|^2$. 
\end{restatable}

With this result,
the final reconstruction error of CP SatID can be bounded straightforwardly according to \cref{prop:sid_err}.
We omit the statement of this result.

\subsection{SatID for Sparse Tensors} \label{sec:sid_sparse}

In this section we develop the SatID algorithm for sparse tensors.
To begin with, we note that there are at most $\cO(\nnz(\t T))$ nonzero scores $d_{\v i}^{(I)}$ we need to compute, which means the ``brute-force'' norm sampling or norm maximization methods only takes $\cO(\nnz(\t T))$ time to find the next column.
Based on this, we first give a direct SatID algorithm that achieves $\cO(k\nnz(\t T))$ complexity, where for simplicity $k_1 = \ldots = k_d = k$ is the target SatID rank. 
Later we show how the marginalization trick and sketching techniques can be used to achieve a complexity of $\cO(\nnz(\t T) + mk^2n^2)$ on an $n^d$ tensor with sketch dimension $m$. If $\t T$ has a Tucker rank at most $r$, it is sufficient to take $m = \cO(r^2)$ (see \cref{sec:nnzalg} for more detail).
This can be a remarkable improvement on higher-order tensors.

\subsubsection{The Direct Algorithm} \label{sec:direct}

As before, we assume without loss of generality that we are selecting $J_d$.
Let $\A = \mat{d, \cdot}\t T$.
From \cref{sec:margin} recall that the scores used in norm maximization and norm sampling algorithms are
\begin{equation*}
    d_{\v i}^{(I)} = \|\A_{:, \v i} - \Q_I \Q_I^\top \A_{:, \v i}\|^2,
\end{equation*}
where $\v i = (i_1,\ldots,i_{d-1})$ is the multi-index, $I$ is the indices of selected columns, and $\Q_I$ is a basis of the space spanned by selected columns.
As a new column $\v \l$ gets selected, the basis gets augmented to
\begin{equation*}
    \Q_{I\union\set{\v \l}} = (\Q_I\vert \v q).
\end{equation*}
Thus, the updated score is
\begin{equation*}
    d_{\v i}^{(I\union\set{\v \l})} 
    =
    \|\A_{:, \v i} - \Q_{I\union\set{\v \l}} \Q_{I\union\set{\v \l}}^\top \A_{:, \v i}\|^2 
    =
    d_{\v i}^{(I)} - (\v q^\top\A)_{\v i}^2.
\end{equation*}
Note $\Q_I \in \R^{n\times |I|}$ can be stored explicitly, and $\v q$ can be obtained with negligible cost, the formula above allows us to update the scores upon each column selection at cost $\cO(\nnz(\t T))$.
This direct algorithm thus runs in $\cO(k\nnz(\t T))$ time to select each index set $J_i,~i = 1,\ldots,d$. 
The algorithm is outlined below.
This algorithm can then be used as the \textsc{MatrixID} method in the skeleton \cref{alg:sid}.

\begin{algorithm}

    \textbf{Input:} sparse matrix $\A$, target rank $k$.\\
    \textbf{output:} index set $J$.

    \vspace*{1ex}
    \begin{algorithmic}[1]
        \caption{Direct column selection for sparse flattened matrix using norm maximization or norm sampling}\label{alg:sid_sparse_direct}
            \Function{DirectSparseSID}{}
                \State If $\A$ has all-zero columns, remove them and re-index the columns of $\A$ 
                \State Denote $n$ the number of columns in $\A$
                \State $d_i \gets \|\A_{:, i}\|^2$ for $i\in[n]$
                \State $\Q = \eset,~J\gets\eset$
                \For{$s = 1,\ldots,k$}
                \CommentIL{Find $s$th column}
                    \State Get column index $\l$ using scores $d$
                    \State Orthonormalize the new column $\A_{:, \l}$ to columns in $\Q$, denote the result as $\v q$
                    \State $\Q\gets (\Q |\v q)$
                    \State $d_i \gets d_i - \|\v q^\top \A_{:, i}\|^2$ for $i \in[n]$
                    \State $J\gets J\union\{\l\}$
                \EndFor
                \State Re-index the selected indices in $J$ if all-zero columns were removed from $\A$
            \EndFunction
        \\
        \Return index set $J$.
    \end{algorithmic}
\end{algorithm}

\subsubsection{An $\cO(\nnz(\t T))$ Time Algorithm with Norm Sampling} \label{sec:nnzalg}

For many large sparse data tensors $\t T$, the complexity $\cO(k\nnz(\t T))$ is often good enough for application.
But we show that under mild assumptions, for high order sparse tensors, this can be reduced to $\cO(\nnz(\t T) + mn^2 k^2)$ with norm sampling using sketching and the marginalization trick, where $m$ is the sketch dimension.
We use count sketch matrices as our sketch operator. If the matricizations $\mat{\cdot, i} \t T$ all have rank no greater than $r$, 
or in particular if $\t T$ has Tucker rank at most $r$, the sufficient sketching dimension is $m = \cO(r^2)$.

Suppose we are selecting index sets $J_d$, and we have already selected some columns $I$ into $J_d$.
For simplicity, assume $\t T$ has a shape $n^d$.
Recall that in the marginalized sampling approach \cref{sec:margin},
we need to compute (see \eqref{eq:margin_norms})
\begin{equation*}
    \pr(X_1 = i_1) 
    ~\propto~ 
    \|(\mat{\cdot, 1}\t T)_{:, i_1}\|_2^2 - \|(\mat{\cdot, 1}(\t T\times_d \m Q_I))_{:, i_1}\|^2_2.
\end{equation*}
We use a count sketch operator \cite{pham2013fast,pagh2013compressed} $\S$ that acts on $\t T$ on modes $2,3,\ldots,d-1$.
This gives
\begin{align*}
    \pr(X_1 = i_1) 
    &~\propto~ 
    \|(\mat{\cdot, 1}\t T)_{:, i_1}\|_2^2 - \|(\mat{\cdot, 1}(\t T\times_d \m Q_I))_{:, i_1}\|^2_2\\
    &=
    \|(\mat{\cdot, 1}(\t T\times_d \m Q_I^\perp))_{:, i_1}\|^2_2\\
    &\approx
    \|(\mat{\cdot, 1}(\S \times_{2,\ldots,d-1}\t T\times_d \m Q_I^\perp))_{:, i_1}\|^2_2\\
    &= \|(\mat{\cdot, 1}\S \times_{2,\ldots,d-1}\t T)_{:, i_1}\|_2^2 - \|(\mat{\cdot, 1}(\S \times_{2,\ldots,d-1}\t T\times_d \m Q_I))_{:, i_1}\|^2_2.
\end{align*}
If $\t T$ is approximately of Tucker rank $(r,\ldots,r)$, then $\mat{\cdot, 1}(\t T\times_d\Q_I^\perp)$ approximately lies in an $r$ dimensional space.
Thus, a sketch dimension of $m \sim\cO(\eps^{-2}r^2)$ is sufficient for the third line to be equal to the second line  up to a multiplicative factor of $(1\pm\eps)$.

We first compute the sketched tensor $\S\times_{2,\ldots,d-1}\t T$ in $\cO(\nnz(\t T))$ time and then compute the norms.
This takes $\cO(|I|\nnz(\S\times_{2,\ldots,d-1}\t T))$ time.
If $\nnz(\S\times_{2,\ldots,d-1}\t T) \approx \nnz(\t T)$, then there are no savings compared to the direct algorithm.
The savings happen when $\S\times_{2,\ldots,d-1}\t T$ has a size much smaller than $\nnz(\t T)$.
The size of $\S\times_{2,\ldots,d-1}\t T$ is $mn^2$, where $m$ is the sketch dimension.
Thus, when $\nnz(\t T) \gg mn^2$, the sketching reduces the complexity of computing these norms.
We outline the algorithm for selecting $J_d$ using sketching and marginalization below. 
The same method can be applied to find $J_1,\ldots,J_{d-1}$.
Full details can be found in \cref{alg:sid_sparse_margin}.
This algorithm can then be used as the $\textsc{MatrixID}$ subroutine in \cref{alg:sid}.
The policy of when to switch to direct method in \cref{alg:sid_sparse_margin} can be determined.
See the proof to \cref{prop:sparse_sid_complexity} in \cref{app:sid_nnzalg} for details.

\begin{enumerate}[itemsep=-1ex, topsep=-1ex]
    \item To select $i_s,~s = 1,\ldots,d-1$, apply a count sketch to all modes of $\t T$ except the $d$th and $s$th modes;
    \item Compute the column norms over the sketched matrix $\mat{\cdot, s}\t T$;
    \item Select $i_s$ using the norms;
    \item Reduce the tensor $\t T \gets \t T_{i_s,\ldots}$, numbering the remaining modes $s+1,\ldots,d$;
    \item Repeat the above until $s = s'$ where a sketch no longer reduces the cost when selecting $i_{s'}$;
    \item Apply the direct method to the tensor for the remaining $i_{s'},\ldots,i_{d-1}$.
\end{enumerate}

\begin{algorithm}[!h]

    \textbf{Input:} sparse tensor $\t T$, target rank $k$.\\
    \textbf{output:} index set $J_d$.

    \vspace*{1ex}
    \begin{algorithmic}[1]
        \caption{Selecting $J_d$ in sparse SatID using sketch and marginalization}\label{alg:sid_sparse_margin}
            \Function{SketchedSparseSID}{}
                \If{Count sketch does not reduce complexity}
                    \State Proceed with the direct method, \Return
                \EndIf
                \State $\Q \gets \eset$, $J_d\gets \eset$
                \State Generate count sketch $\S$, and compute $\wh{\t T} \gets \S\times_{2,\ldots,d-1}\t T$
                \State Compute the scores for $i_1$, $d_j \gets \| (\mat{\cdot, 1}\wh{\t T})_{:, j} \|^2$ 
                \CommentIL{Use direct method for $i_1$}
                \For{$k = 1,\ldots,r$}
                    \State $\b \gets (\text{Null},\ldots,\text{Null})$ of size $d-1$, storing the index of the sampled column
                    \State Sample $i_1$ using scores $d_j$, $b_1\gets i_1$
                    \State $\t T' \gets \t T_{i_1,\ldots}$, number the remaining modes $2,3,\ldots,d$
                    \CommentIL{A copy being modified below}
                    \For{$s = 2,\ldots,d-1$}
                        \If{Count sketch on $\t T'$ does not reduce complexity}
                            \State Proceed with the direct method on $\t T'$ to find $b_s,\ldots,b_{d-1}$, \bf{break}
                        \EndIf
                        \State Generate count sketch $\S$ and compute $\wh{\t T'} \gets \S\times_{s+1,\ldots,d-1}\t T'\times_d\Q$
                        \State Compute the scores for $i_s$, $d_j^{(s)} \gets \| (\mat{\cdot, s}\wh{\t T'})_{:, j} \|^2$
                        \State Sample $i_s$ from $d_j^{(s)}$, $b_s\gets i_s$
                        \State Reduce $\t T' \gets \t T'_{i_s,\ldots}$, number the remaining modes $s+1,\ldots,d$
                    \EndFor
                    \CommentNL{We have selected a column index $\v b$}
                    \State Orthonormalize the new column $\t T_{b_1,\ldots,b_{d-1},:}$ to columns of $\Q$ to get $\v q$
                    \State $\Q\gets(\Q\,\vert\,\v q)$
                    \State Update scores for $i_1$, $d_j \gets d_j - \| \v q^\top(\mat{\cdot, 1}\wh{\t T})_{:, j} \|^2$ 
                    \State $J_d\gets J_d\union\set{\v b}$
                \EndFor
            \EndFunction
        \\
        \Return index set $J_d$.
    \end{algorithmic}
\end{algorithm}

\noindent As the algorithm suggests, the actual running time will depend on the number of nonzeros of subtensors $\t T_{i_1\ldots}$ and $\t T_{i_1i_2\ldots}$ etc. 
We introduce the notations $N_0 = \nnz(\t T)$, $N_1 = \max_{i_1} \nnz(\t T_{i_1,\ldots})$ and similarly for $s < d$, $N_s = \max_{i_1,\ldots,i_s} \nnz(\t T_{i_1,\ldots,i_s,\ldots})$.
We make following assumptions in analyzing the complexity:
\begin{itemize}[itemsep = -1ex, topsep = -1ex]
    \item[(A1)] the tensor has a shape $n^d$ and the SatID target rank is $(k,\ldots,k)$;
    \item[(A2)] $k(N_1 + \ldots + N_d) = \cO(N_0)$.
\end{itemize}
Assumption (A1) is by no means essential. It is introduced to make notation simple.
Assumption (A2) holds for many reasonable tensors when $k \ll n$.
Under the assumption the nonzero entries are distributed uniformly throughout the tensor, we expect $N_s \sim N_0 / n^s$.
If uniformity does not hold (e.g., all nonzeros concentrate in one slice of $\t T$), by selecting the order of marginalization optimally, Assumption (A2) may still hold. 
Now we summarize our result in the following statement.
\begin{restatable}{proposition}{nnzAlg} \label{prop:sparse_sid_complexity}
    Under the assumptions (A1) and (A2)\footnote{If (A2) does not hold, the complexity is $\cO(N_0 + k(N_1+ \ldots + N_d) + mn^2k^2)$.}, the proposed algorithm runs with complexity $\cO(\nnz(\t T) + mn^2k^2)$, treating $d$ as an $\cO(1)$ multiplicative constant and letting $m$ denote an upper bound for the sketch dimensions.
\end{restatable}

\begin{proof}
    See \cref{app:sid_nnzalg}.
\end{proof}

\section{Numerical Experiments} \label{sec:exp}

In this section we demonstrate the performance of our CoreID and SatID algorithms on both CP tensors and sparse tensors.
A Python implementation of the algorithms used in the following experiments as well as demos using the code can be found \href{https://github.com/yifan8/TensorID}{online}.
First in \cref{sec:exp_cp}, we use synthetic CP data to test the CoreID (\cref{alg:cid_cp}) and SatID (\cref{alg:sid_cp}).
Next in \cref{sec:rna}, we apply CP CoreID and CP SatID to an RNA base pairing moment compression problem.
Then in \cref{sec:exp_sparse}, we test our algorithms on sparse real-world datasets for both CoreID (\cref{alg:cid_sparse}) and SatID (\cref{alg:sid_sparse_direct}). 
At a high level, we find the following:
\begin{enumerate}[itemsep=-1ex,topsep=-1ex]
    \item We can scale our approach to very large sparse tensors, e.g., the Enron tensor listed in \cref{tab:sparse} in \cref{sec:exp_sparse}. 
    Computing the decomposition is in fact much faster than computing (or even estimating) the reconstruction error for such problems.
    \item Sketching for CoreID and SatID effectively reduces computational costs without impacting approximation accuracy significantly.
    \item Thanks to the adaptive sequential approach and random sketching, our approach for sparse CoreID (\cref{alg:cid_sparse}) greatly improves the reconstruction error as well as timing compared to the previous study in \cite{minster2020randomized} on the same tensors. 
    In fact, it even outperforms the randomized HOSVD in accuracy as reported in \cite{minster2020randomized}.
    \item All QR-based methods (norm sampling, norm maximization, nuclear maximization) perform comparably on 4 tasks (CoreID/SatID on sparse/CP tensors), but
    norm maximization and nuclear selection yield smaller and more consistent errors than norm sampling over repeated (randomized) runs in our experiments.
\end{enumerate}

In figure legends, we label the methods respectively by NormSamp for norm sampling, NormMax for norm maximization, and Nuclear for nuclear maximization. 
See \cref{sec:compenv} for implementation details including software packages and hardware.

\subsection{Experiments on Synthetic CP Tensors} \label{sec:exp_cp}

We test CoreID (\cref{alg:cid_cp}) and SatID (\cref{alg:sid} plus \cref{alg:sid_cp}) on synthetic CP tensor data.
We generate an order $d = 4$ CP tensor 
\begin{equation*}
    \t T = \sum_{i = 1}^p \v a_i\otimes \v b_i\otimes \v c_i \otimes \v d_i
\end{equation*}
of shape $(n,n,n,n)$.
In the experiment, $n = 128$ and $p = 3300$.
The purpose of this synthetic test is to compare reconstruction errors and running times in a generic setting among CoreID and SatID algorithms using different matrix ID methods, with or without sketching. We also include the accuracy of the standard Tucker decomposition as a point of comparison.

First we give a high-level description of how the data is generated. 
In order to make sure that the tensor $\t T$ is nearly low-rank, we generate vectors $\u_i$, $i = 1,\ldots, 3000$ using a Gaussian mixture model in $\R^{4n}$. 
There are $r = 16$ components in the mixture.
Then each $\u_i$ is chopped into 4 vectors in $\R^n$, yielding the vectors $\a_i,\ldots,\v d_i \in \R^n $.
The remaining 300 rank-1 terms comprising $\t T$ represent noise. The factors $\a_i,\ldots,\v d_i$ for these terms are simply generated as standard Gaussian vectors.
The resulting tensor $\t T$ is approximately of CP rank $r = 16$.

More concretely, the mean vectors of the Gaussian mixture components are themselves instantiated as standard Gaussian vectors, and the covariance matrix of each mixture component is $\sigma^2 \I_{4n} $ where $\sigma = 0.2$.
The mixture weights $w_1,\ldots,w_r$ are independently and uniformly sampled from $[1, 10]$ and then normalized to the probability simplex.
In each factor matrix, in order to examine an adversarial context in which many columns are nearly collinear, we randomly select 25\% of the rows and set them to zero.
CoreID and SatID on tensor $\t T$ can be regarded as feature selection over the $n = 128$ features.

\subsubsection{CoreID for CP Tensors}

We test our CP CoreID algorithm (\cref{alg:cid_cp}), with and without sketching, using different matrix ID subroutines, including uniform sampling, norm sampling, norm maximization, and nuclear maximization.
The sketch dimension is set to $m = 128$, and we repeat each instance 8 times to measure the performance.
The averaged reconstruction accuracy for the sketched algorithm, as well as the best and worst among the 8 runs, are shown in \cref{fig:cpcid} (left);
the comparison between sketched and non-sketched algorithms, in terms of both accuracy and timing, are given in \cref{fig:cpcid} (middle, right). 
In the plot we also include the reconstruction error of the unconstrained Tucker decomposition.
This is used as a reference value since it is optimal among tensors of the specified Tucker rank.

\begin{figure}[!h]
    \centering
    \begin{minipage}{.33\textwidth}
        \centering
        \includegraphics[width=0.96\linewidth]{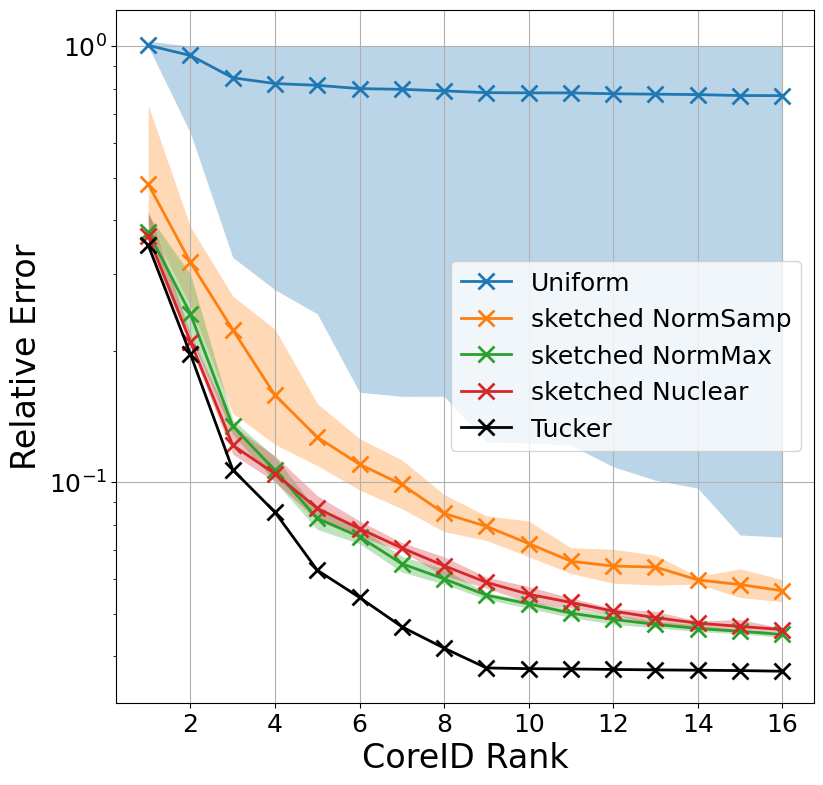}
    \end{minipage}%
    \begin{minipage}{.33\textwidth}
        \centering
        \includegraphics[width=0.96\linewidth]{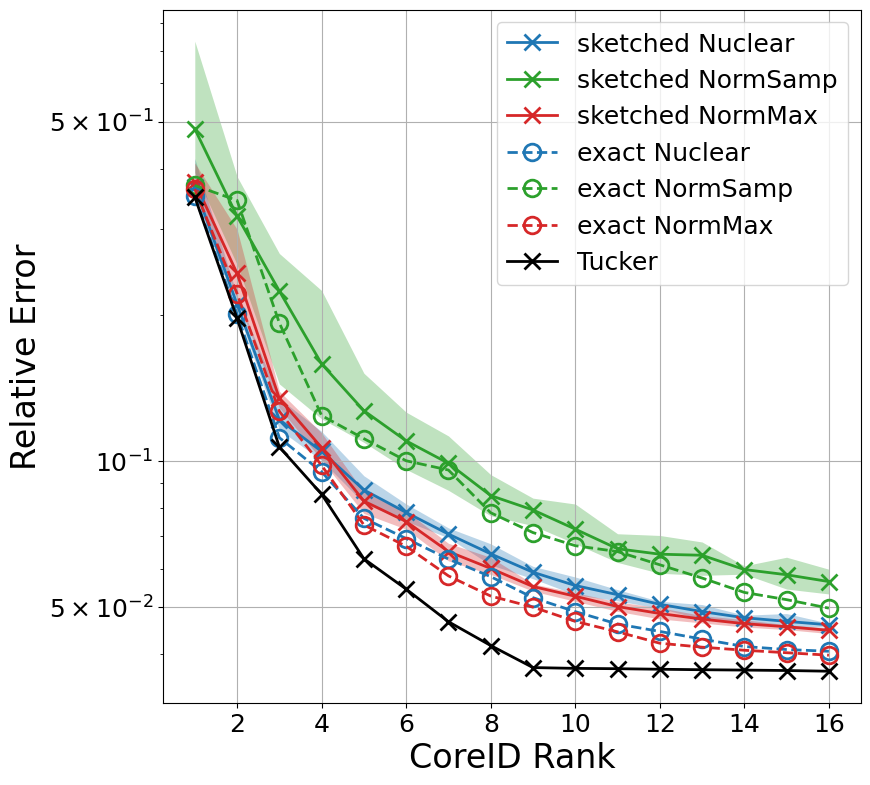}
    \end{minipage}%
    \begin{minipage}{.33\textwidth}
        \centering
        \includegraphics[width=0.96\linewidth]{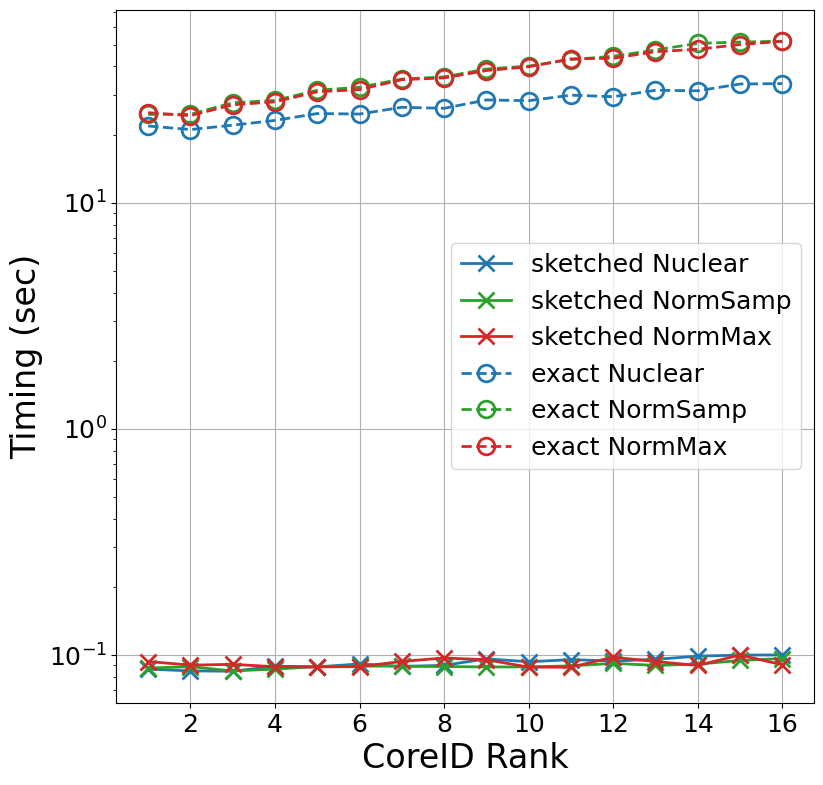}
    \end{minipage}
    \caption{
        \bf{Left:} Reconstruction error of the CP CoreID algorithm with sketching, including the average error (solid lines) and min-max errors (shaded region) for 4 matrix ID algorithms. On the bottom, the Tucker error is plotted as the optimal low rank approximation error.
        \bf{Middle:} Comparison of average reconstruction error for sketched and exact algorithms.
        \bf{Right:} Average computation times for sketched and exact algorithms.
    }
    \label{fig:cpcid}
\end{figure}

From the plot, we see that all QR-based matrix ID algorithms work significantly better than uniform sampling and achieve near-optimal reconstruction errors.
Among the three methods, norm maximization and nuclear maximization are quite similar in performance, and they tend to perform better than norm sampling.
In the middle panel (b), we see that the sketched algorithms are almost as accurate as the exact algorithms for all three matrix ID methods.
Meanwhile, from panel (c), we see that the sketched algorithms are more than 100 times faster than the exact algorithm on this problem and exhibit milder scaling in the approximation rank.

\subsubsection{SatID for CP Tensors}
Next we apply SatID to a CP tensor produced in the same way. We test both naive uniform sampling of the columns and the norm sampling algorithm using the marginalization trick (\cref{alg:sid} plus \cref{alg:sid_cp}).
Besides uniform, only norm sampling is tested; the marginalization trick is exclusive to norm sampling, which is necessary for CP SatID problems of this size.
We test both sketched and exact versions\footnote{For the exact version, $\S$ is identity in \cref{alg:sid_cp}, and the scores are computed by contracting CP tensors in the efficient way.} 
of marginalized norm sampling. 
Reconstruction errors and and timings are plotted in \cref{fig:cpsid}.
Again the Tucker errors are plotted as a reference.

\begin{figure}[!h]
    \centering
    \begin{minipage}{.45\textwidth}
        \centering
        \includegraphics[width=\linewidth]{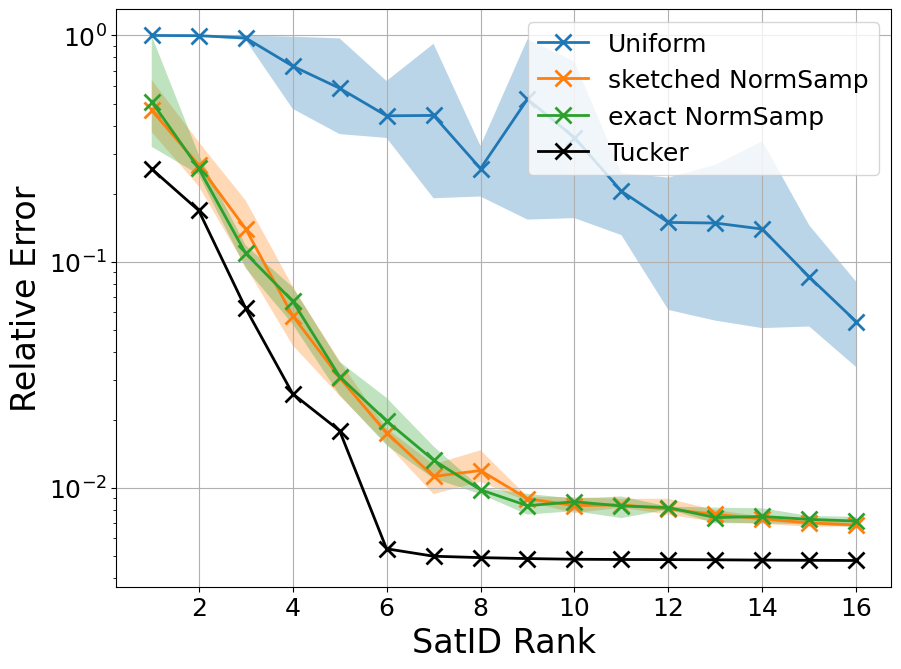}
    \end{minipage}%
    \begin{minipage}{.45\textwidth}
        \centering
        \includegraphics[width=\linewidth]{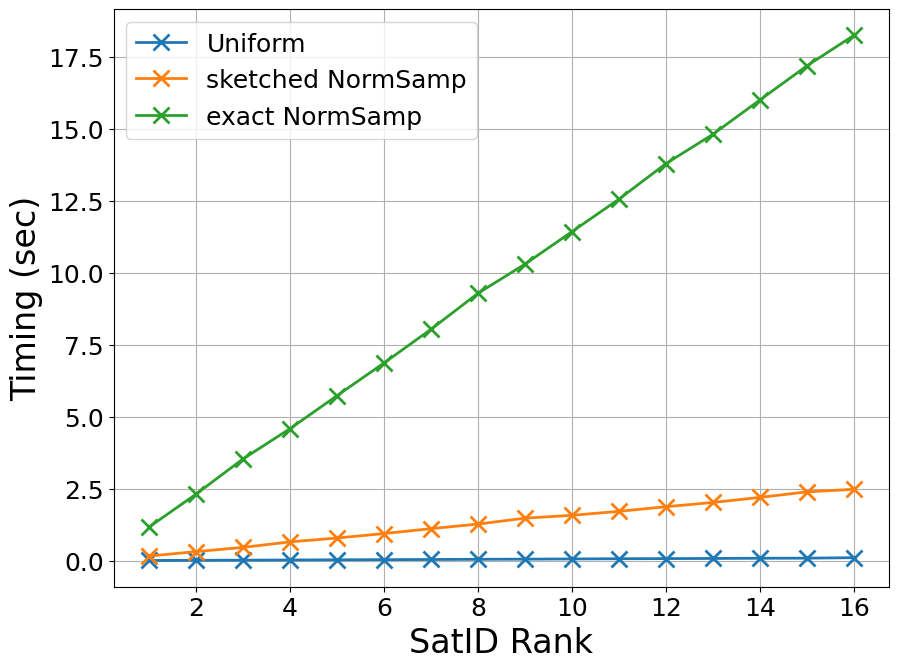}
    \end{minipage}
    \caption{
        \bf{Left:} Reconstruction accuracy of the CP SatID algorithm using marginalization and norm sampling with and without sketch, including the average error (solid lines) and min-max errors (shaded region). On the bottom, the Tucker error is plotted as the optimal low rank approximation error.
        \bf{Right:} Average computation times for sketched and exact algorithms, as well as the Tucker decomposition.
    }
    \label{fig:cpsid}
\end{figure}

Sketched and exact norm sampling are almost identical in terms of reconstruction accuracy and are both capable of finding a near-optimal decomposition.
On the right, we see that the sketched version is faster than the exact version, with milder scaling in approximation rank.
The acceleration is not as significant as in CoreID, possibly because the marginalization requires multiple sketch operators to be generated and applied. However, the asymptotic advantage is clear.

\subsection{Moment Tensors of Equilibrium RNA Base Pairing} \label{sec:rna}

Next we investigate the performance of our methods on fourth-order moment tensor data derived from statistical mechanical modeling of base pairing in a randomly generated RNA sequence of 128 nucleotides.
1,000,000 secondary structures were independently sampled proportional to their equilibrium probabilities in the thermodynamic Boltzmann ensemble \cite{fornace2020unified}.
Each structure was then translated into a 128-length bit vector based on whether each base was paired or unpaired, yielding a dataset $\V = (\v v_1|\ldots|\v v_p)$ of shape $128 \times 10^6$.
Finally, the (symmetric) fourth-order moment tensor was constructed in CP format ($n_1 = n_2 = n_3 = n_4 = 128$, $p = 10^6$) based on this dataset, modeling the second-order correlations between base pairs in the Boltzmann ensemble.
Data generation was carried out using the NUPACK 4.0.1 software package \cite{Fornace2022-nupack} with default RNA model settings (\texttt{rna06} parameters, 1 M Na$^+$, \texttt{stacking} ensemble).

We investigate both CoreID and SatID for this tensor data.
First, we apply the CoreID algorithm (\cref{alg:cid_cp}) with a sketch dimension of $m = 256$.
We test different matrix ID methods, including uniform sampling, norm sampling, norm maximization, and nuclear maximization.
For each ID method, 10 independent runs are carried out, and the reconstruction error is plotted against the approximation rank (\cref{fig:rna}, left). 
In addition to our CoreID using four different matrix ID methods, we also examine an alternative approach to compressing the data that only uses a matrix ID on the dataset to compress the 4th moment tensor.
Specifically, given the target rank $k$, this solves $(J, \X) \gets \textsc{MatrixID}(\v V, k)$ and then furnishes $J_i = J,~\X_i = \X$ for $i = 1,\ldots,4$ to compare with the CoreID result.
We used nuclear maximization algorithm as the matrix ID algorithm in this matrix-only approach, and the reconstruction error for the moment tensor is also plotted in the same figure.
In \cref{fig:rna} (right), we plot the reconstruction errors of CP SatID (\cref{alg:sid} plus \cref{alg:sid_cp}).
The experiments were repeated 5 times.
Again, the only tractable options here are to use naive uniform sampling or sketched norm sampling via marginalization.

\begin{figure}[!h]
    \centering
    \begin{minipage}{.45\textwidth}
        \centering
        \includegraphics[width=\linewidth]{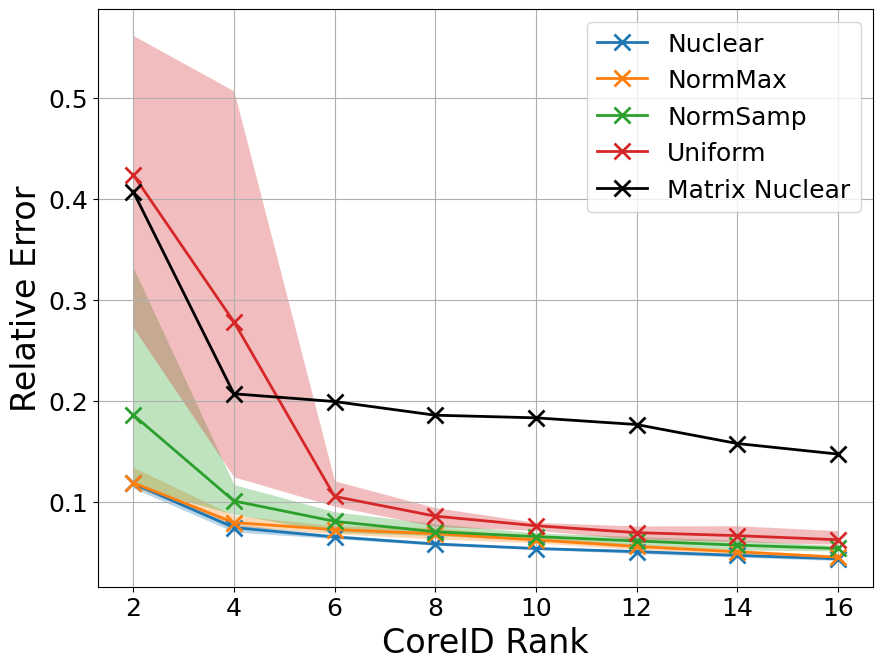}
    \end{minipage}%
    \hspace*{0.03\textwidth}
    \begin{minipage}{.45\textwidth}
        \centering
        \includegraphics[width=\linewidth]{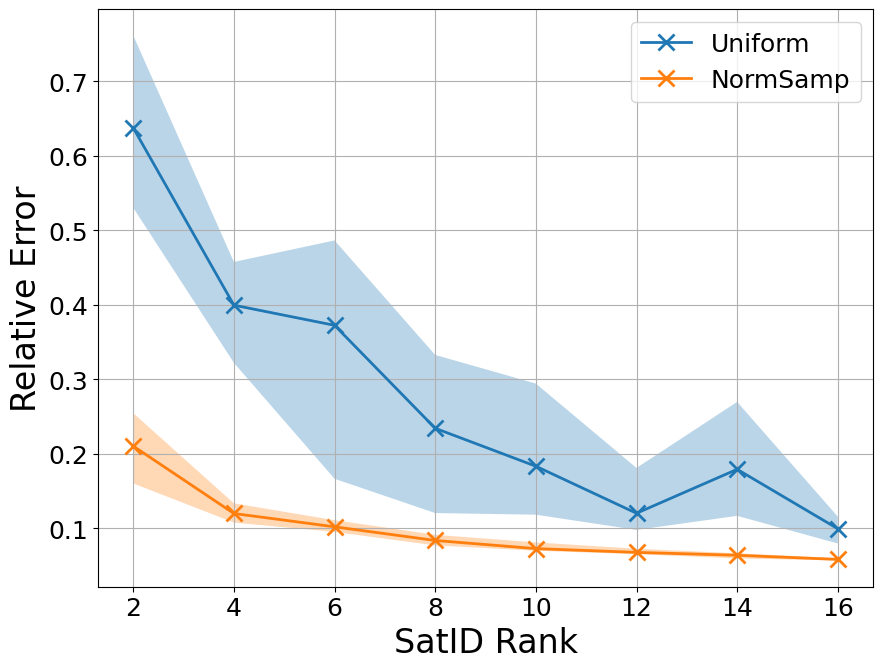}
    \end{minipage}
    \caption{
        Error plot for using CoreID and SatID to approximate the 4th moment tensor of the RNA data.
        The shaded region corresponds to maximum and minimum errors, and the solid line corresponds to the mean error over repeated runs.
        \bf{Left:} CoreID errors, comparing CoreID errors using sketch and 4 different matrix ID algorithms, as well as the reconstruction error when only matrix ID is used (black). 10 repeated runs are carried out.
        \bf{Right:} SatID errors using sketched norm sampling and marginalization. 5 repeated runs are carried out.
    }
    \label{fig:rna}
\end{figure}

From the plot, 
we clearly see that the tensor based CoreID method 
is superior to the matrix-only approach as an approach for producing a Tucker factorization.
This result is intuitive because the matrix-only approach is not directly optimizing the right quantity (the Frobenius error in the fourth moment tensor).
Moreover, the matrix-only approach can also be viewed as an independent approach, and this demonstrates the benefit of using a sequential adaptive approach in CoreID.
The matrix-only approach is even worse than the tensor CoreID method using naive uniform sampling, which also benefits from the adaptive sequential approach in which the satellite nodes $\X_i$ are determined sequentially.

Comparing the tensor CoreID methods, especially in the low-rank regime, nuclear maximization and norm maximization have the best reconstruction errors and are superior to norm sampling in terms of error and consistency over repeated runs.
All QR-based methods perform much better than naive uniform sampling.
The latter requires a larger subset of columns in order to deliver the same relative error.

\subsection{Experiments on Real World Sparse Tensor Data} \label{sec:exp_sparse}

In this section, we test our CoreID and SatID algorithms on sparse tensor data.
In order to compare to the results in \cite{minster2020randomized},
we use two real-world tensor data from the FROSTT \cite{frosttdataset} dataset: the Nell-2 tensor~\cite{carlson2010toward} and the Enron tensor~\cite{shetty2004enron}.
The original shape and number of nonzeros of these tensors, as well as other compressed versions of these tensors used in this experiment, are summarized in \cref{tab:sparse}.

\begin{table}[!h]
    \centering
    \begin{tabular}{l|r|r}
    \hline
    Tensor                 & Shape                      & nnz        \\ \hline
    Nell-2                 & (12092, 9184, 28818)       & 76,879,419 \\
    Nell-2-large           & (2419, 1837, 5764)         & 775,274    \\
    Nell-2-compressed      & (807, 613, 1922)           & 19,841     \\
    Enron                  & (6066, 5699, 244268, 1176) & 54,202,099 \\
    Enron-large            & (1517, 1425, 24427, 589)   & 227,284    \\
    Enron-large-contracted & (3034, 2850, 61068)        & 2,004,366  \\
    Enron-compressed       & (405, 380, 9771)           & 6,131      \\ \hline
    \end{tabular}
    \caption{Shapes and numbers of nonzeros (nnz) for sparse tensors used in our experiments.}
    \label{tab:sparse}
\end{table}

\subsubsection{CoreID for Sparse Tensors} \label{sec:exp_spcid}

First we compare \cref{alg:cid_sparse} to another sRRQR based CoreID algorithm SP-STHOSVD in \cite{minster2020randomized} on tensors Nell-2-compressed and Enron-compressed.
These are tensors obtained by subsampling over each dimension, and for Enron, the last mode is contracted (summed over).
We follow the exact preprocessing procedure in \cite{minster2020randomized} to compress the original tensors (cf.~\cite{minster2020randomized} for further detail).
We use a sketch dimension $m_1 = 400$ for FJLT, $m_2 = 400$ for KFJLT, and $m_3 = 2000$ for count sketch, for both the Nell-2-compressed tensor and the Enron-compressed tensor\footnote{Reconstruction errors were robust to minor changes in sketch dimensions relative to these choices.} (see \cref{fig:sparse_cid}).
Following \cite{minster2020randomized}, we process the modes in order $(3, 1, 2)$.
The algorithms are repeated 10 times on each problem.
The reconstruction errors and timing results for Enron-compressed and Nell-2-compressed are plotted respectively in \cref{fig:spcid_saib_enron} and \cref{fig:spcid_saib_nell}.
For completeness, we also included the result of R-STHOSVD algorithm in \cite{minster2020randomized}, which uses randomized SVD and sequential truncation to compute a Tucker approximation.

\begin{figure}[!h]
    \centering
    \begin{minipage}{.45\textwidth}
        \centering
        \includegraphics[width=\linewidth]{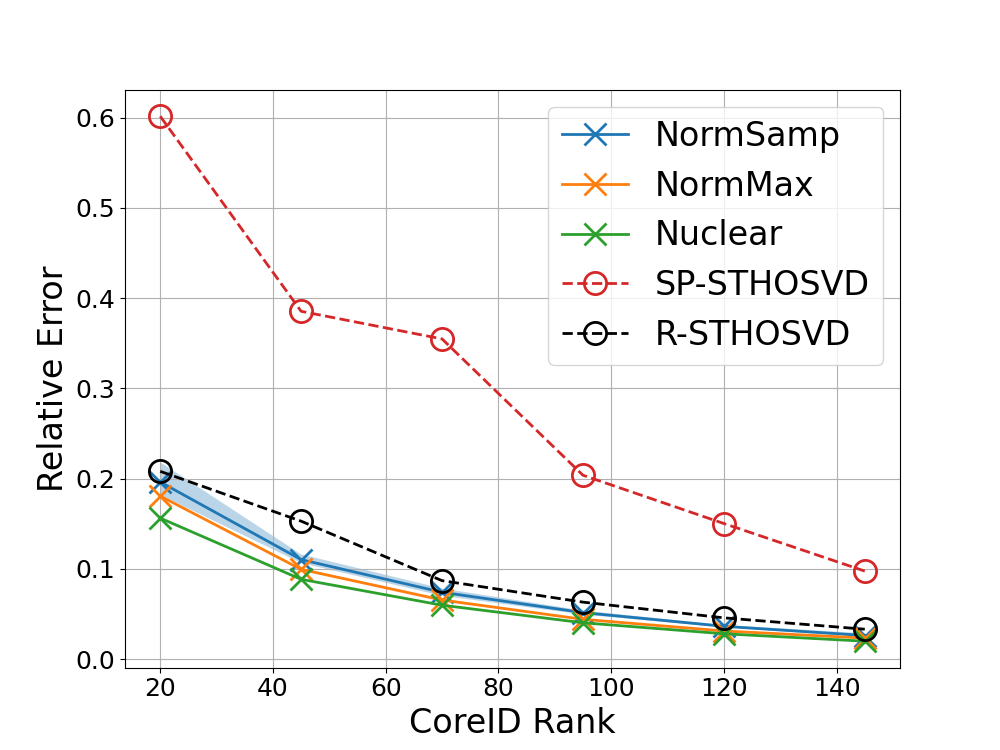}
    \end{minipage}%
    \begin{minipage}{.45\textwidth}
        \centering
        \includegraphics[width=\linewidth]{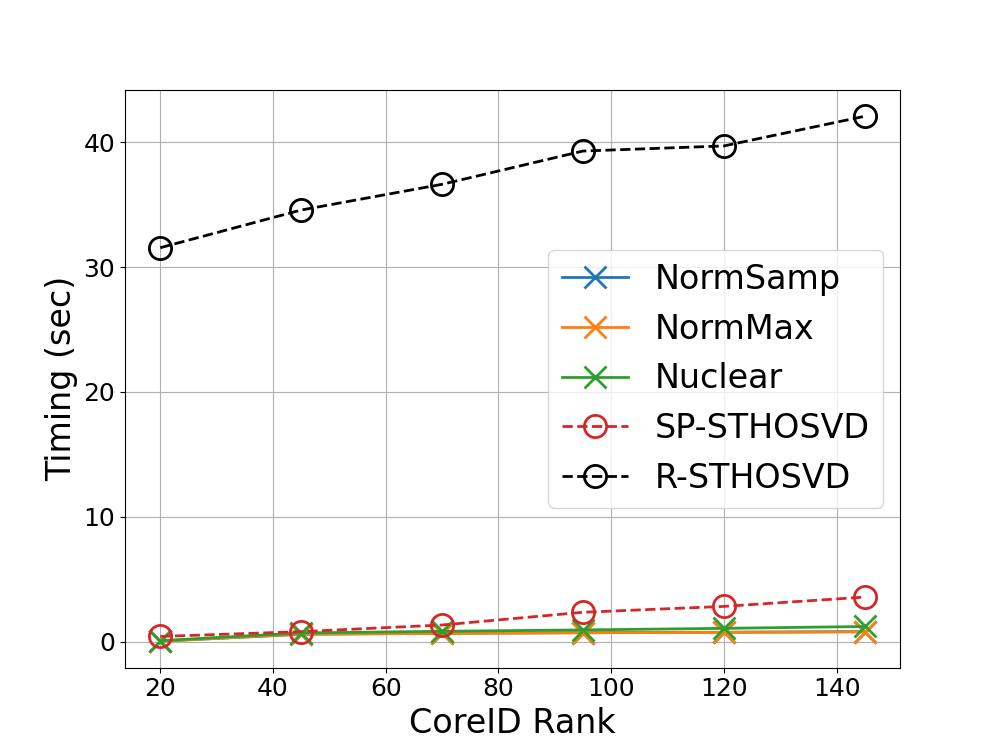}
    \end{minipage}
    \caption{
        \bf{Left:} Reconstruction accuracy of the sparse CoreID algorithm (\cref{alg:cid_sparse}) on the Enron-compressed tensor, including the average error (solid lines) and min-max errors (shaded region). The accuracy reported in \cite{minster2020randomized} of SP-STHOSVD and R-STHOSVD on the same tensor are also included.
        \bf{Right:} Computation times of \cref{alg:cid_sparse} on the Enron-compressed tensor. Only the average time is plotted. The timings of SP-STHOSVD and R-STHOSVD reported in \cite{minster2020randomized} are also included.
    }
    \label{fig:spcid_saib_enron}
\end{figure}

\begin{figure}[!h]
    \centering
    \begin{minipage}{.45\textwidth}
        \centering
        \includegraphics[width=\linewidth]{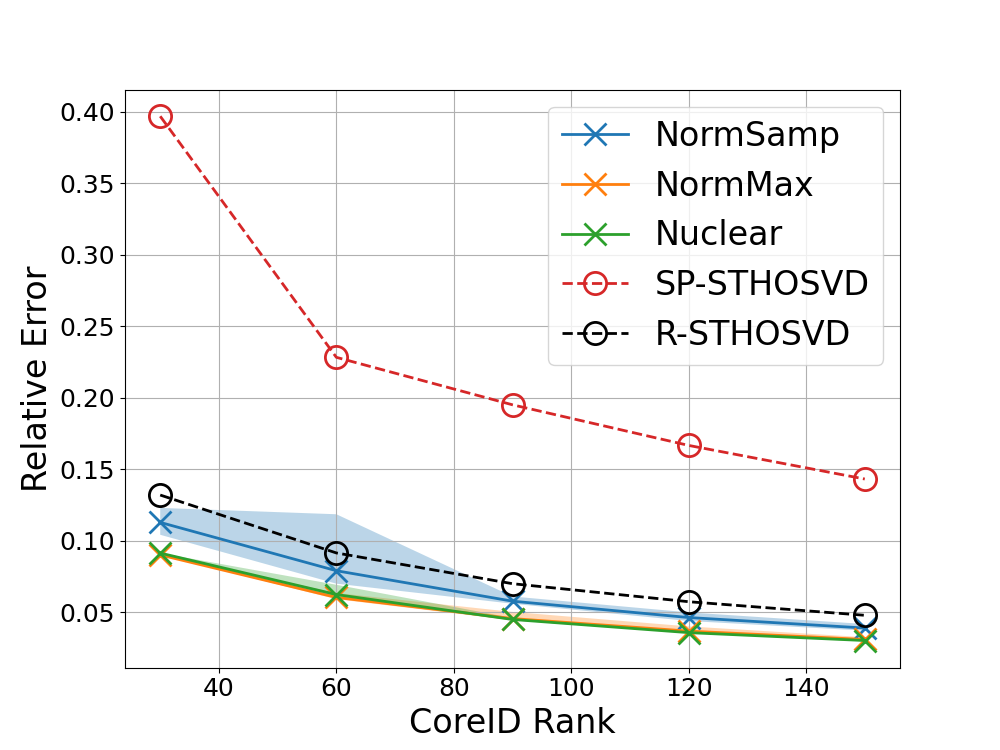}
    \end{minipage}%
    \begin{minipage}{.45\textwidth}
        \centering
        \includegraphics[width=\linewidth]{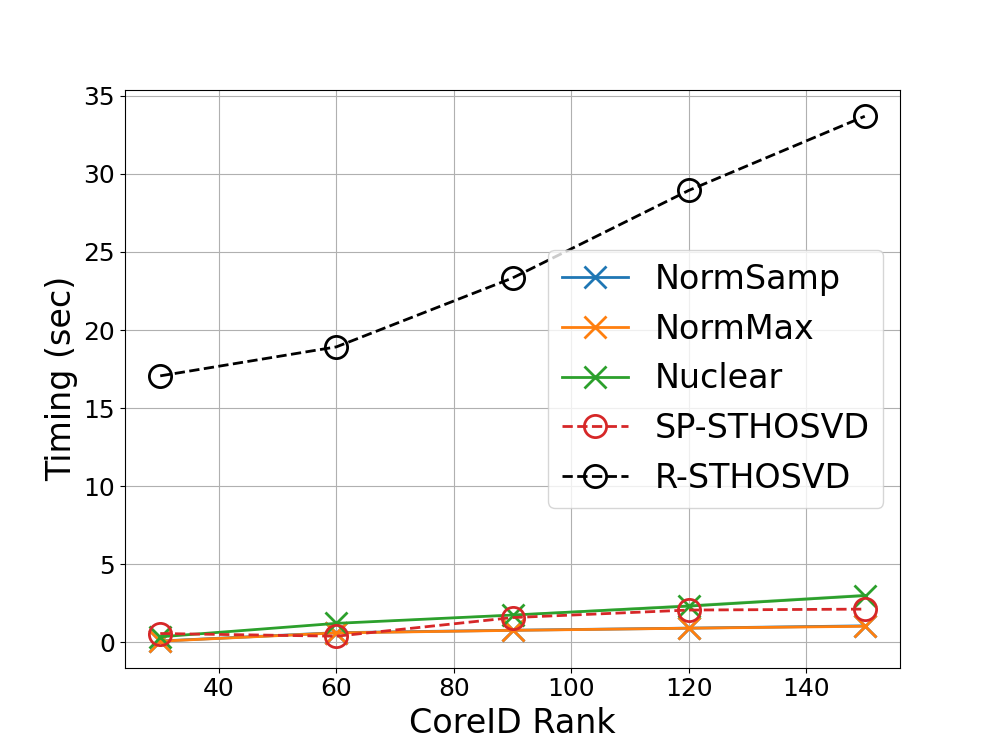}
    \end{minipage}
    \caption{
        Same as \cref{fig:spcid_saib_enron}, but the target is the Nell-2-compressed tensor.
    }
    \label{fig:spcid_saib_nell}
\end{figure}

In the timing plots, the curves for norm maximization and norm sampling are mostly identical since the complexity of the two algorithms are almost the same, making them hard to distinguish in the figures.
We only run the algorithms up to rank 145 for Enron and 150 for Nell-2.
This is because in Enron-compressed, the $\mat{1, 23}$ flattening only has 162 nonzero rows, such that the Tucker rank of the tensor is at most 162.
Similarly, for Nell-2, the $\mat{12, 3}$ flattening has only 176 nonzero columns, such that the Tucker rank of the tensor is at most 176.
From the plots, it is clear that our CoreID algorithm (\cref{alg:cid_sparse}) yields a much lower reconstruction error than SP-STHOSVD, even lower than the randomized Tucker (R-STHOSVD). 
This accords well with our analysis showing that the adaptive sequential approach provides better guarantees compared to the simple sequential approach (see \cref{sec:cid_from_matrix}), which is bounded to sRRQR and a shape-dependent error bound.
With efficient structured sketch methods rather than unstructured Gaussians (as in \cite{minster2020randomized}), \cref{alg:cid_sparse} is also faster than SP-STHOSVD, and much faster than R-STHOSVD (though measured timings may depend on many other factors e.g., the hardware).
Comparing three QR-based matrix ID methods (norm sampling, nuclear maximization, and norm maximization), we see in both experiments that nuclear maximization and norm maximization are slightly better than norm sampling in almost all runs, and they come with a better consistency in errors.
This coincides with our previous observations.

Via efficient sketching, we are able to compute CoreID of larger, less compressed variants of the Enron and Nell tensors than was feasible in \cite{minster2020randomized}.
Specifically, we test sparse CoreID on two larger tensors, Enron-large and Nell-2-large.
Enron-large is obtained by subsampling the Enron tensor every (4, 4, 10, 2) elements in each direction, resulting in a 4-way tensor, and Nell-2-large is obtained by subsampling the Nell-2 tensor every (5, 5, 5) elements in each direction.

Specifically, we use a $m_1 = 1000$ dimensional count sketch for the outer sketch matrix $\S_1$ to make the computation more efficient than using an FJLT operator as before.
Besides using sketching in the CoreID process, we also use a randomized method to efficiently estimate the reconstruction error, which is the distance between a sparse tensor and a low rank Tucker with a sparse core.
This involves an efficient algorithm to apply the KFJLT operator to a sparse tensor. 
We sketch the original sparse tensor and the reconstructed tensor to 200-dimensional vectors, which is sufficient to estimate the distance between two tensors.
In our experiments, the error was not sensitive to changes to this sketch dimension.

In the experiment, each decomposition is repeated for 5 runs, and the results are plotted in \cref{fig:spcid_large}.
Note that there are larger uncertainties in the error over repeated runs, which is expected since the reconstruction error is randomly estimated in itself.
We observe that the nuclear maximization gives better reconstruction accuracy than the alternatives. 

\begin{figure}[!h]
    \centering
    \begin{minipage}{.45\textwidth}
        \centering
        \includegraphics[width=\linewidth]{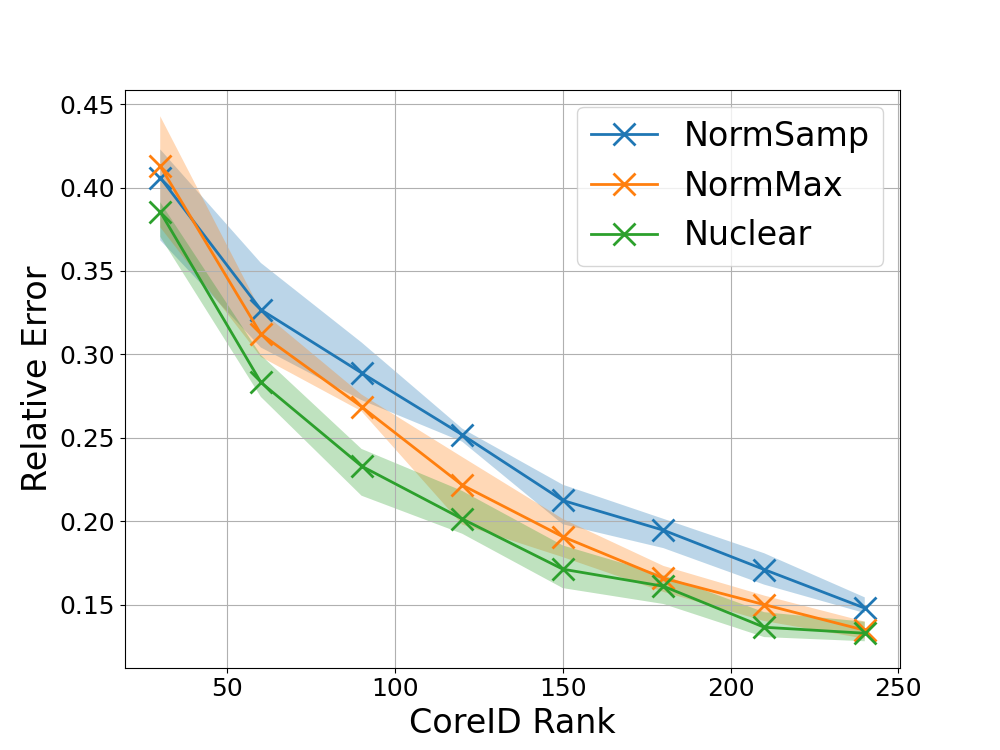}
    \end{minipage}%
    \begin{minipage}{.45\textwidth}
        \centering
        \includegraphics[width=\linewidth]{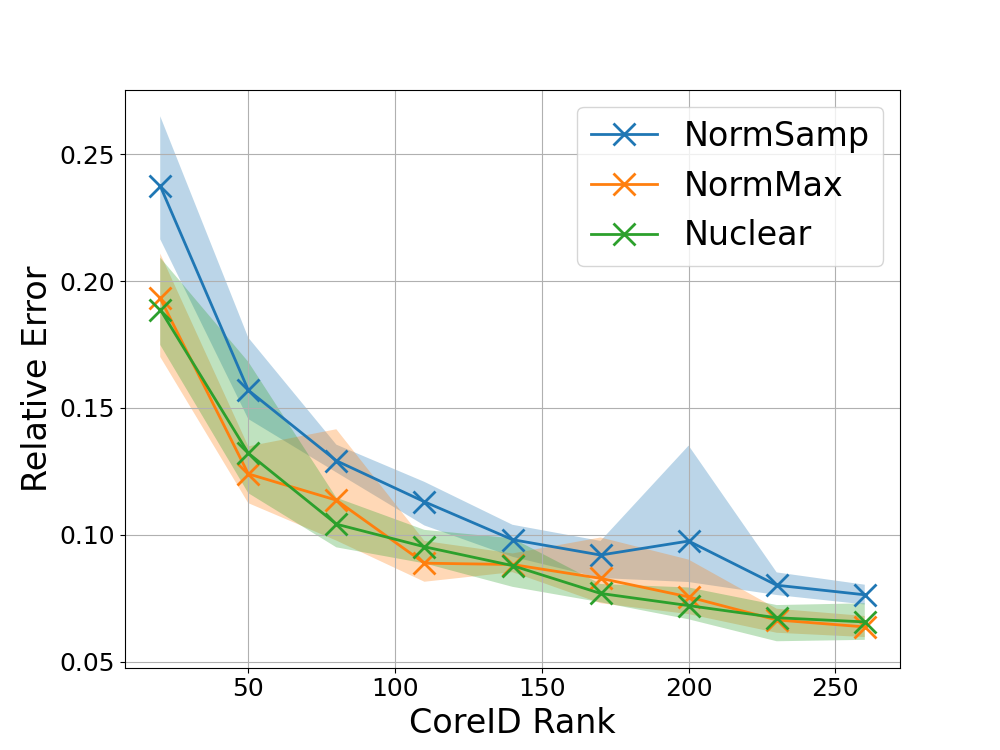}
    \end{minipage}
    \caption{
        CoreID reconstruction errors on Enron-large and Nell-2-large using \cref{alg:cid_sparse}. 
        The solid line and shaded area represent the average error and min-max error over the 5 runs. A sketch dimension of 200 is used to estimate the reconstruction error.
        \bf{Left:} Enron-large, \bf{Right:} Nell-2-large. 
    }
    \label{fig:spcid_large}
\end{figure}

Although our algorithms can efficiently handle the original uncompressed tensors Enron and Nell-2, the reason for subsampling is twofold.
First, as pointed out in \cite{minster2020randomized}, the most challenging and time-consuming part of these larger problems is not performing the CoreID itself, but estimating the reconstruction error.
On the original tensors it takes too much time and memory to compute this error, even with the randomized estimation.
Second, the original tensors seem not to be nearly low-rank, and therefore a high-rank decomposition is required, which makes the error estimation more costly.

\subsubsection{SatID for Sparse Tensors}
Finally, we test the direct method \cref{alg:sid_sparse_direct} on SatID for sparse tensors.
(Sparse SatID via \cref{alg:sid_sparse_margin} is not used here, as sparse SatID via \cref{alg:sid_sparse_direct} is feasible for this data.)
The two tensors used here are Enron-large-contracted and Nell-2-large.
Nell-2-large is the same tensor used in the previous experiment, and Enron-large-contracted is obtained by subsampling Enron tensor every (2, 2, 4) elements from the first 3 modes, and the last mode is contracted after the subsampling. 
The reason for this preprocessing is again to reduce the tensor order and make the tensor closer to low-rank; otherwise storing the dense core of size $k^4$ is infeasible memory-wise on our machine when $k\sim 200$.

The core tensor and the reconstruction error are computed using an efficient implementation of the tensor-times-matrix (ttm) algorithm for sparse tensors.
The results for the Enron and Nell-2 tensors are shown in \cref{fig:spsid}.
Note that \cref{alg:sid_sparse_direct} supports norm maximization and norm sampling, but not nuclear maximization, so we only compare these two algorithms.
Computation is repeated over 5 trials for norm sampling.

\begin{figure}[!h]
    \centering
    \begin{minipage}{.45\textwidth}
        \centering
        \includegraphics[width=\linewidth]{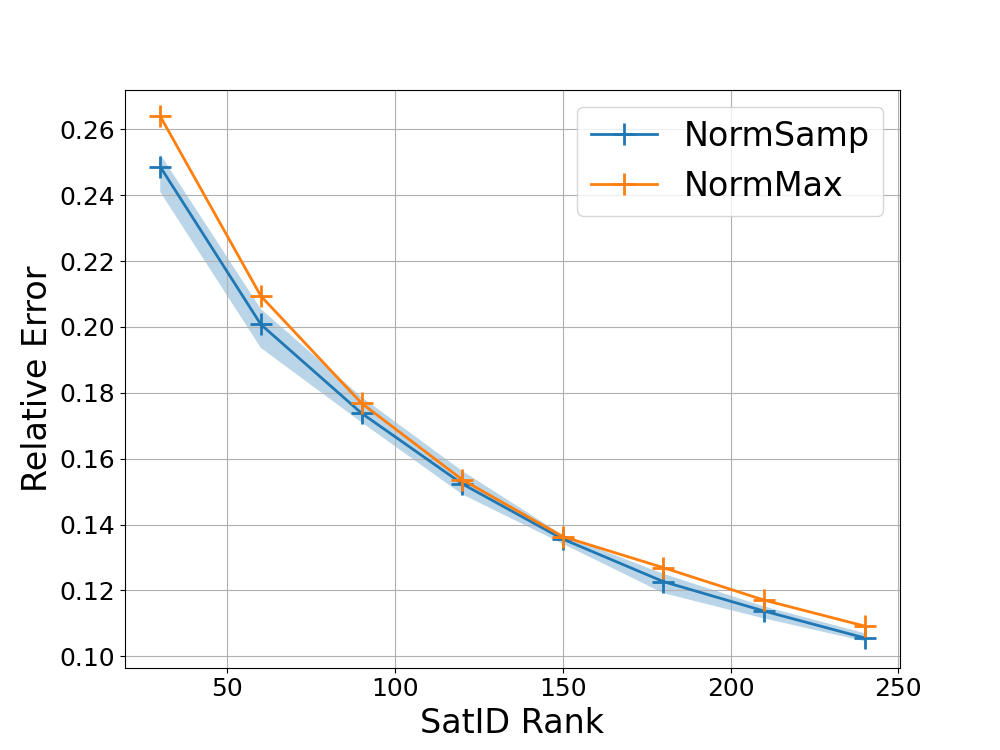}
    \end{minipage}%
    \begin{minipage}{.45\textwidth}
        \centering
        \includegraphics[width=\linewidth]{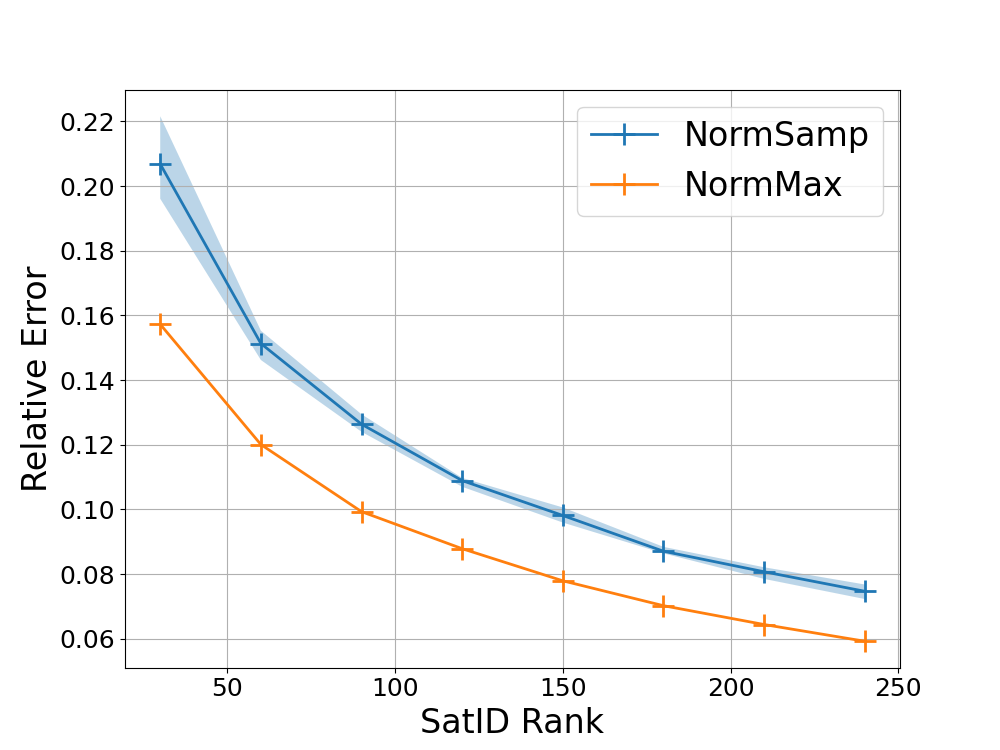}
    \end{minipage}
    \caption{
        SatID reconstruction errors on Enron-large-contracted and Nell-2-large using \cref{alg:sid_sparse_direct}. 
        \bf{Left:} Enron-large-contracted, \bf{Right:} Nell-2-large. 
        }
        \label{fig:spsid}
    \end{figure}
    
The norm sampling algorithm appears to be slightly better than norm maximization on the Enron tensor, but it is worse on the Nell tensor.
Compared to CoreID, SatID is usually more expensive memory-wise due to the dense core, and thus it is difficult to apply to higher order tensors, as in the Enron case.

\subsection{Computational Details} \label{sec:compenv}

All programs were implemented in Python\footnote{Version 3.12.2, \url{https://www.python.org/}} using standard packages, mostly based on numpy\footnote{Version 1.26.4, \url{https://numpy.org/}}, scipy\footnote{Version 1.12.0, \url{https://scipy.org/}}, and cupy\footnote{Version 13.2.0, \url{https://cupy.dev/}}.
All timed computations were carried out on a personal machine with a Core i7-10700K CPU and 32 GB RAM. 
Unconstrained Tucker decompositions were performed using solvers from Python package tensorly\footnote{Version 0.8.1, \url{https://tensorly.org/stable/index.html}}
\cite{kossaifi2019tensorly}.
A software package containing the developed algorithms and relevant demos is available \href{https://github.com/yifan8/TensorID}{online}.

\section{Conclusion} \label{sec:conclusion}

In this paper we considered two types of interpolative tensor decomposition problems, the core interpolative decomposition \eqref{eq:cid_def} (CoreID) and the satellite interpolative decomposition \eqref{eq:sid_def} (SatID).

To solve for CoreID, we introduced two methods to improve the accuracy and scalability, respectively.
For accuracy, we proposed a new adaptive sequential approach and proved that this approach leads to a dimension-independent error bound, which improves upon the $\Omega(\sqrt{N})$ error bound in \cite{minster2020randomized}, $N$ being the number of elements in the tensor.
To enhance scalability, we introduced sketching into the matrix ID subroutines in CoreID, and provided relevant theoretical guarantees.
We specialized these ideas to CP and sparse tensors and proposed new algorithms (\cref{alg:cid_cp} and \cref{alg:cid_sparse}).
When tested on synthetic and real-world data, these algorithms performed well.
Remarkable improvements in accuracy and timing was observed when compared to case studies in \cite{minster2020randomized}. 

For the SatID problem, we introduced a marginalization trick to prevent forming $n^{d-1}$ scores in the matrix ID subroutine,
and we proposed combining this marginalization trick with sketching methods to improve scalability.
We implemented these ideas and proposed \cref{alg:sid_cp} for structure-exploiting SatID on CP tensors.
This algorithm performed well in numerical experiments.
For sparse tensors, the direct algorithm \cref{alg:sid_sparse_direct} takes $\cO(k \nnz)$ time per mode, where $k$ is the target rank.
This is often acceptable, and the algorithm performed well on real-world datasets.
For higher-order sparse tensors, we used the marginalization trick and sketching to propose an $\cO(\nnz)$ time algorithm (\cref{alg:sid_sparse_margin}) under mild assumptions.

\vspace*{2ex}
\noindent
\bf{Acknowledgments:} This work was supported in part by the Applied Mathematics Competitive Portfolios
program (M.F. and M.L.) and the Center of Advanced
Mathematics for Energy Research Applications (Y.Z.) funded by the U.S. Department of Energy's Office of Advanced
Scientific Computing Research and Basic Energy Sciences
under Contract No. DE-AC02-05CH11231.
M.L. was supported in part by a Sloan Research Fellowship.
Y.Z. was also supported in part by NSF DMS 2309782, DE SC0025312, and University of Texas at Austin Graduate School Continuing Fellowship. This work was partially done during an internship of Y.Z. at Lawrence Berkeley National Laboratory.

\vspace*{2ex}
\printbibliography

\newpage
\appendix
\section{Additional proofs and details for core interpolative decomposition} \label{app:cid}

\subsection{A pessimistic counterexample for independent CoreID} \label{app:bad_indep}

In this section, we give an explicit example where using independent selection in CoreID can result in an arbitrarily bad error.
Nuclear maximization is used for this example, but similar examples can be constructed for other matrix ID methods. 
(Note that rank-1 nuclear maximization is Frobenius-optimal for the rank-1 matrix ID considered in this scenario.)

\begin{proposition}
    \label{prop:counterex_indep_cid}
    Let $M, n\gg 1$. Consider matrix
    \begin{equation*}
    \m A=\begin{pmatrix}
        1 & 1 & 1 & 1 & \cdots & 1 \\
        1 & 0 & M & M & \cdots & M \\
        1 & M & 0 & M & \cdots & M \\
        1 & M & M & 0 & \cdots & M \\
        \vdots & \vdots & \vdots & \vdots & \ddots & \vdots \\
        1 & M & M & M & \cdots & 0
    \end{pmatrix}_{n \times n}.
    \end{equation*}
    Let $J_1$ and $J_2$ be respectively the selected indices of rank-1 nuclear maximization on columns and rows of $\A$.
    Let $\X_1 = \argmin_{\X}\|\A - \A_{:, J_1}\X\|$ and $\X_2 = \argmin_{\X}\|\A - \X\A_{J_2, :}\|$ be the optimal satellite nodes. 
    Denote $\delta_1 = \|\A - \A_{:, J_1}\X_1\|$ and $\delta_2 = \|\A - \X_2\A_{J_2, :}\|$ the matrix ID errors.
    Then the final CoreID error is
    \begin{equation*}
        \|\A - \X_2\A_{J_2, J_1}\X_1\| = \Omega(M)\|\A\| = \Omega(\sqrt{n}M)(\delta_1 + \delta_2).
    \end{equation*}
\end{proposition}

\begin{proof}
    First we show that the nuclear maximization will pick the first column and row
    \begin{equation*}
        \m K=\m A^{\top} \m A=
            \begin{pmatrix}
            n & 1+(n-2) M & 1+(n-2)M & \cdots & 1+(n-2) M \\
            (n-2) M+1 & 1+(n-2) M^2 & 1+(n-3)M^2 &\cdots & 1+(n-3) M^2 \\
            (n-2) M+1 & 1+(n-3) M^2 & 1+(n-2)M &\cdots & 1+(n-3) M^2 \\
            \vdots & \vdots & \vdots & \ddots& \vdots \\
            (n-2) M+1 & 1+(n-3) M^2 & 1+(n-3) M^2 & \cdots & 1+(n-2) M^2
            \end{pmatrix}
    \end{equation*}
    The score of the first column is (for $n, M$ not trivially small)
    \begin{equation*}
        S_1=\frac{n^2+(n-1)(1+(n-2) M)^2}{n}=n+\frac{n-1}{n}(1+(n-2) M)^2>\frac{(n-1)(n-2)^2}{n} M^2
    \end{equation*}
    The scores of the other columns are identically
    \begin{equation*}
        \begin{aligned}
            S_2 & =\frac{(1+(n-2) M)^2+\left(1+(n-2) M^2\right)^2+(n-2)\left(1+(n-3) M^2\right)^2}{1+(n-2) M^2}\\
            & =\cO(n)+(n-2) M^2+(n-3)^2 M^2 \\
            & =\cO(n)+\left(n^2-5 n+7\right) M^2<\frac{(n-1)(n-2)^2}{n} M^2<S_1
        \end{aligned}
    \end{equation*}
    Thus by symmetry, nuclear maximization will take the first row and column.
    Then, we compute $\delta_2$, the column approximation error. Straightforwardly, the residual matrix is
    \begin{equation*}
    \A-\wh{\A}=
        \begin{pmatrix}
            0 & 1-\frac{1}{n}((n-2) M+1) & 1-\frac{1}{n}((n-2) M+1) & \cdots & 1-\frac{1}{n}((n-2) M+1) \\
            0 & -\frac{1}{n}((n-2) M+1) & \frac{2}{n} M-\frac{1}{n} & \cdots & \frac{2}{n} M-\frac{1}{n} \\
            0 & \frac{2}{n} M-\frac{1}{n} & -\frac{1}{n}((n-2) M+1) & \cdots & \frac{2}{n} M-\frac{1}{n} \\
            \vdots & \vdots & \vdots & \ddots & \vdots \\
            0 & \frac{2}{n} M-\frac{1}{n} & \frac{2}{n} M-\frac{1}{n} & \cdots & -\frac{1}{n}((n-2) M+1)
        \end{pmatrix}
    \end{equation*}
    So $\delta_2^2=\|\A-\wh{\A}\|_F^2<(2 n+4) M^2$. By symmetry, for $n \gg 1$ we have
    \begin{equation*}
        \delta_1+\delta_2<3 \sqrt{n} M.
    \end{equation*}
    Finally, we compute the total CoreID error $\Delta = \|\A - \X_2\A_{J_2, J_1}\X_1\|$. 
    Let $\theta=\frac{1}{n}((n-2) M+1)$. then the approximation matrix is $\wt{\A}=(1, \theta \ldots \theta)^{\top}(1, \theta, \ldots, \theta)$, so
    \begin{equation*}
    \begin{aligned}
    \Delta^2 
    &=
    \|\A-\wt{\A}\|_F^2 = 2(n-1)(1-\theta)^2+(n-1)\left(\theta^2\right)^2+(n-1)(n-2)\left(M-\theta^2\right)^2 \\
    & >(n-1)(n-2)\left(M-\theta^2\right)^2 \\
    & =(n-1)(n-2)\left[\frac{(m-2) M+1)^2}{n^2}-M\right]^2 \\
    & >\frac{1}{2} n^2 \cdot\left(\frac{1}{2} M^2\right)^2 \quad (\text {as } n, M \rightarrow+\infty)
    \end{aligned}
    \end{equation*}
    Thus $\Delta = \Omega\left(n M^2\right) = \|\A\|_F \Omega(M)=\left(\delta_1\left(J_1\right)+\delta_2\left(J_2\right)\right) \Omega(\sqrt{n} M)$
    This confirms that we can make $\Delta /\left(\delta_1+\delta_2\right)$ arbitrarily large when sending $n$ or $M$ to $\infty$.
\end{proof}

\subsection{Proof of \cref{thm:err_cid_seq}}\label{app:cid_err_bnd}

\ciderrorbnd*

\begin{proof}
    First we focus on the case where matrix ID is $\Phi$-accurate.
    Denote $\t T_1 = \t T$, and for $i \geq 2$ denote $\t T_i$ the reconstructed tensor using $J_1,\ldots,J_{i-1}$ and $\X_1,\ldots,\X_{i-1}$, from which we will select $J_i$ and compute $\X_i$.
    Denote the Tucker error at reference rank
    \begin{equation*}
        \eps_i = \|\t T_i - \t T_i^{(\v r)}\|.
    \end{equation*}
    By the triangle inequality,
    \begin{equation*}
        \|\t T_{i+1} - \t T_i^{(\v r)}\|
        \leq
        \|\t T_{i} - \t T_i^{(\v r)}\| + \|\t T_{i+1} - \t T_i\|.
    \end{equation*}
    Note that $\|\t T_{i+1} - \t T_i\|$ is the reconstruction error of the matrix ID selecting $\J_i$ and $\X_i$, which by assumption is controlled via
    \begin{equation*}
        \|\t T_{i+1} - \t T_i\| \leq \Phi_i(\|\t T_{i} - \t T_i^{(\v r)}\|).
    \end{equation*}
    Since $\t T_i^{(\v r)}$ is of Tucker rank $\v r$, we have the recursive bound
    $\eps_{i+1} \leq \eps_i + \Phi_i(\eps_i)$.
    Hence $\eps_{d-1} = \|\t T_{d-1} - \t T_{d-1}^{(\v r)}\| \leq (1 + \Phi_{d-1}) \circ \cdots\circ(I + \Phi_1)(\|\t T - \t T^{(\v r)}\|)$.
    The final assertion follows by using $\Phi$ accuracy of the matrix ID algorithm again on the last mode.

    Next consider the case of being relatively $\Phi$-accurate. 
    For a nonzero tensor $\t Y$, we denote $\wh{\t Y} = \t Y / \|\t Y\|$ the normalized tensor.
    Note that tensors of bounded Tucker rank form a (symmetric) cone.
    Thus for any nonzero tensor $\t Y$,
    \begin{equation*}
        \frac{\|\t Y - \t Y^{(\v r)}\|}{\|\t Y\|} = \|\wh{\t Y} - (\wh{\t Y})^{(\v r)}\|,
    \end{equation*}
    and $\t Y - \t Y^{(\v r)} \perp \t Y^{(\v r)}$.
    For convenience, 
    denote $\A = \mat{\cdot, i} \t T_i$ the matrix from which we select columns in step $i$.
    Denote $\B = \mat{\cdot, i} \t T_{i+1} = \A_{:, J_i}\X_i$ the reconstructed matrix after matrix ID on $\A$.
    For every choice of $J_i$, since the set of all possible matrix ID reconstructions as $\X_i$ varies also forms a cone, we have $\A - \B \perp \B$.
    Now denote the angles under the Frobenius inner product $\theta = \measuredangle(\A, \B) = \measuredangle(\t T_i, \t T_{i+1})$ and $\phi = \measuredangle(\t T_i, \t T_i^{(\v r)})$ with $\theta, \phi \in [0, \pi/2)$.
    Then
    \begin{equation*}
        \measuredangle\left(\t T_{i+1}, \t T_{i+1}^{(\v r)}\right)
        \leq
        \measuredangle\left(\t T_{i+1}, \t T_{i}^{(\v r)}\right)
        \leq 
        \measuredangle\left(\t T_{i+1}, \t T_{i}\right)
            + \measuredangle(\t T_{i}, \t T_{i}^{(\v r)})
        \leq
        \theta + \phi.
    \end{equation*}
    Thus using orthogonality,
    \begin{equation*}
        \|\wh{\t T}_{i+1} - (\wh{\t T}_{i+1})^{(\v r)}\| \leq \sin(\theta + \phi).
    \end{equation*}
    Similarly, we also have
    \begin{equation*}
        \|\wh{\t T}_i - (\wh{\t T}_i)^{(\v r)}\| = \sin\phi,~~
        \left\|\wh{\t T}_i - \frac{\t T_{i+1}}{\|\t T_i\|}\right\| = \sin\theta.
    \end{equation*}
    Since $\sin(\theta + \phi) = \sin\theta\cos\phi + \sin\phi\cos\theta \leq \sin\theta + \sin\phi$,
    We have
    \begin{align*}
        \frac{\|\t T_{i+1} - \t T_{i+1}^{(\v r)}\|}{\|\t T_{i+1}\|} 
        &=
        \|\wh{\t T}_{i+1} - (\wh{\t T}_{i+1})^{(\v r)}\|
        \leq
        \|\wh{\t T}_i - (\wh{\t T}_i)^{(\v r)}\|
            + \left\|\wh{\t T}_i - \frac{\t T_{i+1}}{\|\t T_i\|}\right\|\\
        &=
        \frac{\|\t T_{i} - \t T_{i}^{(\v r)}\|}{\|\t T_{i}\|}
            + \frac{\|\t T_{i} - \t T_{i+1}\|}{\|\t T_{i}\|}.
    \end{align*}
    For all $i\leq d-1$, denote $\eps_i = \|\t T_i - \t T_{i}^{(\v r)}\|/\|\t T_i\|$.
    Then the above gives $\eps_{i+1} \leq \eps_i + \Phi(\eps_i)$.
    The rest of the argument is identical to the $\Phi$-accurate case.

    Finally, for the expectation case, the tensors $\t T_i$ now become random variables (as are the errors $\eps_i$).
    Taking expectations on both sides of the recurrence relations and using the concavity of $\Phi_i$, we have
    \begin{equation*}
        \E\eps_{i+1} \leq \E\eps_i + \E\Phi_i(\eps_i) \leq \E\eps_i + \Phi_i(\E\eps_i).
    \end{equation*}
    Then repeating our previous argument leads to the conclusion straightforwardly. 
\end{proof}

\subsection{Proof of \cref{thm:pert_rpchol}}\label{app:cid_rpqr}

\rpqrError*

First we give a lemma that proves the first step in \eqref{eq:rpchol_approach}.

\begin{lemma}\label{lem:step1}
    Suppose that $\S$ is a $\delta$-SE of $\A$. Let $\X^*$ be the sketched least squares solution, i.e., 
    \begin{equation*}
        \X^* = \argmin_{\X}\|\S\A_{:, J}\X - \S\A\|.
    \end{equation*}
    Then $\|\A_{:, J}\X^* - \A\|^2 \leq (1-\delta)^{-1}\min_{\X}\|\A_{:, J} \X - \A\|^2$.
\end{lemma}

\begin{proof}
    Since the least squares problem is separable over the columns, it is sufficient to show that for each column $i$,
    \begin{equation*}
        \|\A_{:, J}\X^*_{:, i} - \A_{:, i}\|^2 \leq (1-\delta)^{-1}\min_{\x}\|\A_{:, J}\x - \A_{:, i}\|^2.
    \end{equation*}
    
    To make notation simple, fix $i$ and define $\a := \A_{:, i}$. Denote the exact least squares solution as $\u := \argmin_{\x}\|\A_{:, J}\x - \A_{:, i}\|$ and the projection of $\a$ as $\a_0 := \A_{:, J}\u$.
    Denote the sketched least squares solution as $\u_s := \X^*_{:, i}$ and in turn define $\a_s := \A_{:, J}\u_s$. We also define $\wh{\a} := \S \a$, $\wh{\a}_0 := \S \a_0$, and $\wh{\a}_s := \S \a_s$.
    Then in this simplified notation, what we want to show is precisely that 
    \begin{equation*}
        \|\a - \a_s\|^2 \leq (1-\delta)^{-1}\|\a - \a_0\|^2.
    \end{equation*}

    To this end, first note that by rescaling we can assume without loss of generality that $\|\a - \a_0\| = 1$.
    Moreover, define $\alpha := \|\wh{\a}_0 - \wh{\a}_s\|$ and $\beta := \|\a_0 - \a_s\|$.
    By the optimality of $\a_0$ (i.e., the fact that $\a_0$ is the least-squares projection of $\a$ to the column space of $\A_{:,J}$), we know that $\a - \a_0$ is orthogonal to the column space of $\A_{:,J}$, hence in particular to $\a_0 - \a_s$. Therefore it follows via the Pythagorean theorem that 
    \[
    \|\a - \a_s\|^2 = \|\a - \a_0\|^2 + \|\a_0 - \a_s \|^2 = 1 + \beta^2.
    \]
    Likewise, by the optimality of $\wh{\a}_s$ (i.e., the fact that $\wh{\a}_s$ is the least-squares projection of $\wh{\a}$ to the column space of $\S \A_{:,J}$), we know that $\wh{\a} - \wh{\a}_s$ is orthogonal to the column space of $\S \A_{:,J}$, hence in particular to $\wh{\a}_0 - \wh{\a}_s$. Therefore it follows via the Pythagorean theorem that 
    \[
    \| \wh{\a} - \wh{\a}_s\|^2  = \| \wh{\a} - \wh{\a}_0\|^2 - \| \wh{\a}_s - \wh{\a}_0\|^2 
    \leq 1+\delta - \alpha^2,
    \]
    where we have used the fact that $\S$ is a $\delta$-SE for $\A$ in the last line.

    Then from these facts we deduce:
    \begin{equation}
        \label{eq:lp1}
        1 + \beta^2 = \|\a - \a_s\|^2 \leq (1-\delta)^{-1}\|\wh{\a} - \wh{\a}_s\|^2 \leq (1-\delta)^{-1}(1+\delta - \alpha^2),
    \end{equation}
    where in the first inequality we have again used the SE property directly.
    Moreover, the SE property also directly implies that 
    \begin{equation}
        \label{eq:lp2}
        \alpha^2 \geq (1-\delta) \beta^2.
    \end{equation}
    By substituting \eqref{eq:lp2} into \eqref{eq:lp1} and simplifying, we deduce that $\gb^2 \leq \delta / (1-\delta)$.
    Hence
    \begin{equation*}
        \|\a - \a_s\|^2 = 1 + \beta^2 \leq (1-\delta)^{-1},
    \end{equation*}
    as we desired to show.
\end{proof}

The rest of the proof of \cref{thm:pert_rpchol} is inspired by the analysis in \cite{chen2023randomly}.
In the rest, we abbreviate symmetric positive semidefinite matrices  as PSD matrices.
For PSD matrices $\A$ and $\B$, we write $\A \preccurlyeq \B$ to indicate the Loewner partial order, meaning that $\B-\A$ is PSD.
The next lemma is handy.

\begin{lemma}\label{lem:phiD}
    For any $\beta\in [0, 1]$, define the map on PSD matrices\footnote{We use the convention $\Phi_\beta(0) = 0$.}
    \begin{equation*}
        \Phi_{\beta}(\M) = \m M - \beta\frac{\m M^2}{\tr(\m M)}.
    \end{equation*}
    Then $\Phi_{\beta}$ maps PSD matrices to PSD matrices, and for any $\theta \geq 0$, $\bar{\theta} \geq 0$, and any $\A\succcurlyeq 0$, $\B \succcurlyeq 0$,
    \begin{equation}\label{eq:phiB}
        \Phi_\beta(\theta\A + \bar{\theta}\B) \succcurlyeq \theta\Phi_\beta(\A) + \bar{\theta}\Phi_\beta(\B).
    \end{equation}
    In particular, $\Phi_\beta$ is concave and increasing with respect to the Loewner partial order.
\end{lemma}

\begin{proof}
    We rewrite $\Phi_{\beta}$ as
    \begin{equation*}
        \Phi_{\beta}(\M) 
        = 
        \M - \beta\sum_{j}\frac{\M_{jj}}{\sum_{j'} \M_{j'j'}} \M_{jj}^{-1} \M_{:, j}\M_{:, j}^\top
        \succcurlyeq
        \sum_{j}\frac{\M_{jj}}{\sum_{j'} \M_{j'j'}} \cdot (\M - \M_{jj}^{-1} \M_{:, j}\M_{:, j}^\top).
    \end{equation*}
    To prove $\Phi_{\beta}(\M)$ is also PSD, it suffices to show $\M - \M_{jj}^{-1}\M_{:, j}\M_{:, j}^\top$ is PSD.
    To this end, write $\M = \A^\top\A$. Then 
    \begin{equation*}
        \M - \M_{jj}^{-1} \M_{:, j}\M_{:, j}^{\top}  
        = 
        \A^{\top} (\m I - \hat{\v v}\hat{\v v}^{\top})\A \succcurlyeq 0,
    \end{equation*}
    where $\hat{\v v}$ is the unit vector in the direction of $\A_{:, j}$.
    Thus, $\Phi_{\beta}(\M)$ is also PSD.

    Now we turn to the proof of \eqref{eq:phiB}. 
    A direct computation shows that 
    \begin{align*}
        \Phi_{\beta}(\theta \A + \bar{\theta} \B) - \theta\Phi_{\beta}(\A) - \bar{\theta}\Phi_{\beta}(\B)
        &=
        \beta \big(\Phi_{1}(\theta \A + \bar{\theta} \B) - \theta\Phi_{1}(\A) - \bar{\theta}\Phi_{1}(\B)\big)\\
        &=
        \beta \cdot \frac{\theta\bar{\theta}}{\tr(\theta\A+ \bar{\theta}\B)}
        \m C \m C^\top \succcurlyeq 0,
    \end{align*}
    where 
    \begin{equation*}
        \m C = \sqrt{\frac{\tr(\A)}{\tr(\B)}}\B 
            -
        \sqrt{\frac{\tr(\B)}{\tr(\A)}}\A. 
    \end{equation*}
    Taking $\theta \in [0, 1]$ and $\bar{\theta} = 1 - \theta$ shows $\Phi_\beta$ is concave.
    For any $\X \preccurlyeq \Y$, taking $\theta = \bar{\theta} = 1$ we have $\Phi_{\beta}(\Y) \preccurlyeq \Phi_{\beta}(\Y - \X) + \Phi_{\beta}(\X)$, and thus $\Phi_{\beta}(\X) \preccurlyeq \Phi_\beta(\Y)$.
\end{proof}

Now we prove the key result \cref{lem:approx_rpchol}, from which the result of \cref{thm:pert_rpchol} is almost immediate.

\approxRpqr*

\begin{proof}
    Following the analysis in \cite{chen2023randomly}, we will work with the PSD matrix $\K = \A^\top\A$ instead.
    Then the desired result \eqref{eq:pert_rpchol11} is equivalent to 
    \begin{equation*}
        \E \left[ \tr(\K - \K^{[k]}) \right] \leq (1 + \eps) \, \tr(\K - \K^{(r)}),
    \end{equation*}
    where $\K^{[k]}$ is the Nystr\"{o}m approximation to $\K$ of rank $k$ using columns selected by approximate diagonal sampling, which corresponds to norm sampling for the interpolative decomposition.

    Let $\m E^{[j]} = \K - \K^{[j]}$ be the residual matrix after selecting $j$ columns under $\wt{\pr}$.
    Then the expected residual after adding a column to the selection is
    \begin{align}
        \E(\m E^{[j+1]} \,|\, \m E^{[j]}) 
        &= 
        \m E^{[j]} -
        \sum_{i = 1}^n \wt{\pr}(i)(\m E^{[j]}_{ii})^{-1} \m E^{[j]}_{:, i}(\m E^{[j]}_{:, i})^\top \nonumber\\
        &\preccurlyeq
        \m E^{[j]} -
        \sum_{i = 1}^n
            \beta\pr(i)
            (\m E^{[j]}_{ii})^{-1} \m E^{[j]}_{:, i}(\m E^{[j]}_{:, i})^\top \nonumber\\
        &=
        \m E^{[j]} - \beta\cdot\frac{(\m E^{[j]})^2}{\tr \, (\m E^{[j]})}. \nonumber\\
        &=
        \Phi_{\beta}(\m E^{[j]}).\nonumber
    \end{align}
    Here $\Phi_{\beta}$ is as defined in \cref{lem:phiD}, and in the third step we used $\pr(i) = \m E^{[j]}_{i,i} \, \big/ \tr \,  ( \m E^{[j]} )$ by the definition of diagonal sampling.

    Now since $\Phi_{\beta}$ is concave and $\m E^{[j]}$ is PSD, we can bound
    \begin{equation*}
        \E \left[ \m E^{[j+1]} \right]
        =
        \E \left[ \E(\m E^{[j+1]} | \m E^{[j]}) \right]
        \preccurlyeq
        \E \left[ \Phi_{\beta}(\m E^{[j]}) \right]
        \preccurlyeq
        \Phi_{\beta} \left( \E [ \m E^{[j]} ] \right).
    \end{equation*}
    Hence by induction it follows that 
    \begin{equation*}
        \E  \left[ \tr \, ( \m E^{[k]} ) \right] = \tr \left( \E [ \m E^{[k]} ] \right) \leq \tr \left( \Phi_{\beta}^{(k)}(\K) \right) ,
    \end{equation*}
    where $\Phi_{\beta}^{(k)}$ is the $k$-fold composition of $\Phi_{\beta}$ with itself.
    It remains to bound the right-hand side.

    The rest of the proof is almost identical to \cite{chen2023randomly}, and more detailed explanations can be found there.
    First, we observe that the map $\tr \circ \, \Phi_{\beta}^{(k)}$ is invariant under conjugation, so we may restrict our attention to positive diagonal matrices $\m \gL$ in place of more general PSD matrices $\K$ and try to find the worst case scenario maximizing $\tr  \, ( \m \Phi^{(k)} ( \boldsymbol{\Lambda} ) )$ under the constraints $\tr \, ( \m \gL^{(r)} ) = a$ and $\tr  \, ( \m \gL - \m \gL^{(r)} )  = b$ for some given $a$ and $b$.

    Indeed, we claim that the map $\m \gL \mapsto \tr \, ( \Phi_{\beta}^{(k)}(\m \gL) )$ attains its maximum subject to these constraints at 
    \[ \m \gL^* = \diag{a/r,\ldots, a/r, b/q,\ldots,b/q},\]
    where $q:=n-r$.
    To see this, observe that $\Phi_\beta$ is both increasing and concave with respect to the Loewner order (cf. \cref{lem:phiD}). Therefore composition with $\Phi_\beta$ preserves concavity, and in turn $\Phi_\beta^{(k)}$ is concave. Consequently $\tr\circ\,\Phi_\beta^{(k)}$ is a concave function on diagonal matrices, and since it is moreover invariant under permutation of the diagonal entries, it follows that $\tr\circ\,\Phi_\beta^{(k)}$ is Schur concave, and $\m\gL^*$ is indeed the maximizer.

    Therefore 
    \begin{equation*}
        \tr \left( \Phi_{\beta}^{(k)}(\m \gL) \right) \leq 
        \tr \left( \Phi_{\beta}^{(k)}(\m \gL^*) \right).
    \end{equation*}
    It is straightforward to verify that $\Phi_{\beta}^{(j)}(\m \gL^*)$ takes the form
    \[
    \Phi_{\beta}^{(j)}(\m \gL^*) = \diag{a_j/r,\ldots,a_j/r, b_j/q,\ldots, b_j/q},
    \]
    where the values $a_j$ and $b_j$ are determined via the recursion relations 
    \begin{gather*}
        a_{j+1} - a_j = -\beta\cdot \frac{a_j^2}{r(a_j + b_j)},\\
        b_{j+1} - b_j = -\beta\cdot \frac{b_j^2}{q(a_j + b_j)},
    \end{gather*}
    with initial condition $a_0 = a$ and $b_0 = b$. 

    Under the dynamics given above, we are left to upper bound $\tr \left( \Phi_{\beta}^{(k)}(\m \gL^*) \right) = a_k + b_k$.
    In \cite{chen2023randomly}, it is proved that it can be bounded by $\hat{a}_k + b$ where $\hat{a}_k$ is given by the following dynamics
    \begin{equation}\label{eq:hata}
        \hat{a}_{j+1} - \hat{a}_j = -\beta\cdot\frac{\hat{a}_j^2}{r(\hat{a}_j + b)},~~\hat{a}_0 = a.
    \end{equation}
    To derive an analytic bound on $\hat{a}_k$, in \cite{chen2023randomly} the dynamics \eqref{eq:hata} are shown to be upper bounded by a continuous-time dynamics $x(t)$:
    \begin{equation*}
        a_j \leq x(j),~~\text{where}~~
        \frac{dx(t)}{dt} = -\beta\cdot\frac{x^2}{r(x + b)},~ x(0) = a.
    \end{equation*}
    Now given $\eps > 0$, in order for it to hold that 
    \begin{equation*}
        \tr \left( \Phi_{\beta}^{(k)}(\K) \right) \leq (1 + \eps) \tr\, \left( \K - \K^{(r)} \right) = (1 + \eps) b,
    \end{equation*}
    it suffices to have $x(k) \leq \eps b$. Since $x(t)$ is decreasing in $t$, it is sufficient to take
    \begin{equation*}
        k = \int_{a}^{\eps b} -\beta^{-1}\cdot\frac{r(x + b)}{x^2} dx
        \leq
        \beta^{-1}\left(
            \frac{r}{\eps} + r\log\left(\frac{1}{\eps\eta}\right)
        \right).
    \end{equation*}
    Therefore, the claimed result is proved. 
\end{proof}

Now we can proceed with the proof of \cref{thm:pert_rpchol} combining \cref{lem:step1} and \cref{lem:approx_rpchol}.
\begin{proof}[Proof of \cref{thm:pert_rpchol}]
    Given \cref{lem:step1}, we are left to prove \eqref{eq:pert_rpchol1}.
    Recall the formula \eqref{eq:qr_score} for the exact scores for sampling column $i$ given a subset $I$ of previously selected columns:
    \begin{equation*}
        d^{(I)}_i = \min_{\x}\|\A_{:, I} \x - \A_{:, i}\|^2.
    \end{equation*}
    When $\S$ is a $\delta$-SE of $\A$, the sketched scores are
    \begin{equation*}
        \wh{d}_i^{(I)} = \min_{\x}\|\S\A_{:, I} \x - \S\A_{:, i}\|^2 \asymp (1\pm\delta) \, d^{(I)}_i.
    \end{equation*}
    Since this holds for all $I$ and $i$, it follows that the sampling probabilities $\wt{\pr}_I(i)$ and $\pr_I(i)$ (with and without sketching, respectively) satisfy
    \begin{equation*}
        \wt{\pr}_I(i) = \frac{\wh{d}_i^{(i)}}{\sum_j \wh{d}_i^{(j)}}
        \geq \frac{1-\delta}{1+\delta}\frac{{d}_i^{(i)}}{\sum_j {d}_i^{(j)}} = \frac{1-\delta}{1+\delta} \pr_I(i).
    \end{equation*}
    The conclusion follows from \cref{lem:approx_rpchol} using $\beta = (1-\delta)/(1+\delta)$.
\end{proof}

\subsection{Proof of \cref{thm:pert_nuc}}\label{app:cid_nuc}

\nucError*

\begin{proof}
    First, since each column of $\A_{:, J}\X^* - \A$ is in $\col(\A)$, by the SE property of $\S$, we have
    \begin{equation*}
        \|\S\A_{:, J}\X^* - \S\A\|^2 \asymp (1\pm\delta)\|\A_{:, J}\X^* - \A\|^2.
    \end{equation*}
    Since $\X^*$ minimizes $\|\S\A_{:, J}\X - \S\A\|^2$ over all $\X$, the first step of \eqref{eq:nuc_approach} follows, i.e., 
    \begin{equation*}
        \|\A_{:, J}\X^* - \A\|^2 \leq \frac{1}{1-\delta}\min_{\X} \|\S\A_{:, J}\X - \S\A\|^2.
    \end{equation*}
    
    Next, using the error bound from \cite{fornace2024column}, in the $\eps \rightarrow 0$ limit we have that 
    \begin{equation*}
        \min_{\X}\|\S\A_{:, J}\X - \S\A\|^2 \leq (1 + \eps)\|\S\A - (\S\A)^{(r)}\|^2
    \end{equation*}
    provided that 
    \begin{equation*}
        k \geq C \left(\frac{r}{\varepsilon }+r-1\right) \left(\log \left(\hat{\nu}^{-1}\right)+\log \left(\varepsilon^{-1}-r^{-1}+1\right)\right),
    \end{equation*}
    where $C$ is a universal constant and $\hat{\nu} = \|\S\A - (\S\A)^{(r)}\|^2 / \|(\S\A)^{(r)}\|^2$.
    
    Finally, since $\S$ is a $\delta$-SE of $\A$, it follows from the min-max principle for singular values\footnote{
    For a subspace $U$, denote $B(U)$ the unit ball in $U$. Using the min-max principle,
        \begin{equation*}
            \sigma_i(\A)^2 = \min_{\dim(U) = n-i+1} \, \max_{\x \in B(U)}\|\A\x\|^2 \, 
            \leq 
            \frac{1}{1-\delta} \ 
            \min_{\dim(U) = n-i+1} \, \max_{\x \in B(U)}\|\S\A\x\|^2 \, 
            =
            \frac{1}{1-\delta} \ \sigma_i(\S\A)^2.
        \end{equation*}
        A similar argument yields the other direction.
    }
    that $\sigma_i(\S\A)^2 \asymp (1\pm\delta)\sigma_i(\A)^2$ for all $i \leq\rank(\A)$.
    Hence,
    \begin{equation*}
        \hat{\nu}^{-1} \leq \frac{1+\delta}{1-\delta} \ \nu^{-1}
    \end{equation*}
    and 
    \begin{equation*}
        \|\S\A - (\S\A)^{(r)}\|^2 \leq (1 + \delta) \|\A - \A^{(r)}\|^2.
    \end{equation*}
    It follows that once
    \begin{equation*}
        k \geq C \left(\frac{r}{\varepsilon }+r-1\right) \left(\log \left(\frac{1+\delta}{1-\delta} \ \nu^{-1}\right)+\log \left(\varepsilon^{-1}-r^{-1}+1\right)\right),
    \end{equation*}
    we have 
    \begin{align*}
        \|\A_{:, J}\X^* - \A\|^2 
        &\leq 
        \frac{1}{1-\delta}\min_{\X}\|\S\A_{:, J}\X - \S\A\|^2\\
        &\leq
        \frac{1}{1-\delta}(1 + \eps)\|\S\A - (\S\A)^{(r)}\|^2\\
        &\leq 
        \frac{1+\delta}{1-\delta}(1+\eps)\|\A - \A^{(r)}\|^2.
    \end{align*}
    The proof is now completed.
\end{proof}

\section{Additional proofs and details for satellite interpolative decomposition} \label{app:sid}

\subsection{Proof of \cref{prop:sid_err}}\label{app:sid_errbnd}

\satIdError*

\begin{proof}
    First we consider the case of deterministic error.
    Notice that the optimal core (in Frobenius error) is given by
    \begin{equation*}
        \t C^* = \tucker(\t T, \m T_1^\dagger,\ldots,\m T_d^\dagger).
    \end{equation*}
    Thus, letting $\m P_j$ denote the orthogonal projection operator to $\col(\m T_j)$, the reconstructed tensor can alternatively be written as 
    \begin{equation*}
        \wh{\t T} = \tucker(\t T, \m P_1,\ldots,\m P_d).
    \end{equation*}
    Moreover, on every mode, the error is given by
    \begin{eqnarray*}
        \eps_j & = &   \|\t T - \tucker(\t T, \m I,\ldots,\m I, \m P_j, \m I,\ldots,\m I)\| \\
        & = &  
        \Vert \tucker(\t T, \m I,\ldots,\m I, \, \m I - \m P_j, \, \m I,\ldots,\m I)\| 
    \end{eqnarray*}
    Define $\t T_j = \tucker(\t T, \m P_1,\ldots,\m P_j, \m I,\ldots,\m I)$ and, for convenience, $\t T_0 = \t T$. 
    Then we can bound the error via telescoping sum:
    \begin{equation}
    \label{eq:telescope}
        \|\t T - \wh{\t T}\|
        \leq
        \sum_{j = 1}^{d} \|\t T_{j-1} - \t T_{j}\|.
    \end{equation}
    Now observe 
    \begin{equation*}
        \t T_{j-1} - \t T_{j} = 
        \tucker(\t T, \m P_1 ,\ldots,\m P_{j-1}, \, \m I - \m P_j, \, \m I,\ldots,\m I).
    \end{equation*}
    By considering a sequence of suitable matricizations and noting that multiplication by an orthogonal projector $\m P_i$ cannot enlarge the Frobenius norm, it follows that 
    \begin{equation*}
    \Vert  \t T_{j-1} - \t T_{j} \Vert \leq 
    \Vert \tucker(\t T, \m I ,\ldots,\m \I, \, \m I - \m P_j, \, \m I,\ldots,\m I) \Vert = \eps_j,
    \end{equation*}
    which together with \eqref{eq:telescope} completes the proof.
    The proof for the expected error follows from the same argument. 
\end{proof}

\subsection{Proof of \cref{prop:sparse_sid_complexity}} \label{app:sid_nnzalg}

Recall the assumptions we made in \cref{sec:nnzalg} to simplify the complexity analysis:

\begin{itemize}[itemsep = -1ex, topsep = -1ex]
    \item[(A1)] the tensor has a shape $n^d$ and the SatID rank is $(k,\ldots,k)$;
    \item[(A2)] $k(N_1 + \ldots + N_d) = \cO(N_0)$.
\end{itemize}

\nnzAlg*

\begin{proof}
    Following the algorithm, to choose a column $\v i= (i_1,\ldots,i_{d-1})$, we will use a sketch for selecting $i_1,\ldots$ for the first few modes, and then switch to direct method.
    Let the sketch matrices for selecting $i_b$ be $\S^{(b)}$, $b = 1,\ldots$. 
    We maintain the matrix $\Q_I\in\R^{n\times |I|}$ representing a basis of the selected columns.
    The complexity for updating this $\Q_I$ throughout the selection (Lines 19-20) is $\cO(nk)$.
    We denote by $C_s$ the total complexity needed to sample $k$ tuples of $(i_s,\ldots,i_{d-1})$ given $k$ tuples of $(i_1,\ldots,i_{s-1})$.

    First we apply $\S^{(1)}$ to modes $2,\ldots,d-1$ to $\t T$ to get 
    $\wh{\t T}^{(1)} := \S^{(1)} \times_{2,\ldots,d-1}\t T$.
    Then we select $i_1$ using the direct method from \cref{sec:direct} on $\mat{\cdot, 1}\wh{\t T}^{(1)} \in \R^{mn\times n}$ for all the $k$ columns.
    Since there are only $mn^2$ elements in $\wh{\t T}^{(1)}$, the complexity for selecting all $i_1$ for $k$ columns using the direct method (Lines 6, 19--21) is $\cO(k\nnz(\wh{\t T}^{(1)})) = \cO(kmn^2)$.
    Thus, the total complexity for finding all $i_1$, including the application of $\S^{(1)}$ (Line 5), is $\cO(N_0 + kmn^2)$.

    Next, consider $s > 1$ and we are selecting $i_s$ when $i_1,\ldots,i_{s-1}$ have been selected and we have not yet switched to the direct method (Lines 12--17). 
    Given selected column indices $I$, if we switch to the direct method, we will sample $i_s,\ldots,i_{d-1}$ all at once by sampling a single column from 
    \begin{equation*}
        \mat{d, \cdot} (\t T_{i_1,\ldots,i_{s-1},\ldots}\times_d \m Q_I^\perp).
    \end{equation*} 
    Let $\A_s =  \mat{d, \cdot}(\t T_{i_1,\ldots,i_{s-1},\ldots})$ the matrix we sample columns from.
    The sampling probability for column $\v i^{(s)} = (i_s,\ldots,i_{d-1})$ is given by
    \begin{equation*}
        \pr(\v i^{(s)} \,|\, i_1,\ldots,i_{s-1})
        ~\propto~
        \|(\A_s - \Q_I \Q_I^\top \A_s)_{:, \v i^{(s)}}\|^2_2.
        =
        \|(\A_s)_{:, \v i^{(s)}}\|^2 - \|(\Q_I^\top \A_s)_{:, \v i^{(s)}}\|^2_2
    \end{equation*}
    As there are at most $N_s$ nonzeros in $\A_s$, the scores can be computed in $\cO(|I| N_s) = \cO(kN_s)$ time.
    As we need to select $k$ columns in total, if we switch to the direct method at mode $s$ (Line 12 gives True), the total cost of selecting $k$ tuples of $(i_s,\ldots,i_{d-1})$ is then $\cO(k^2 N_s)$\footnote{Note that we cannot apply the dynamic updating trick used in the direct method to reduce this to $\cO(k N_s)$, because for each $(i_1,\ldots,i_{s-1})$, we will have different $\A_s$ and the scores need to be recomputed.}.

    On the other hand if we use marginalization and sketching at mode $s$ (Line 12 gives False), the sampling scores for $i_s$ is given by
    \begin{align*}
        \pr(i_s \,|\, i_1,\ldots,i_{s-1})
        &~\propto~
        \|(\mat{\cdot, s}\t T_{i_1,\ldots,i_{s-1},\ldots})_{:, i_s}\|_2^2 - \|(\mat{\cdot, s}\t (\t T_{i_1,\ldots,i_{s-1},\ldots}\times_d \m Q_I))_{:, i_s}\|^2_2\\
        &\approx
        \|(\mat{\cdot, s}\wh{\t T}^{(s)})_{:, i_s}\|_2^2 - \|(\mat{\cdot, s} (\wh{\t T}^{(s)}\times_d \m Q_I))_{:, i_s}\|^2_2,
    \end{align*}
    where $\wh{\t T}^{(s)} = \S^{(s)}\times_{s+1,\ldots,d-1}\t T_{i_1,\ldots,i_{s-1},\ldots}$.
    Thus, computing $\wh{\t T}^{(s)}$ requires $\cO(N_s)$ complexity.
    Since $\wh{\t T}^{(s)}$ has $mn^2$ elements, 
    these probabilities can be computed in $\cO(|I|mn^2) = \cO(kmn^2)$ time.
    So the complexity of selecting $k$ different $i_s$ in all columns (Lines 14 -- 15) is 
    $k \cdot \cO(N_s + kmn^2)$.

    Let $C_s^*$ denote the optimal complexity for sampling $(i_s,\ldots,i_{d-1})$ for all $k$ columns following the best policy of when to switch to direct method.
    Then combining the two options, switching to the direct method at $s$ or using sketch and marginalization at $s$, we have
    \begin{equation*}
        C_s^* = \min(k^2 \cO(N_s),~ C^*_{s+1} + k\cdot \cO(N_s + kmn^2)).
    \end{equation*}
    From here we can find out the optimal policy for when to switch to the direct method -- for $s = 2,\ldots,d-1$, we switch to the direct method when the first term is smaller.
    The total complexity under this policy is 
    \begin{equation*}
        C_0^* = \cO(N_0 + kmn^2) + C_1^* = \cO(N_0) + k\cdot \cO(N_1 + \ldots + N_d) + \cO(mk^2n^2) = \cO(N_0 + mk^2n^2),
    \end{equation*}
    where assumption (A2) has been used.
    This is the claimed complexity bound we desire.
\end{proof}

\end{document}